\documentclass[11pt]{amsart}   	
\usepackage{geometry}             
\usepackage[dvipsnames]{xcolor}   		             		
\usepackage{graphicx}				
\usepackage{amssymb,mathrsfs}
\usepackage{amsthm,amsmath,stmaryrd}
\usepackage{tikz}
\usepackage{tikz-cd}
\usepackage{accents,upgreek,enumerate}
\usepackage[headings]{fullpage}
\usepackage{bm}
\usepackage{mathtools}
\usepackage[all]{xy}
\usepackage{caption}
\usepackage[euler-digits]{eulervm}
\usepackage[normalem]{ulem}

\usepackage[utf8x]{inputenc}

\tikzset{
  commutative diagrams/.cd, 
  arrow style=tikz, 
  diagrams={>=stealth}
}

\usepackage{setspace}
\usetikzlibrary{fadings}

\newtheorem{innercustomthm}{Theorem}
\newenvironment{customthm}[1]
  {\renewcommand\theinnercustomthm{#1}\innercustomthm}
  {\endinnercustomthm}

  \makeatletter
\def\@tocline#1#2#3#4#5#6#7{\relax
  \ifnum #1>\c@tocdepth 
  \else
    \par \addpenalty\@secpenalty\addvspace{#2}%
    \begingroup \hyphenpenalty\@M
    \@ifempty{#4}{%
      \@tempdima\csname r@tocindent\number#1\endcsname\relax
    }{%
      \@tempdima#4\relax
    }%
    \parindent\z@ \leftskip#3\relax \advance\leftskip\@tempdima\relax
    \rightskip\@pnumwidth plus4em \parfillskip-\@pnumwidth
    #5\leavevmode\hskip-\@tempdima
      \ifcase #1
       \or\or \hskip 1em \or \hskip 2em \else \hskip 3em \fi%
      #6\nobreak\relax
    \dotfill\hbox to\@pnumwidth{\@tocpagenum{#7}}\par
    \nobreak
    \endgroup
  \fi}
\makeatother

\usetikzlibrary{calc}
\usetikzlibrary{fadings}
\usetikzlibrary{decorations.pathmorphing}
\usetikzlibrary{decorations.pathreplacing}

\newcounter{marginnote}
\DeclareMathAlphabet{\mathpzc}{OT1}{pzc}{m}{it}

\usepackage[backref=page]{hyperref}
\hypersetup{
  colorlinks   = true,          
  urlcolor     = violet,          
  linkcolor    = violet,          
  citecolor   = blue             
}

\newtheorem{theorem}{Theorem}[subsection]

\newtheorem{lemma}[theorem]{Lemma}
\newtheorem{proposition}[theorem]{Proposition}
\newtheorem{conjecture}[theorem]{Conjecture}
\newtheorem{quasi-theorem}[theorem]{Quasi-Theorem}

\theoremstyle{definition}
\newtheorem{definition}[theorem]{Definition}

\newtheorem{remark}[theorem]{Remark}

\newtheorem{construction}[theorem]{Construction}
\newtheorem{example}[theorem]{Example}
\newtheorem{blank remark}[theorem]{}

\newtheorem{not1}[theorem]{Notation}

\newcommand{\A}{{\mathbb{A}}}           
\newcommand{\CC} {{\mathbb C}}          
            
\newcommand{\NN} {{\mathbb N}}		
\newcommand{\PP}{\mathbb{P}}         
\newcommand{\QQ} {{\mathbb Q}}		
\newcommand{\RR} {{\mathbb R}}		
\newcommand{\ZZ} {{\mathbb Z}}		

\def\setminus{\smallsetminus}

\newcommand{\Hom}{\operatorname{Hom}}

\DeclareMathOperator{\spec}{Spec}

\newcommand{\cal}{\mathcal}

\def\cM{{\cal M}}

\newcommand{\Mbar}{\overline{\cM}\vphantom{\cM}}

\def\trop{\mathsf{trop}}

\makeatletter
\def\blfootnote{\xdef\@thefnmark{}\@footnotetext}
\makeatother

\title[Logarithmic DT theory]{Logarithmic Donaldson--Thomas theory}

\date{}

\author{Davesh Maulik {\it \&} Dhruv Ranganathan}

\address{Davesh Maulik \\ Department of Mathematics\\
Massachusetts Institute of Technology, Cambridge, MA, USA}
\email{\href{mailto:maulik@mit.edu}{maulik@mit.edu}}

\address{Dhruv Ranganathan \\ Department of Pure Mathematics {\it \&} Mathematical Statistics\\
University of Cambridge, Cambridge, UK}
\email{\href{mailto:dr508@cam.ac.uk}{dr508@cam.ac.uk}}

\begin{document}

\maketitle

\begin{abstract}
Let $X$ be a smooth and projective threefold with a simple normal crossings divisor $D$. We construct the Donaldson--Thomas theory of the pair $(X|D)$ enumerating ideal sheaves on $X$ relative to $D$. These moduli spaces are compactified by studying subschemes in expansions of the target geometry, and the moduli space carries a virtual fundamental class leading to numerical invariants with expected properties.   We formulate punctual evaluation, rationality and wall-crossing conjectures, in parallel with the standard theory.  Our formalism specializes to the Li--Wu theory of relative ideal sheaves when the divisor is smooth, and is parallel to recent work on logarithmic Gromov--Witten theory with expansions. 
\end{abstract}

\vspace{0.3in}
\setcounter{tocdepth}{1}
\tableofcontents
\eject

\section*{Introduction}

A basic technique in enumerative geometry is to degenerate from a smooth variety to a singular one, whose irreducible components may be easier to analyze.  
In Gromov--Witten theory, this is an essential tool in the subject, with both symplectic and algebraic incarnations \cite{IP03,LR01, Li01, Li02}.  Associated to a smooth pair $(X,D)$, \textit{relative} Gromov--Witten theory studies maps with fixed tangency along $D$, and uses these to study degenerations of a smooth variety to two smooth components meeting transversely.  The key geometric idea is that the target is dynamic, expanding along $D$ to prevent components mapping to the divisor.  

Logarithmic Gromov--Witten theory was developed~\cite{AC11,Che10,GS13} to handle more general logarithmically smooth pairs $(X,D)$, for instance, where $D$ is a normal crossings boundary divisor; these occur as components of degenerations with more complicated combinatorics.  In this approach, the target is no longer dynamic and, instead, logarithmic structures are used to maintain a form of transversality in limits.

When $X$ is a threefold, there are sheaf-theoretic approaches to enumerating curves, developed by Donaldson--Thomas and Pandharipande--Thomas~\cite{DT, RPWT, PT09}.  In the case of a smooth 
pair $(X,D)$, there is again a powerful relative theory, developed by Li--Wu~\cite{LiWu15}, which is one of the central tools of the subject.   The goal of this paper is to construct the Donaldson--Thomas theory of a pair $(X,D)$ where $D$ is a simple normal crossings divisor, referred to here as the logarithmic DT theory of $(X|D)$.  

One challenge in building such a theory is that there is no clear analogue of the unexpanded formalism of~\cite{GS13}. The moduli space of prestable curves and its universal family are themselves normal crossings pairs, so the logarithmic mapping space provides a starting point. The Hilbert scheme exhibits no such structure, and there is not, as yet, a good definition of a logarithmic structure on an ideal sheaf.  Instead, our strategy is closer to the original expanded formalism of~\cite{LiWu15}.  In an upcoming paper~\cite{MR22}, we complete the parallel to logarithmic Gromov--Witten theory, with its associated degeneration formalism \cite{ACGS15,ACGS20, R19}, and formulate the logarithmic version of the correspondence between Gromov--Witten and DT/PT theories.

\subsection{Main results} 

We work over the complex numbers. Let $X$ be a smooth and projective threefold equipped with a simple normal crossings divisor $D$. We further assume the intersection of any subset of irreducible components of $D$ is connected. 
Our main results are described below. Precise definitions and statements can be found in the main text. Fix a curve class $\beta$ in $H_2(X)$ and an integer $n$ for the holomorphic Euler characteristic. The main construction of the paper is a moduli problem $\mathsf{DT}_{\beta,n}(X|D)$ associated to these data.

\begin{customthm}{A}
The moduli problem $\mathsf{DT}_{\beta,n}(X|D)$ is representable by a proper Deligne--Mumford stack. It compactifies the moduli space of ideal sheaves on $X$ relative to $D$ with numerical invariants $\beta$ and $n$, and is equipped with universal diagram
\[
\begin{tikzcd}
\mathcal Z \arrow[rr,hook]\arrow{dr}& &\mathcal Y\arrow{dl} \arrow{r} & X\\
&  \mathsf{DT}_{\beta,n}(X|D).&
\end{tikzcd}
\]
The fibers of $\mathsf{DT}_{\beta,n}(X|D)$ parameterize relative ideal sheaves on expansions of $X$ along $D$. The space $\mathsf{DT}_{\beta,n}(X|D)$ carries a perfect obstruction theory and a virtual fundamental class with expected properties. 
\end{customthm}

The structure of our result mirrors the structure of relative Donaldson--Thomas theory.  Namely, we first construct a moduli stack parametrizing allowable target expansions; the moduli space of stable relative ideal sheaves will then consist of ideal sheaves on fibers of the universal family, satisfying a certain stability condition.
However, in our setting, both the classifying stack and the correct notion of stability are more subtle, and ideas from tropical geometry are crucial to finding the correct formulation. The approach is outlined in the final sections of this introduction.

Numerical invariants are obtained by integration against the virtual fundamental class. Primary and descendant fields are given by the Chern characters of the universal ideal sheaf, as in the standard theory~\cite{MNOP06a}. The Hilbert scheme of points on the divisor components, relative to the induced boundary divisor, can be analogously compactified. 

\begin{customthm}{B}
Each Donaldson--Thomas moduli space $\mathsf{DT}_{\beta,n}(X|D)$ is equipped with an evaluation morphism
\[
\mathsf{ev}\colon \mathsf{DT}_{\beta,n}(X|D)\to \mathsf{Ev}(D).
\]
The space $\mathsf{Ev}(D)$ is a compactification of the Hilbert scheme of points on the smooth locus of $D$. Donaldson--Thomas invariants are defined by pairing the virtual class with cohomology classes pulled back from $\mathsf{Ev}(D)$, and primary and descendant fields. 
\end{customthm}

Once the numerical invariants are defined, we formulate the basic conjectures of the sheaf theory side of the subject, in our general logarithmic setting. These conjectures include the precise evaluation of the punctual series and the rationality of the normalized generating function for primary DT invariants.  These are collected in Section~\ref{sec: invariants-conjectures}. 

Formal similarities between the map and sheaf sides give a natural generalization of the relative GW/DT correspondence to normal crossings geometries, and motivate a study of the correspondence via normal crossings degenerations~\cite{MNOP06b,MOOP}. A prototype version of this theory plays a key role in the study of the GW/Pairs correspondence of Pandharipande and Pixton~\cite{PP12}. These ideas are a central motivation for our study, and are developed in our follow-up~\cite{MR22}.

The results complete an exact parallel to~\cite[Theorem~A]{R19}, where an expanded version of logarithmic Gromov--Witten theory is constructed. 
 As in that paper, a notable feature of our moduli spaces is that they are not unique.  Rather, there is a combinatorial choice in constructing the stack of expansions, which leads to an infinite collection of moduli spaces with the desired properties, naturally organized into an inverse system.   There are natural compatibilies between virtual classes, so that the numerical invariants are independent of this choice.  This is very similar to the combinatorial choice in studying familes of degenerating abelian varieties and constructing toroidal compactifications.  While the geometry is a little different, the non-uniqueness arises in an identical fashion, via a system of polyhedral structures on a fan-like object.

However, there is a critical difference between this paper and~\cite{R19}, which creates an additional layer of complexity.  In Gromov--Witten theory with expansions, we {\it start} with the existing mapping stack of logarithmic maps to a fixed target and then apply logarithmic modifications to construct the expanded moduli space. As a result, the output is automatically proper.  In our case, there is no unexpanded space to start with, and we are forced to construct the entire system of modifications directly. Notably, properness of the moduli problem needs to be understood in a new way.  In order to do this, we first provide a tropical algorithm for finding transverse limits for families of subschemes of $X$, and then use this tropical algorithm to guess the correct moduli problem, after which we establish algebraicity, boundedness, and somewhat tautologically, properness. In the following subsections, we give a detailed outline of our approach.

\subsection{Transversality} Let $Z\subset X$ be a subscheme with ideal sheaf $\mathcal I_Z$. We are interested in subschemes that intersect $D$ in its smooth locus, with the property no one-dimensional components or embedded points of $Z$ lie in $D$. Algebraically, this is the condition that the map
\[
\mathcal I_Z\otimes \mathcal O_D\to \mathcal O_D
\]
is injective. We refer to subschemes satisfying this condition as \textit{algebraically transverse subschemes}. 

The locus of algebraically transverse subschemes is a non-proper open subscheme of the Hilbert scheme of $X$. We aim to find a compactification of this moduli problem with the prescription that the universal subscheme continues to be algebraically transverse, in an appropriate sense.


\subsection{Limits from tropicalization, after Tevelev}  In order to achieve transversality for limits of families, the scheme $X$ must be allowed to break. The degenerations here are built from tropical geometry, using an elegant argument due to Tevelev, based on Kapranov's \textit{visible contours}~\cite{Kap93,Tev07}. 

The cone over the dual complex of $D$ is denoted $\Sigma$ and can be identified with a union of faces inside $\mathbb R_{\geq 0}^k$.  Let $\Sigma^+$ be product of the cone $\RR_{\geq 0}$ with $\Sigma$. We view these as the fans associated to the toric stacks $\mathsf{A}^k$ and $\mathsf{A}^k\times\A^1$ respectively. Given an injection of cone complexes $\Delta\hookrightarrow  \Sigma^+$, the toric dictionary gives rise to a modification 
\[
\mathcal Y\to X\times \A^1\to \A^1.
\]
Geometrically, the expansion is obtained from the constant family by performing birational modifications to the strata of $D$ in the special fiber and passing to an open subscheme. When $D$ is a smooth divisor, this essentially recovers the class of expansions considered by Li and Li--Wu~\cite{Li01,LiWu15}, and we explain this in detail in Section~\ref{sec: revisiting-Li-Wu}. In the general case, they recover the class of targets in~\cite{R19}. 

We can now apply Tevelev's approach.  Given an algebraically transverse family of subschemes $\mathcal Z_\eta$ over a $\CC^\star$, the tropicalization of $\mathcal Z_\eta$ is a subset of $\Sigma^+$. This subset can be given the choice of a fan $\Delta$ contained in $\Sigma^+$. For an appropriate choice of fan structure, this produces a degeneration $\mathcal Y$ of $X$ over $\A^1$. The flat limit of $\mathcal Z_\eta$ in this degeneration is algebraically transverse to the strata of $\mathcal Y$. This was proved by Tevelev when $(X|D)$ is toric, and we make the necessary extensions in the main text. These limits have strong uniqueness and functoriality properties, making them appropriate for constructing moduli. In Figure~\ref{fig: expansion}, we caricature a subscheme in $(X|D)$ and a potential expansion of $(X|D)$ with a limiting subscheme, together with the corresponding tropical data. 
\begin{figure}

\tikzset{every picture/.style={line width=0.75pt}} 

\begin{tikzpicture}[x=0.75pt,y=0.75pt,yscale=-1,xscale=1, scale=0.8]

\draw    (80,220) -- (80,82) ;
\draw [shift={(80,80)}, rotate = 90] [color={rgb, 255:red, 0; green, 0; blue, 0 }  ][line width=0.75]    (10.93,-3.29) .. controls (6.95,-1.4) and (3.31,-0.3) .. (0,0) .. controls (3.31,0.3) and (6.95,1.4) .. (10.93,3.29)   ;
\draw    (80,220) -- (228,220) ;
\draw [shift={(230,220)}, rotate = 180] [color={rgb, 255:red, 0; green, 0; blue, 0 }  ][line width=0.75]    (10.93,-3.29) .. controls (6.95,-1.4) and (3.31,-0.3) .. (0,0) .. controls (3.31,0.3) and (6.95,1.4) .. (10.93,3.29)   ;
\draw [shift={(80,220)}, rotate = 0] [color={rgb, 255:red, 0; green, 0; blue, 0 }  ][fill={rgb, 255:red, 0; green, 0; blue, 0 }  ][line width=0.75]      (0, 0) circle [x radius= 3.35, y radius= 3.35]   ;
\draw   (340,412.95) .. controls (340,376.18) and (378.54,346.36) .. (426.09,346.36) .. controls (473.63,346.36) and (512.17,376.18) .. (512.17,412.95) .. controls (512.17,449.73) and (473.63,479.55) .. (426.09,479.55) .. controls (378.54,479.55) and (340,449.73) .. (340,412.95) -- cycle ;
\draw    (368.7,375) .. controls (370.37,373.33) and (372.03,373.33) .. (373.7,375) .. controls (375.37,376.67) and (377.03,376.67) .. (378.7,375) .. controls (380.37,373.33) and (382.03,373.33) .. (383.7,375) .. controls (385.37,376.67) and (387.03,376.67) .. (388.7,375) .. controls (390.37,373.33) and (392.03,373.33) .. (393.7,375) .. controls (395.37,376.67) and (397.03,376.67) .. (398.7,375) .. controls (400.37,373.33) and (402.03,373.33) .. (403.7,375) .. controls (405.37,376.67) and (407.03,376.67) .. (408.7,375) .. controls (410.37,373.33) and (412.03,373.33) .. (413.7,375) .. controls (415.37,376.67) and (417.03,376.67) .. (418.7,375) .. controls (420.37,373.33) and (422.03,373.33) .. (423.7,375) .. controls (425.37,376.67) and (427.03,376.67) .. (428.7,375) .. controls (430.37,373.33) and (432.03,373.33) .. (433.7,375) .. controls (435.37,376.67) and (437.03,376.67) .. (438.7,375) .. controls (440.37,373.33) and (442.03,373.33) .. (443.7,375) .. controls (445.37,376.67) and (447.03,376.67) .. (448.7,375) .. controls (450.37,373.33) and (452.03,373.33) .. (453.7,375) .. controls (455.37,376.67) and (457.03,376.67) .. (458.7,375) .. controls (460.37,373.33) and (462.03,373.33) .. (463.7,375) .. controls (465.37,376.67) and (467.03,376.67) .. (468.7,375) .. controls (470.37,373.33) and (472.03,373.33) .. (473.7,375) .. controls (475.37,376.67) and (477.03,376.67) .. (478.7,375) .. controls (480.37,373.33) and (482.03,373.33) .. (483.7,375) .. controls (485.37,376.67) and (487.03,376.67) .. (488.7,375) .. controls (490.37,373.33) and (492.03,373.33) .. (493.7,375) .. controls (495.37,376.67) and (497.03,376.67) .. (498.7,375) .. controls (500.37,373.33) and (502.03,373.33) .. (503.7,375) .. controls (505.37,376.67) and (507.03,376.67) .. (508.7,375) .. controls (510.37,373.33) and (512.03,373.33) .. (513.7,375) .. controls (515.37,376.67) and (517.03,376.67) .. (518.7,375) .. controls (520.37,373.33) and (522.03,373.33) .. (523.7,375) .. controls (525.37,376.67) and (527.03,376.67) .. (528.7,375) .. controls (530.37,373.33) and (532.03,373.33) .. (533.7,375) .. controls (535.37,376.67) and (537.03,376.67) .. (538.7,375) .. controls (540.37,373.33) and (542.03,373.33) .. (543.7,375) .. controls (545.37,376.67) and (547.03,376.67) .. (548.7,375) .. controls (550.37,373.33) and (552.03,373.33) .. (553.7,375) .. controls (555.37,376.67) and (557.03,376.67) .. (558.7,375) -- (560,375) -- (560,375) ;
\draw    (460,459.55) .. controls (458.37,457.84) and (458.41,456.18) .. (460.12,454.55) .. controls (461.83,452.92) and (461.87,451.26) .. (460.24,449.55) .. controls (458.61,447.84) and (458.65,446.18) .. (460.36,444.55) .. controls (462.07,442.92) and (462.11,441.26) .. (460.48,439.55) .. controls (458.85,437.84) and (458.89,436.18) .. (460.6,434.55) .. controls (462.31,432.92) and (462.35,431.26) .. (460.72,429.55) .. controls (459.1,427.84) and (459.14,426.18) .. (460.85,424.56) .. controls (462.56,422.93) and (462.6,421.27) .. (460.97,419.56) .. controls (459.34,417.85) and (459.38,416.19) .. (461.09,414.56) .. controls (462.8,412.93) and (462.84,411.27) .. (461.21,409.56) .. controls (459.58,407.85) and (459.62,406.19) .. (461.33,404.56) .. controls (463.04,402.93) and (463.08,401.27) .. (461.45,399.56) .. controls (459.82,397.85) and (459.86,396.19) .. (461.57,394.56) .. controls (463.28,392.94) and (463.32,391.28) .. (461.69,389.57) .. controls (460.06,387.86) and (460.1,386.2) .. (461.81,384.57) .. controls (463.52,382.94) and (463.56,381.28) .. (461.93,379.57) .. controls (460.3,377.86) and (460.34,376.2) .. (462.05,374.57) .. controls (463.76,372.94) and (463.8,371.28) .. (462.17,369.57) .. controls (460.54,367.86) and (460.58,366.2) .. (462.29,364.57) .. controls (464,362.94) and (464.04,361.28) .. (462.41,359.57) .. controls (460.79,357.86) and (460.83,356.2) .. (462.54,354.58) .. controls (464.25,352.95) and (464.29,351.29) .. (462.66,349.58) .. controls (461.03,347.87) and (461.07,346.21) .. (462.78,344.58) .. controls (464.49,342.95) and (464.53,341.29) .. (462.9,339.58) .. controls (461.27,337.87) and (461.31,336.21) .. (463.02,334.58) .. controls (464.73,332.95) and (464.77,331.29) .. (463.14,329.58) .. controls (461.51,327.87) and (461.55,326.21) .. (463.26,324.58) .. controls (464.97,322.96) and (465.01,321.3) .. (463.38,319.59) .. controls (461.75,317.88) and (461.79,316.22) .. (463.5,314.59) .. controls (465.21,312.96) and (465.25,311.3) .. (463.62,309.59) .. controls (461.99,307.88) and (462.03,306.22) .. (463.74,304.59) .. controls (465.45,302.96) and (465.49,301.3) .. (463.86,299.59) .. controls (462.23,297.88) and (462.27,296.22) .. (463.98,294.59) .. controls (465.69,292.97) and (465.73,291.31) .. (464.11,289.6) .. controls (462.48,287.89) and (462.52,286.23) .. (464.23,284.6) .. controls (465.94,282.97) and (465.98,281.31) .. (464.35,279.6) -- (464.35,279.55) -- (464.35,279.55) ;
\draw  [draw opacity=0][dash pattern={on 0.84pt off 2.51pt}] (80,220) -- (81.19,88.67) -- (219.96,87.13) -- (218.77,218.46) -- cycle ; \draw  [dash pattern={on 0.84pt off 2.51pt}] (80,220) -- (218.77,218.46)(80.18,200) -- (218.95,198.46)(80.36,180) -- (219.13,178.47)(80.54,160) -- (219.31,158.47)(80.73,140) -- (219.49,138.47)(80.91,120) -- (219.67,118.47)(81.09,100) -- (219.85,98.47) ; \draw  [dash pattern={on 0.84pt off 2.51pt}] (80,220) -- (81.19,88.67)(100,219.78) -- (101.19,88.45)(120.01,219.56) -- (121.2,88.22)(140.01,219.34) -- (141.2,88)(160.01,219.11) -- (161.2,87.78)(180.01,218.89) -- (181.21,87.56)(200.02,218.67) -- (201.21,87.34) ; \draw  [dash pattern={on 0.84pt off 2.51pt}]  ;
\draw    (80,470) -- (80,322) ;
\draw [shift={(80,320)}, rotate = 90] [color={rgb, 255:red, 0; green, 0; blue, 0 }  ][line width=0.75]    (10.93,-3.29) .. controls (6.95,-1.4) and (3.31,-0.3) .. (0,0) .. controls (3.31,0.3) and (6.95,1.4) .. (10.93,3.29)   ;
\draw    (80,470) -- (228,470) ;
\draw [shift={(230,470)}, rotate = 180] [color={rgb, 255:red, 0; green, 0; blue, 0 }  ][line width=0.75]    (10.93,-3.29) .. controls (6.95,-1.4) and (3.31,-0.3) .. (0,0) .. controls (3.31,0.3) and (6.95,1.4) .. (10.93,3.29)   ;
\draw [shift={(80,470)}, rotate = 0] [color={rgb, 255:red, 0; green, 0; blue, 0 }  ][fill={rgb, 255:red, 0; green, 0; blue, 0 }  ][line width=0.75]      (0, 0) circle [x radius= 3.35, y radius= 3.35]   ;
\draw  [draw opacity=0][dash pattern={on 0.84pt off 2.51pt}] (80,470) -- (81.35,321.52) -- (220,320) -- (218.65,468.48) -- cycle ; \draw  [dash pattern={on 0.84pt off 2.51pt}] (80,470) -- (218.65,468.48)(80.18,450) -- (218.83,448.48)(80.36,430) -- (219.02,428.48)(80.54,410) -- (219.2,408.48)(80.73,390) -- (219.38,388.48)(80.91,370) -- (219.56,368.48)(81.09,350) -- (219.74,348.48)(81.27,330.01) -- (219.92,328.48) ; \draw  [dash pattern={on 0.84pt off 2.51pt}] (80,470) -- (81.35,321.52)(100,469.78) -- (101.35,321.3)(120.01,469.56) -- (121.35,321.08)(140.01,469.34) -- (141.36,320.86)(160.01,469.12) -- (161.36,320.64)(180.01,468.9) -- (181.36,320.42)(200.02,468.68) -- (201.36,320.2) ; \draw  [dash pattern={on 0.84pt off 2.51pt}]  ;
\draw    (80.54,410) -- (160.56,409.12) ;
\draw [shift={(160.56,409.12)}, rotate = 359.37] [color={rgb, 255:red, 0; green, 0; blue, 0 }  ][fill={rgb, 255:red, 0; green, 0; blue, 0 }  ][line width=0.75]      (0, 0) circle [x radius= 3.35, y radius= 3.35]   ;
\draw [shift={(80.54,410)}, rotate = 359.37] [color={rgb, 255:red, 0; green, 0; blue, 0 }  ][fill={rgb, 255:red, 0; green, 0; blue, 0 }  ][line width=0.75]      (0, 0) circle [x radius= 3.35, y radius= 3.35]   ;
\draw    (160.01,469.11) -- (160.56,409.12) ;
\draw [shift={(160.56,409.12)}, rotate = 270.52] [color={rgb, 255:red, 0; green, 0; blue, 0 }  ][fill={rgb, 255:red, 0; green, 0; blue, 0 }  ][line width=0.75]      (0, 0) circle [x radius= 3.35, y radius= 3.35]   ;
\draw [shift={(160.01,469.11)}, rotate = 270.52] [color={rgb, 255:red, 0; green, 0; blue, 0 }  ][fill={rgb, 255:red, 0; green, 0; blue, 0 }  ][line width=0.75]      (0, 0) circle [x radius= 3.35, y radius= 3.35]   ;
\draw    (160.56,409.12) -- (160.01,332) ;
\draw [shift={(160,330)}, rotate = 89.6] [color={rgb, 255:red, 0; green, 0; blue, 0 }  ][line width=0.75]    (10.93,-3.29) .. controls (6.95,-1.4) and (3.31,-0.3) .. (0,0) .. controls (3.31,0.3) and (6.95,1.4) .. (10.93,3.29)   ;
\draw    (160.56,409.12) -- (218,409.97) ;
\draw [shift={(220,410)}, rotate = 180.85] [color={rgb, 255:red, 0; green, 0; blue, 0 }  ][line width=0.75]    (10.93,-3.29) .. controls (6.95,-1.4) and (3.31,-0.3) .. (0,0) .. controls (3.31,0.3) and (6.95,1.4) .. (10.93,3.29)   ;
\draw    (368.7,279.55) .. controls (370.37,277.88) and (372.03,277.88) .. (373.7,279.55) .. controls (375.37,281.22) and (377.03,281.22) .. (378.7,279.55) .. controls (380.37,277.88) and (382.03,277.88) .. (383.7,279.55) .. controls (385.37,281.22) and (387.03,281.22) .. (388.7,279.55) .. controls (390.37,277.88) and (392.03,277.88) .. (393.7,279.55) .. controls (395.37,281.22) and (397.03,281.22) .. (398.7,279.55) .. controls (400.37,277.88) and (402.03,277.88) .. (403.7,279.55) .. controls (405.37,281.22) and (407.03,281.22) .. (408.7,279.55) .. controls (410.37,277.88) and (412.03,277.88) .. (413.7,279.55) -- (416.52,279.55) -- (416.52,279.55) .. controls (418.19,277.88) and (419.85,277.88) .. (421.52,279.55) .. controls (423.19,281.22) and (424.85,281.22) .. (426.52,279.55) .. controls (428.19,277.88) and (429.85,277.88) .. (431.52,279.55) .. controls (433.19,281.22) and (434.85,281.22) .. (436.52,279.55) .. controls (438.19,277.88) and (439.85,277.88) .. (441.52,279.55) .. controls (443.19,281.22) and (444.85,281.22) .. (446.52,279.55) .. controls (448.19,277.88) and (449.85,277.88) .. (451.52,279.55) .. controls (453.19,281.22) and (454.85,281.22) .. (456.52,279.55) .. controls (458.19,277.88) and (459.85,277.88) .. (461.52,279.55) .. controls (463.19,281.22) and (464.85,281.22) .. (466.52,279.55) .. controls (468.19,277.88) and (469.85,277.88) .. (471.52,279.55) .. controls (473.19,281.22) and (474.85,281.22) .. (476.52,279.55) .. controls (478.19,277.88) and (479.85,277.88) .. (481.52,279.55) .. controls (483.19,281.22) and (484.85,281.22) .. (486.52,279.55) .. controls (488.19,277.88) and (489.85,277.88) .. (491.52,279.55) .. controls (493.19,281.22) and (494.85,281.22) .. (496.52,279.55) .. controls (498.19,277.88) and (499.85,277.88) .. (501.52,279.55) .. controls (503.19,281.22) and (504.85,281.22) .. (506.52,279.55) .. controls (508.19,277.88) and (509.85,277.88) .. (511.52,279.55) .. controls (513.19,281.22) and (514.85,281.22) .. (516.52,279.55) .. controls (518.19,277.88) and (519.85,277.88) .. (521.52,279.55) .. controls (523.19,281.22) and (524.85,281.22) .. (526.52,279.55) .. controls (528.19,277.88) and (529.85,277.88) .. (531.52,279.55) .. controls (533.19,281.22) and (534.85,281.22) .. (536.52,279.55) .. controls (538.19,277.88) and (539.85,277.88) .. (541.52,279.55) .. controls (543.19,281.22) and (544.85,281.22) .. (546.52,279.55) .. controls (548.19,277.88) and (549.85,277.88) .. (551.52,279.55) .. controls (553.19,281.22) and (554.85,281.22) .. (556.52,279.55) -- (560,279.55) -- (560,279.55) ;
\draw    (560,480) .. controls (558.33,478.33) and (558.33,476.67) .. (560,475) .. controls (561.67,473.33) and (561.67,471.67) .. (560,470) .. controls (558.33,468.33) and (558.33,466.67) .. (560,465) .. controls (561.67,463.33) and (561.67,461.67) .. (560,460) .. controls (558.33,458.33) and (558.33,456.67) .. (560,455) .. controls (561.67,453.33) and (561.67,451.67) .. (560,450) .. controls (558.33,448.33) and (558.33,446.67) .. (560,445) .. controls (561.67,443.33) and (561.67,441.67) .. (560,440) .. controls (558.33,438.33) and (558.33,436.67) .. (560,435) .. controls (561.67,433.33) and (561.67,431.67) .. (560,430) .. controls (558.33,428.33) and (558.33,426.67) .. (560,425) .. controls (561.67,423.33) and (561.67,421.67) .. (560,420) .. controls (558.33,418.33) and (558.33,416.67) .. (560,415) .. controls (561.67,413.33) and (561.67,411.67) .. (560,410) .. controls (558.33,408.33) and (558.33,406.67) .. (560,405) .. controls (561.67,403.33) and (561.67,401.67) .. (560,400) .. controls (558.33,398.33) and (558.33,396.67) .. (560,395) .. controls (561.67,393.33) and (561.67,391.67) .. (560,390) .. controls (558.33,388.33) and (558.33,386.67) .. (560,385) .. controls (561.67,383.33) and (561.67,381.67) .. (560,380) .. controls (558.33,378.33) and (558.33,376.67) .. (560,375) .. controls (561.67,373.33) and (561.67,371.67) .. (560,370) .. controls (558.33,368.33) and (558.33,366.67) .. (560,365) .. controls (561.67,363.33) and (561.67,361.67) .. (560,360) .. controls (558.33,358.33) and (558.33,356.67) .. (560,355) .. controls (561.67,353.33) and (561.67,351.67) .. (560,350) .. controls (558.33,348.33) and (558.33,346.67) .. (560,345) .. controls (561.67,343.33) and (561.67,341.67) .. (560,340) .. controls (558.33,338.33) and (558.33,336.67) .. (560,335) .. controls (561.67,333.33) and (561.67,331.67) .. (560,330) .. controls (558.33,328.33) and (558.33,326.67) .. (560,325) .. controls (561.67,323.33) and (561.67,321.67) .. (560,320) .. controls (558.33,318.33) and (558.33,316.67) .. (560,315) .. controls (561.67,313.33) and (561.67,311.67) .. (560,310) .. controls (558.33,308.33) and (558.33,306.67) .. (560,305) .. controls (561.67,303.33) and (561.67,301.67) .. (560,300) .. controls (558.33,298.33) and (558.33,296.67) .. (560,295) .. controls (561.67,293.33) and (561.67,291.67) .. (560,290) .. controls (558.33,288.33) and (558.33,286.67) .. (560,285) .. controls (561.67,283.33) and (561.67,281.67) .. (560,280) -- (560,279.55) -- (560,279.55) ;
\draw   (330,145) .. controls (330,103.58) and (370.29,70) .. (420,70) .. controls (469.71,70) and (510,103.58) .. (510,145) .. controls (510,186.42) and (469.71,220) .. (420,220) .. controls (370.29,220) and (330,186.42) .. (330,145) -- cycle ;
\draw    (360,116) .. controls (361.73,114.4) and (363.4,114.47) .. (365,116.2) .. controls (366.59,117.93) and (368.26,118) .. (369.99,116.4) .. controls (371.72,114.8) and (373.39,114.87) .. (374.99,116.6) .. controls (376.58,118.33) and (378.25,118.4) .. (379.98,116.8) .. controls (381.71,115.2) and (383.38,115.27) .. (384.98,117) .. controls (386.58,118.73) and (388.25,118.8) .. (389.98,117.2) .. controls (391.71,115.6) and (393.38,115.67) .. (394.97,117.4) .. controls (396.57,119.13) and (398.24,119.2) .. (399.97,117.6) .. controls (401.7,116) and (403.37,116.07) .. (404.96,117.8) .. controls (406.56,119.53) and (408.23,119.6) .. (409.96,118) .. controls (411.69,116.4) and (413.36,116.47) .. (414.96,118.2) .. controls (416.55,119.93) and (418.22,120) .. (419.95,118.4) .. controls (421.68,116.8) and (423.35,116.87) .. (424.95,118.6) .. controls (426.54,120.33) and (428.21,120.4) .. (429.94,118.8) .. controls (431.67,117.2) and (433.34,117.27) .. (434.94,119) .. controls (436.54,120.73) and (438.21,120.8) .. (439.94,119.2) .. controls (441.67,117.6) and (443.34,117.67) .. (444.93,119.4) .. controls (446.53,121.13) and (448.2,121.2) .. (449.93,119.6) .. controls (451.66,118) and (453.33,118.07) .. (454.92,119.8) .. controls (456.52,121.53) and (458.19,121.6) .. (459.92,120) -- (460,120) -- (460,120) ;
\draw    (450,205.33) .. controls (448.33,203.66) and (448.33,202) .. (450,200.33) .. controls (451.67,198.66) and (451.67,197) .. (450,195.33) .. controls (448.33,193.66) and (448.33,192) .. (450,190.33) .. controls (451.67,188.66) and (451.67,187) .. (450,185.33) .. controls (448.33,183.66) and (448.33,182) .. (450,180.33) .. controls (451.67,178.66) and (451.67,177) .. (450,175.33) .. controls (448.33,173.66) and (448.33,172) .. (450,170.33) .. controls (451.67,168.66) and (451.67,167) .. (450,165.33) .. controls (448.33,163.66) and (448.33,162) .. (450,160.33) .. controls (451.67,158.66) and (451.67,157) .. (450,155.33) .. controls (448.33,153.66) and (448.33,152) .. (450,150.33) .. controls (451.67,148.66) and (451.67,147) .. (450,145.33) .. controls (448.33,143.66) and (448.33,142) .. (450,140.33) .. controls (451.67,138.66) and (451.67,137) .. (450,135.33) .. controls (448.33,133.66) and (448.33,132) .. (450,130.33) .. controls (451.67,128.66) and (451.67,127) .. (450,125.33) .. controls (448.33,123.66) and (448.33,122) .. (450,120.33) .. controls (451.67,118.66) and (451.67,117) .. (450,115.33) .. controls (448.33,113.66) and (448.33,112) .. (450,110.33) -- (450,110) -- (450,110) ;
\draw  [color={rgb, 255:red, 208; green, 2; blue, 27 }  ,draw opacity=1 ] (390.63,374.42) .. controls (389.25,384.03) and (387.95,393.19) .. (391.74,395.66) .. controls (395.54,398.12) and (403.38,393.21) .. (411.59,388.04) .. controls (419.81,382.87) and (427.65,377.96) .. (431.45,380.43) .. controls (435.24,382.89) and (433.94,392.05) .. (432.56,401.66) .. controls (431.18,411.27) and (429.88,420.43) .. (433.67,422.9) .. controls (437.46,425.36) and (445.3,420.45) .. (453.52,415.28) .. controls (461.74,410.11) and (469.58,405.2) .. (473.38,407.67) .. controls (477.17,410.13) and (475.87,419.29) .. (474.49,428.9) .. controls (473.11,438.51) and (471.81,447.67) .. (475.6,450.13) .. controls (479.39,452.6) and (487.23,447.69) .. (495.45,442.52) .. controls (503.67,437.35) and (511.51,432.44) .. (515.3,434.91) .. controls (519.1,437.37) and (517.8,446.53) .. (516.42,456.14) .. controls (515.03,465.75) and (513.73,474.91) .. (517.53,477.37) .. controls (521.32,479.84) and (529.16,474.93) .. (537.38,469.76) .. controls (545.6,464.59) and (553.44,459.68) .. (557.23,462.15) .. controls (558.42,462.92) and (559.11,464.35) .. (559.45,466.24) ;
\draw [color={rgb, 255:red, 208; green, 2; blue, 27 }  ,draw opacity=1 ]   (410,280) -- (390,380) ;
\draw  [color={rgb, 255:red, 208; green, 2; blue, 27 }  ,draw opacity=1 ] (390,300) .. controls (446.67,353.33) and (503.33,353.33) .. (560,300) ;
\draw [color={rgb, 255:red, 208; green, 2; blue, 27 }  ,draw opacity=1 ]   (550,480) .. controls (590,450) and (470,310) .. (510,280) ;
\draw [color={rgb, 255:red, 208; green, 2; blue, 27 }  ,draw opacity=1 ]   (400,330) ;
\draw [shift={(400,330)}, rotate = 0] [color={rgb, 255:red, 208; green, 2; blue, 27 }  ,draw opacity=1 ][fill={rgb, 255:red, 208; green, 2; blue, 27 }  ,fill opacity=1 ][line width=0.75]      (0, 0) circle [x radius= 3.35, y radius= 3.35]   ;
\draw [color={rgb, 255:red, 208; green, 2; blue, 27 }  ,draw opacity=1 ]   (430,420) ;
\draw [shift={(430,420)}, rotate = 0] [color={rgb, 255:red, 208; green, 2; blue, 27 }  ,draw opacity=1 ][fill={rgb, 255:red, 208; green, 2; blue, 27 }  ,fill opacity=1 ][line width=0.75]      (0, 0) circle [x radius= 3.35, y radius= 3.35]   ;
\draw [color={rgb, 255:red, 208; green, 2; blue, 27 }  ,draw opacity=1 ]   (520,330) ;
\draw [shift={(520,330)}, rotate = 0] [color={rgb, 255:red, 208; green, 2; blue, 27 }  ,draw opacity=1 ][fill={rgb, 255:red, 208; green, 2; blue, 27 }  ,fill opacity=1 ][line width=0.75]      (0, 0) circle [x radius= 3.35, y radius= 3.35]   ;
\draw  [color={rgb, 255:red, 208; green, 2; blue, 27 }  ,draw opacity=1 ] (389.25,124.94) .. controls (405.5,114.31) and (424.35,106.72) .. (432.81,116.59) .. controls (449.72,136.33) and (391.23,186.45) .. (387.33,181.89) .. controls (383.42,177.34) and (441.92,127.22) .. (458.83,146.97) .. controls (475.75,166.71) and (417.26,216.83) .. (413.35,212.27) .. controls (409.45,207.71) and (467.94,157.6) .. (484.86,177.34) .. controls (485.08,177.61) and (485.3,177.88) .. (485.5,178.15) ;
\draw [color={rgb, 255:red, 208; green, 2; blue, 27 }  ,draw opacity=1 ]   (370,160) .. controls (410,130) and (430,190) .. (470,160) ;

\draw (47,410.4) node [anchor=north west][inner sep=0.75pt]    {$V_{1}$};
\draw (147,482.4) node [anchor=north west][inner sep=0.75pt]    {$V_{2}$};
\draw (171,382.4) node [anchor=north west][inner sep=0.75pt]    {$V_{3}$};
\draw (67,482.4) node [anchor=north west][inner sep=0.75pt]    {$V_{0}$};
\draw (356,429.95) node [anchor=north west][inner sep=0.75pt]    {$Y_{0} \ =\ X$};
\draw (368,172.4) node [anchor=north west][inner sep=0.75pt]    {$X$};
\draw (368,301.95) node [anchor=north west][inner sep=0.75pt]    {$Y_{1}$};
\draw (538,330.4) node [anchor=north west][inner sep=0.75pt]    {$Y_{3}$};
\draw (521,452.4) node [anchor=north west][inner sep=0.75pt]    {$Y_{2}$};
\draw (341,120.4) node [anchor=north west][inner sep=0.75pt]    {$D_{1}$};
\draw (461,180.4) node [anchor=north west][inner sep=0.75pt]    {$D_{2}$};
\draw (67,232.4) node [anchor=north west][inner sep=0.75pt]    {$V_{0}$};

\end{tikzpicture}
\caption{This captures how the authors visualize expansions and subschemes in them. The two pictures on the left are typical tropicalizations of subschemes. The top left is the tropicalization of a subscheme in the interior of the moduli problem, while the bottom left is a more complicated one. The pictures on the left are outputs of the Tevelev theorem, and determine expansions of $X$ -- a gluing together of projective bundles over the strata of $X$. The reader can notice, for example in the second row, a bijection between components of the expansion and the vertices of the graph on the left. The circled component on the right, in both pictures, is the main component, i.e. $X$ itself. The subschemes drawn, indicated in red, can have singularities, or embedded components -- three embedded points can be seen on the bottom right. However (i) the subscheme must be disjoint from the codimension $2$ strata, i.e. the corners, and (ii) the embedded points must lie in the interiors of irreducible components.}
\label{fig: expansion}
\end{figure}

Precedent for building moduli spaces, particularly in contexts adjacent to toric geometry, via Tevelev's work is provided by Hacking--Keel--Tevelev, and the method has been used in Gromov--Witten theory before, see~\cite{HKT,KT06,R15b}. 

\subsection{The universal tropical expansion} By axiomatizing the output of Tevelev's argument, we propose a class of \textit{ideal sheaves on $X$ relative to $D$} for the Donaldson--Thomas moduli problem. These are subschemes of expansions of $X$ along $D$ that are transverse to the strata. In order to construct a global moduli problem, we identify an Artin stack that encodes the possible expansions of $X$ that could arise from Tevelev's procedure. 

The discussion above predicts the one-parameter degenerations of the target, but subtleties arise in extending them over higher dimensional bases, having to do with flattening the universal degeneration.  We tackle this by first studying an appropriate tropical moduli problem, using recent work that identifies a category of certain locally toric Artin stacks with purely combinatorial objects~\cite{CCUW}. 

The outcome of the tropical study is a system of moduli stacks of universal expansions, related to each other by birational transformations, and organized into an inverse system. Each element in this system is ``good enough'' to function as a stack of expansions for our moduli problem, but there is typically no distinguished choice. The system depends only on the combinatorics of the boundary divisor $D\subset X$.  Both the moduli space and its universal family have this structure and compatible choices give rise to a universal degeneration.

\subsection{Moduli space of stable relative ideal sheaves} After fixing a stack of expansions, we define a notion of DT stability for ideal sheaves on $X$ relative to $D$.  An important subtlety appears when considering \textit{tube subschemes}, namely subschemes in a component of an expansion that are pulled back from a surface in its boundary.  In relative DT theory, these are ruled out by stability; in our setting, they are forced on us by the combinatorial algorithms in an analogous fashion to how \textit{trivial bubbles} arise in the stable maps geometry.  
In any fiber of the universal family of the stack of expansions, there are distinguished irreducible components, denoted tube components;  the DT stability condition we impose is that these are precisely the components which host tube subschemes.  We show this defines a moduli problem with the expected properties.  
The transversality hypotheses guarantee that the morphism to the moduli stack of expansions
\[
\mathsf{DT}_{\beta,n}(X|D)\to \mathsf{Exp}(X|D)
\] 
has a perfect obstruction theory and consequently a virtual class. This establishes an appropriate Donaldson--Thomas theory for the pair $(X|D)$. In Section~\ref{sec: revisiting-Li-Wu}, we explain how, in hindsight, the tube geometry above can be artificially introduced into Li--Wu's theory, and why it can be avoided it that case.

{
\subsection{Pairs, etc} We have chosen to focus on the moduli theory of ideal sheaves in this paper, but the methods appear to be adaptable to other settings. In particular, one can define logarithmic stable pair invariants by replacing the Hilbert scheme with the stable pair moduli spaces of Pandharipande and Thomas.   The stable pair adaptations are recorded in Remark~\ref{rem: stable-pairs}.   It seems reasonable to hope for further applications.  For instance, it may be possible to rederive logarithmic Gromov--Witten theory, relying on target expansions from the very beginning. In another direction, the logarithmic theory of quasimaps has only been treated in the smooth pair case~\cite{BN19}. 
}

\subsection{Toroidal embeddings} The expectation is that the theory set up in this paper can be extended to any logarithmically smooth target, without either the simple normal crossing or connectivity restrictions placed on $D$. The arguments in the present paper carry over with cosmetic changes to treat generalized Deligne--Faltings logarithmic structures~\cite{AC11}. This includes all singular toric varieties. A more delicate modification of the combinatorics can likely be used to treat divisor geometries with disconnected intersections.  These two variants, together with the case where $D$ has a self-intersecting component, will be addressed elsewhere. Logarithmically \'etale descent and virtual birational invariance techniques are likely to play a role, see~\cite{ACMW,AW}. 

\subsection{Further directions and recent progress}  
We mention briefly some natural directions to pursue with the theory constructed here.
In a recent sequel \cite{MR22} to this paper,  we develop a degeneration formalism, generalizing that of \cite{LiWu15}, and parallel to~\cite{R19}.  We also develop a logarithmic version of the GW/Pairs correspondence, and show it is compatible with normal crossings degenerations.  With this in place, our subsequent goal is to extend the inductive strategy of Pandharipande--Pixton and prove the GW/Pairs correspondence for a broader class of threefolds, i.e. varieties which are not easily studied by double-point degenerations.  For example, one can envisage a proof GW/Pairs for threefolds admitting an algebro-geometric SYZ fibration, i.e. a normal crossings degeneration to a union of rational varieties. A natural class of examples comes from taking zero loci of sections of toric vector bundles.  

In another direction, the formalism of relative DT theory (in cohomology and $K$-theory) interacts well with the representation-theoretic structure on the Hilbert scheme of points on a surface \cite{maulik-okounkov}, and we expect that our logarithmic theory will extend this circle of ideas.  In a similarly speculative vein, logarithmic Gromov--Witten invariants in genus $0$ are related to the symplectic cohomology of the open variety $X \backslash D$, which is the natural replacement for quantum cohomology for open geometries. It would be interesting to examine whether logarithmic Donaldson--Thomas theory can be related to the symplectic cohomology of Hilbert schemes of points on open surfaces.

Our focus here is on logarithmic moduli spaces of subschemes of dimension at most $1$, due to the applications for DT theory. The methods here make use of the simplicity of the transversality condition for $1$-dimensional subschemes, and the simpler combinatorics in this case. However, in the time since this paper first appeared on ar$\chi$iv, further progress has been made.  In recent work, Kennedy-Hunt~\cite{KH22} proposes a general logarithmic Quot scheme with no constraints on the dimension of the support, building on the techniques introduced here.


\subsection*{Outline of paper}  We briefly outline the sections of this paper.  In Section~\ref{sec: flavours}, we review some basic constructions and results from tropical geometry.  In Section~\ref{sec: flat-limits}, we give the tropical algorithm for constructing algebraically transverse flat limits of subschemes.
In Section~\ref{sec: target-moduli}, we construct the tropical moduli spaces of expansions as well as their geometric counterparts.  In Section~\ref{moduli-of-sheaves}, we define stable relative ideal sheaves and show their moduli functor is represented by a proper Deligne--Mumford stack.  In Section~\ref{sec: log-DT-theory}, we study the virtual structure on the moduli space, define logarithmic DT invariants, and state the basic conjectures.  In Section~\ref{sec: first-examples}, we give a handful of simple examples, demonstrating the basic theory.  In Section~\ref{sec:expanded}, we complete the proof of the valuative criterion of properness, initiated in Section~\ref{sec: flat-limits}, to deal with the case where the generic fiber is expanded.  

\subsection*{Background and conventions} We have put some effort into minimizing the amount of logarithmic geometry that is explicitly used in this paper, and there is nothing that we use beyond~\cite[Sections~1--5]{ACMUW}. We do use the combinatorics of cone complexes and cone spaces heavily, and refer the reader to~\cite[Sections~2 {\it \&} 6]{CCUW}. Logarithmic schemes and stacks that appear will be fine and saturated unless otherwise specified, locally of finite type, and over the complex numbers. 

\subsection*{Acknowledgements} We are grateful to Dan Abramovich, Dori Bejleri, Qile Chen, Johan de Jong, Mark Gross, Eric Katz, Max Lieblich, Hannah Markwig, Navid Nabijou, Andrei Okounkov, Sam Payne, Mattia Talpo, Richard Thomas, Martin Ulirsch, Jeremy Usatine, and Jonathan Wise for numerous conversations over the years on related topics. We extend a special thanks to Patrick Kennedy-Hunt, for numerous questions and comments on earlier versions of this paper that helped improve both the presentation and the mathematics. The first author would like to thank Daniel Tiger for providing extra childcare during the final stages of this project. The second author was a Moore Instructor at MIT and a visitor at the Chennai Mathematical Institute during work on this project, and is grateful to both institutions for excellent working conditions. Two anonymous referees provided detailed feedback which led to various improvements to both exposition and mathematics.

\noindent
D.R. is partially supported by EPSRC grant EP/V051830/1. 


\section{Flavours of tropicalization}\label{sec: flavours}

We require some elementary notions from logarithmic geometry, and a reference that is well suited to our point of view is~\cite[Sections~3--5]{ACMUW}. There are a few different ways in which tropicalizations arise in logarithmic geometry and we recall these for the reader.

\subsection{Cone complexes and their morphisms} We start with the building blocks of toric geometry. A \textit{polyhedral cone with integral structure} $(\sigma,M)$ is a topological space $\sigma$ equipped with a finitely generated abelian group $M$ of continuous real-valued functions from $\sigma$ to $\mathbb R$, such that the evaluation
\[
\sigma\to\Hom(M,\RR)
\]
is a homeomorphism onto a strongly convex polyhedral cone. If this cone is \textit{rational} with respect to the dual lattice of $M$, then we say that $(\sigma,M)$ is rational. We define the lattice of integral points of $\sigma$ by taking the preimage of the dual lattice of $M$ under the evaluation map.  The set of elements of $M$ that are non-negative on $\sigma$ form a monoid $S_\sigma$ referred to as the \textit{dual monoid}. The cone $\sigma$ is recovered as the space of monoid homomorphisms $\Hom(S_\sigma,\RR_{\geq 0})$.  

Henceforth, a \textit{rational polyhedral cone with integral structure} will be referred to as a \textit{cone}. 

\begin{definition}[{Cone complexes}]
A \textit{rational polyhedral cone complex} is a topological space that is presented as a colimit of a partially ordered set of cones, where all arrows are given by isomorphisms onto proper faces. A \textit{morphism} of cone complexes is a continuous map
\[
\Sigma'\to \Sigma
\]
such that (i) the image of every cone in $\Sigma'$ is contained in a cone of $\Sigma$ and (ii) the restriction of the map to any cone of $\Sigma'$ is given by an integer linear map. 
\end{definition}

Cone complexes are nearly identical the \textit{fans} considered in toric geometry~\cite{Ful93}. The key differences are that they do not come equipped with a global embedding into a vector space, and two cones can intersect along a union of faces. However, we will typically restrict to the case where an intersection of two cones is a face of each, so the main thing to keep in mind is the lack of a global embedding. Concretely, the reader may keep in mind that toric fans of $\mathbb P^1$ and $\mathbb A^2\setminus\{(0,0)\}$ are isomorphic as cone complexes.

A cone complex is \textit{smooth} if every cone is isomorphic to a standard orthant with its canonical integral structure. We record two combinatorial notions associated to morphisms between cone complexes. 

\begin{definition}[Flat maps and reduced fibers]
Let $\Sigma$ be a smooth cone complex and let $\pi: \Sigma'\to \Sigma$ be a morphism of cone complexes. Then $\pi$ is \textit{flat} if the image of every cone of $\Sigma'$ is a cone of $\Sigma$. A flat map is said to have \textit{reduced fibers} if for every cone $\sigma'$ of $\Sigma'$ with image $\sigma$, the image of the lattice of $\sigma'$ is equal to the lattice of $\sigma$. 
\end{definition}

The terminology is compatible with the identically named geometric notions, when applied to toric maps~\cite[Section~4 {\it \&} 5]{AK00}.

\subsection{Tropicalization for the target} A basic fact from the theory of logarithmic schemes is that every Zariski logarithmically smooth scheme gives rise to a cone complex. We take a moment to unpack this statement in concrete terms. Let $(X|D)$ be a smooth scheme of finite type over $\CC$ equipped with a simple normal crossings divisor $D$. Assume that the intersections of irreducible components of $D$ are always connected. The presence of $D$ gives $X$ the structure of a logarithmically smooth scheme. We let $X^\circ$ be the complement of the divisor, where the logarithmic structure is trivial. 

We unwind the definition of the logarithmic structure in this case for the benefit of the reader. The components of $D$ give rise to a distinguished class of functions in the structure sheaf. For an open set $U\subset X$, we record the values of the logarithmic structure sheaf and the characteristic monoid sheaf\footnote{These may regarded as sheaves in either the Zariski or \'etale topology. }:
\[
M_X(U) = \left\{f\in\mathcal O_X(U): f \textrm{ is invertible on } U\setminus D\right\} \ \ \ \textrm{and} \ \ \overline M_X(U) = M_X(U)/\mathcal O_X^\star(U).
\]
Locally at each point $x\in X$, there are functions, canonical up to multiplication by a unit, cutting out the irreducible components of $D$ passing through $x$. In particular, the stalk of the characteristic monoid sheaf at $x$ is naturally identified with $\NN^e$, where $e$ is the number the such components. 

The dual cones of the stalks of the characteristic monoids of $X$ give rise to a collection of cones, one for each point of $X$. The generization maps naturally give rise to gluing morphisms, and these cones form a finite type cone complex. We denote it by $\Sigma_X$, refer to it as the \textit{tropicalization} of $X$ or the \textit{cone complex of $X$}, and regard it as a cone complex equipped with an integral structure.  

In our case, the tropicalization has a more practical description. Given the pair $(X,D)$, let $d$ be the number of irreducible components of $D$ and enumerate these components $D_1,\ldots, D_d$. Each $k$-dimensional face of the cone $\RR_{\geq 0}^d$ is spanned by rays $v_{i_i},\ldots, v_{i_k}$. Call such a face \textit{relevant to $(X,D)$} if the corresponding intersection $D_{i_1},\ldots, D_{i_k}$ is nonempty. Then the tropicalization $\Sigma_X$ is the union of cones in $\RR_{\geq 0}^d$ that are relevant to $(X,D)$. In the case of a toric variety this construction recovers the fan, as an abstract cone complex.  The construction extends to the case of logarithmically smooth schemes~\cite{ACP,U15}. Further details and generalizations may be found in the references~\cite{ACMUW,CCUW,Kat89,KKMSD}.

\subsection{Subdivisions}\label{sec: subdivisions} Let $(X|D)$ be a simple normal crossings pair with tropicalization $\Sigma_X$. For simplicity, we assume that the non-empty intersections of components of $D$ are connected. We will produce target expansions by using subdivisions of the tropicalization.

\begin{definition}
A \textit{subdivision} is a cone complex $\Delta$ and a morphism of cone complexes
\[
\Delta\hookrightarrow \Sigma_X
\]
that is injective on the support of $\Delta$ and further, such that the integral points of the image of each cone $\tau\in \Delta$ are exactly the intersection of the integral points of $\Sigma_X$ with $\tau$. 
\end{definition}

Note that this more flexible than the standard definition; the underlying map of sets need not be a bijection because we wish to have the flexibility to discard closed strata after blowing up. When the map on integral points is a bijection we call it a \textit{complete subdivision}. 

In toric geometry, subdivisions give rise to possibly non-proper birational models of $X$. The same is true for $(X|D)$. The cone complex $\Sigma_X$ associated to $D$ is smooth, and the intersection of any two cones is a face of each. If  there are $d$ rays in $\Sigma_X$, we may embed it via
\[
\Sigma_X\hookrightarrow \RR^{d}_{\geq 0},
\]
by mapping each ray in $\Sigma_X$ isomorphically onto the positive ray on the corresponding axis. The positive orthant in this vector space is the fan associated to the toric variety $\mathbb A^d$ with dense torus $\mathbb G_m^d$. The subdivision
\[
\Delta\hookrightarrow \Sigma_X\hookrightarrow \RR^d_{\geq 0}
\]
defines a non-complete fan with associated toric variety $\A_\Delta$, equipped with an $\mathbb G_m^d$-equivariant birational map $\A_\Delta\to \A^d$. By passing to quotients, we have a morphism of stacks
\[
[\A_\Delta/\mathbb G_m^d]\to[\A^d/\mathbb G_m^d].
\]
The presence of $D$ gives rise to a tautological morphism $(X|D)\to [\A^d/\mathbb G_m^d]$.

\begin{definition}\label{def: subdivision}
The {\it birational model of $X$ associated to the subdivision $\Delta\hookrightarrow \Sigma$} is given by 
\[
X_\Delta:= X\times_{ [\A^d/\mathbb G_m^d]} [\A_\Delta/\mathbb G_m^d].
\]
\end{definition}

\begin{remark}\label{rem: divisorial}
In practice, we apply this construction to $X\times \A^1$ with the divisor $X\times \{0\}\cup D\times \A^1$. The deformation to the normal cone of a stratum of $D$ is a special case of the construction.
\end{remark}

\begin{remark}
It is natural to formulate subdivisions in terms of subfunctors of the functor on logarithmic schemes defined by $X$ and by its cone complex. The birational model associated to a subdivision is defined by pulling back the subdivision along the tropicalization map, see~\cite{CCUW,Kat89}. 
\end{remark}

\subsection{Tropicalization for subschemes via valuations}\label{sec: trop-via-val} In this section we assume that $X$ is proper, with the exception of Remark~\ref{rem: non-compact-trop} and situations where it becomes active. 

Locally at each point on $X$, the components of the divisor $D$ provide a distinguished set of functions -- a subset of coordinates -- at that point. The tropicalization of a subscheme is the image of its ``coordinatewise valuation'', in these coordinates. 

Let $(K,\nu)$ be a rank $1$ valued field extending $\CC$ and let $\nu$ be the valuation. Consider a morphism
\[
\spec K \to X^\circ.
\] 
Since $X$ is proper, this morphism extends to the valuation ring
\[
\spec R\to X.
\]
Let $x$ be the image of the closed point. The smallest closed stratum of $D$ containing $x$ is the intersection of a (possibly empty) subset of irreducible divisor components $D_{i_1},\ldots, D_{i_r}$. The associated equations generate the stalk of the characteristic sheaf at the point $x$, which is abstractly isomorphic to $\NN^r$. For each element $\overline f \in \NN^r$, we may lift it to a function $f$ on $X$ in a neighborhood of the point $x$, pull back to $\spec R$ along the map above, and compose with the valuation on $R$. The ratio of two such lifts is a unit, so this gives rise to a well-defined element
\[
[\NN^r\to \RR_{\geq 0}]\in \Sigma_X\subset \Hom(\NN^r,\RR_{\geq 0}).
\]
For any \textit{valued} field $K$, we have a well-defined morphism
\[
\trop\colon X^\circ(K)\to \Sigma_X.
\]
Let $K$ be a valued field whose associated valuation map $K^\times\to \RR$ is \text{surjective}. Let $Z^\circ \subset X^\circ$ be a subscheme. Let $\trop(Z^\circ)$ be the subset of $\Sigma_X$ obtained by restricting $\trop$ to $Z^\circ(K)$; it is independent of the choice of valued field $K$ and its functoriality properties are outlined in~\cite{Gub13,U13}.\footnote{Comprehensive treatments of tropicalization regularly make use of Berkovich's analytification. We omit this here to avoid additional technical machinery, although the proofs of these basic results use them. See~\cite{ACMUW} for a survey of the relationship between tropical, logarithmic, and non-archimedean geometry.}

\subsection{Properties of the tropicalization} The shapes of tropicalizations are governed by the Bieri--Groves theorem~\cite{BG84,U15}.

\begin{theorem}\label{thm: bieri-groves}
Let $Z^\circ\subset X^\circ$ be a closed subscheme. Then the set $\trop(Z^\circ)$ is the \textit{support} of a rational polyhedral cone complex of $\Sigma_X$. The topological dimension of $\trop(Z^\circ)$ is bounded above by the algebraic dimension of $Z^\circ$. If $X^\circ$ is a closed subvariety of an algebraic torus, the topological dimension of $\trop(Z^\circ)$ is equal to the algebraic dimension of $X$. 
\end{theorem}

The set $\trop(Z^\circ)$ has no distinguished polyhedral structure in general~\cite[Example~3.5.4]{MS14}. It is simply a set, and this set \textit{can} be given the structure of a cone complex.

The role of tropicalization in degeneration and compactification problems has its origin in the following two theorems, proved by Tevelev for toric varieties and Ulirsch for logarithmic schemes, see~\cite[Theorem~1.2]{Tev07} and~\cite[Theorem~1.2]{U15}. 

The first concerns the properness of closures of subschemes in partial compactifications of $X^\circ$. Let $Z^\circ$ be a subscheme of $X^\circ$ and let $X'$ be a simple normal crossings compactification. Let $X''\subset X'$ be the complement of a union of closed strata of $X'$ and let $\Sigma_{X''}$ be the subfan of $\Sigma_{X'}$ obtained by deleting the corresponding union of open cones.

\begin{theorem}\label{thm: properness-criterion}
The closure $Z$ of $Z^\circ$ in the partial compactification $X''$ is proper if and only if $\trop(Z^\circ)$ is set theoretically contained in $\Sigma_{X''}$.
\end{theorem}

The second concerns transversality. We say the closure $Z$ of $Z^\circ$ in $X'$ \textit{intersects strata in the expected dimension if} 
\[
\dim E\cap Z = \dim Z-\mathsf{codim}_{X'}E.
\]

\begin{theorem}\label{thm: transversality-criterion}
The closure $Z$ of $Z^\circ$ in the compactification $X'$ of $X^\circ$ intersects strata in the expected dimension if and only if $\trop(Z^\circ)$ is a union of cones in $\Sigma_{X'}$. 
\end{theorem}

\begin{remark}
When $X$ is a toric variety, Tevelev has shown that there exists a toric blowup $X'\to X$ such that the closure has a stronger transversality property, hinted at in the introduction, called \textit{algebraic transversality}. We will require and refine this result in the course of our main result~\cite[Theorem~1.2]{Tev07}; a simple proof is given in~\cite[Theorem~6.4.17]{MS14}.
\end{remark}

\begin{remark}\label{rem: non-compact-trop}
If $X$ is not proper, there is no longer a tropicalization map defined on the set $X^\circ(K)$, because limits need not exist. However, the above relationship with the cone complex persists. For each valued field $K$ with valuation ring $R$ as above, let $X^\beth(K)$ be the subset $X^\circ(K)$ consisting of those $K$-points that extend to $R$-points. There is a morphism
\[
X^\beth(K)\to \Sigma_X
\]
defined exactly as defined above. Similarly, given a subscheme $Z\subset X^\circ$, we can define its tropicalization as the image of $Z^\beth(K)$ in $\Sigma_X$, see~\cite[Section~5.2]{U13}. 
\end{remark}

\subsection{Tropicalization via compactifications} In the previous section, tropicalizations were seen to select partial compactifications of $X^\circ$, in which the closure of a subvariety meets each stratum in the expected codimension. There is a partial converse. 

Let $X'$ be a simple normal crossings compactification of $X^\circ$. Let $Z$ be a subscheme of $X'$ whose intersection with the strata of $X'$ have the expected codimension, with
\[
Z^\circ = X^\circ\cap Z.
\] 
The following perspective is due to Hacking--Keel--Tevelev, extended by Ulirsch, and is sometimes referred to as \textit{geometric tropicalization}, see~\cite{HKT,U15}.

\begin{theorem}\label{thm: divisorial-trop}
The tropicalization of $Z^\circ$ is equal to the union of cones $\sigma$ in the underlying set of $\Sigma_{X'}$ such $Z$ nontrivially intersects the locally closed strata dual to $\sigma$. That is, there is an equality of subsets of $\Sigma_{X'}$ given by
\[
\trop(Z^\circ) \ \ \  {=}\bigcup_{\sigma\in \Sigma_X: V(\sigma)\cap Z \neq \emptyset} \sigma\ \ \ \ \ \ \ \ \textrm{in} \ \ \ \ \ |\Sigma_X|,
\]
where $V(\sigma)$ is the locally closed stratum of $X'$ corresponding to the cone $\sigma$. 
\end{theorem}

The definition of $\trop(Z^\circ)$ via coordinatewise valuation described in the previous subsection depends on a choice of compactification $X$ of $X^\circ$, but if $X$ is replaced with a blowup along a stratum, the set $\trop(Z^\circ)$ is unchanged. The result above can therefore be viewed as a computational tool; it describes the tropicalization using a single compactification in which $Z$ is dimensionally transverse. We return to this in Section~\ref{sec:expanded}.

\begin{remark}[Asymptotics and stars]\label{rem: asymptotics}
Let $Z^\circ\subset X^\circ$ be a subscheme with tropicalization $\trop(Z^\circ)$ in $\Sigma_X$. Assume that the tropicalization is a union of cones in $\Sigma_X$ and let $Z$ denote the closure of $Z^\circ$. If $D_i\subset X$ is an irreducible component, we can view it as a simple normal crossings pair in its own right with interior $D_i^\circ$ and divisor equal to the intersection of $D_i$ with the remaining components of $D$. It contains the subscheme $Z_i^\circ = Z\cap D_i^\circ$. The tropicalization of $Z_i^\circ\subset D_i^\circ$ can be read off from the larger $\trop(Z^\circ)$ as follows. The divisor $D_i$ determines a ray $\rho_i$ in $\Sigma_X$. The star of $\rho_i$ is the union of cones that contain $\rho_i$ and can be identified with $\Sigma_{D_i}$. The tropicalization of $\trop(Z_i^\circ)$ is the union of the cones under this identification where $\trop(Z^\circ)$ itself is supported. In practice, it is visible as the collection of asymptotic directions of $\trop(Z^\circ)$ parallel to $\rho_i$. 
\end{remark}

\subsection{Tropicalization for a family of subschemes} The constructions extend to flat one-parameter families of subschemes. We will consider subschemes in $X$ that are defined over a valued field $K$ that extends the trivially valued ground field $\CC$. 

In order to avoid foundational issues, we will assume that all subschemes are defined over the localization of a smooth algebraic curve of finite type. The assumption is made so we can appeal to the relationship between tropicalization via valuations and via logarithmic geometry, which has only received a definitive treatment under these hypotheses~\cite{U15}. The valued field will arise for us in the study of the valuative criterion for properness. The relevant moduli spaces will be shown to be of finite type, so this is a harmless assumption. 

\subsubsection{Tropicalization over a valued field} Let $(X|D)$ be a simple normal crossings compactification as above, with interior $X^\circ$. Let $K$ be a valued field extending $\CC$ as above. Consider an algebraically transverse subscheme
\[
\mathcal Z_\eta\hookrightarrow X\times_{\CC} K,
\]
and let $\mathcal Z_\eta^\circ$ be the open subscheme contained in $X^\circ$. By the algebraic transversality hypothesis, this open subscheme is dense. After passing from $K$ to a valued extension $L$ with real surjective valuation, we may once again consider the tropicalization map
\[
\trop\colon \mathcal Z^\circ_\eta(L)\to \Sigma_X.
\]
The tropicalization is independent of the choice of $L$, provided the valuation is surjective onto the real numbers.\footnote{In the toric setting, the independence of $L$ follows, for example, from non-emptiness results for the fibers of the tropicalization map~\cite{Pay07}, applied to the toric charts of $X$, following~\cite{U15}. It can also be deduced from elementary properties of the Berkovich analytification and the formalism in~\cite{ACP,U15}.}

The basic structure result extends to this nontrivially valued setting with a small twist.

\begin{theorem}
The set $\trop(\mathcal Z_\eta^\circ)$ is the \textit{support} of a rational polyhedral complex of $\Sigma_X$. The topological dimension of $\trop(\mathcal Z_\eta^\circ)$ is bounded above by the algebraic dimension of $\mathcal Z^\circ$. If $X^\circ$ is a closed subvariety of an algebraic torus, then the topological dimension of the tropicalization is equal to the algebraic dimension of $\mathcal Z^\circ$.
\end{theorem}

 The toric case of this result was established by Gubler without restriction on the base field~\cite{Gub13}. The statement for simple normal crossings targets is a consequence of~\cite[Theorem~5.1]{BU18}.

To emphasize the point, the tropicalization of a variety that is defined over a nontrivial valued extension of $\CC$ is polyhedral but \textit{not necessarily conical}. In other words, it can have bounded cells. The target $X$ itself is still defined over $\CC$, not merely over $K$, so {\it its} tropicalization remains conical.

\subsubsection{Families over a punctured curve} Let $C$ be a smooth algebraic curve of finite type with a distinguished point $0\in C$, let $C^\circ$ be the complement of this point. The pair $(C,0)$ has a cone complex $\Sigma_C$, canonically identified with the positive real line $\RR_{\geq 0}$. 

Let $(X|D)$ be a pair as before and $X^\circ$ its interior. Consider a flat family of $1$-dimensional subschemes over $C^\circ$
\[
\mathcal Z^\circ\subset C^\circ\times X^\circ.
\]
There are two tropicalizations associated to this family. For the first, we consider $\mathcal Z^\circ$ as a $2$-dimensional subscheme of $C^\circ\times X^\circ$. The partial compactification $C\times X$ determines a tropicalization, as explained in Remark~\ref{rem: non-compact-trop}: 
\[
\trop(\mathcal Z^\circ)\subset \RR_{\geq 0} \times \Sigma_X. 
\]
For the second, note that the function field of $C$ is equipped with a discrete valuation arising from order of vanishing at $0$. Let $K$ be an extension of this field with real surjective valuation and note that there is a canonical inclusion of $\spec K$ to $C^\circ$. Consider the tropicalization of the $K$-valued points of the base change
\[
\mathcal Z_\eta = \mathcal Z^\circ\times_{C^\circ} K\subset X(K),
\]
where the morphism $\spec K \rightarrow C^\circ$ is determined by the fraction field of the local ring of $C$ at $0$.  Denote the result by $\trop(\mathcal Z_\eta)$ -- note that the subscript $\eta$ indicates that we are considering it as a generic fiber. 

These two procedures are related. When considering the total space of $\mathcal Z^\circ$ as a surface over $C^\circ$, there is a map
\[
\upsilon\colon \trop(\mathcal Z^\circ)\to \RR_{\geq 0}.
\]
The fiber $\upsilon^{-1}(1)$ coincides with the second tropicalization $\trop(\mathcal Z_\eta)$. This follows from a tracing of definitions for the tropicalization of fibers of maps of Berkovich spaces and the functoriality results in~\cite{U15}. 

\subsection{Transversality for one-parameter families} The appropriate generalizations of the properness and transversality statements earlier in this section are as follows. Let $R$ be the valuation ring of $K$ and equip $\spec R$ with the divisorial logarithmic structure at the closed point. The scheme $X\times \spec R$ is equipped with the simple normal crossings divisor given by $D\times\spec R\cup X\times 0$. Once equipped with this divisor, it has tropicalization $\Sigma_X\times\RR_{\geq 0}$. 

Let $\mathcal Y'\to X\times\spec R$ be a toroidal modification of the constant family and let $\mathcal Y''\subset \mathcal Y'$ be the complement of a union of closed strata. In practice, $\mathcal Y''$ will be the complement of all closed strata of codimension $2$ in the special fiber. These have tropicalizations $\Sigma_{\mathcal Y'}$ and $\Sigma_{\mathcal Y''}$ that are, respectively, a complete subdivision and a subdivision of $\Sigma_X\times\RR_{\geq 0}$. There is a morphism of cone complexes by composition:
\[
\Sigma_{\mathcal Y''}\to \RR_{\geq 0}
\]
and we denote the fiber over $1$ by $\Sigma_{\mathcal Y''}(1)$. We view this as a polyhedral complex. 

Let $\mathcal Z^\circ$ be a subscheme of $X^\circ\times \spec K$. We examine the question of when the closure is proper. 

\begin{theorem}
The closure $\mathcal Z$ of $\mathcal Z^\circ$ in the degeneration $\mathcal Y''$ of $X$ is proper over $\spec R$ if and only if $\trop(\mathcal Z^\circ)$ is set theoretically contained in $\Sigma_{\mathcal Y''}(1)$. 
\end{theorem}

\begin{proof}
This is well-known to experts, but in the form stated, we have been unable to locate a suitable reference. We explain how it can be deduced from results that do appear in the literature. 

First, we note that the closure of $\mathcal Z^\circ$ in the larger degeneration $\mathcal Y'$ is certainly proper over $\spec R$ because $X$ is proper and $\mathcal Y'\to X\times\spec R$ is a proper and birational morphism. We now use the hypothesis on $K$. Specifically, since $K$ is a localization of the function field of a smooth curve $C$ with the valuation associated to a closed point $0$ with complement $C^\circ$. Moreover, the given subscheme $\mathcal Z^\circ$, as well as the degenerations $\mathcal Y'$ and $\mathcal Y''$ can be assumed to arise via base change from corresponding families over $C$, along the inclusion
\[
\spec R\to C
\]
of the local ring at $0$. Rather than overburdening the notation, we replace the families over $\spec R$ with the corresponding families over $C$. We now view $C$, $X\times C$, $\mathcal Y'$ and $\mathcal Y''$ as logarithmic schemes over $\CC$. Note that due to the potential non-properness of $C$, Remark~\ref{rem: non-compact-trop} is in effect. The reader is also advised to keep in mind the relationship between the two tropicalizations of $\mathcal Z^\circ$ described above.

We are now in a position to apply the Tevelev--Ulirsch Lemma for logarithmic schemes as stated in~\cite[Lemma~4.1]{U15}. From it, we deduce that the closure of $\mathcal Z^\circ$ in $\mathcal Y'$ coincides with the closure in $\mathcal Y''$ precisely under the hypotheses stated in the theorem. We conclude the result. 
\end{proof}

We keep the notation above, and now deal with the corresponding transversality statement. 

\begin{theorem}
The closure $\mathcal Z$ of $\mathcal Z^\circ$ in the degeneration $\mathcal Y''$ of $X$ intersects the strata of $\mathcal Y'$ in the expected dimension if and only if $\trop(\mathcal Z^\circ)$ is a union of polyhedra in $\Sigma_{\mathcal Y''}(1)$. 
\end{theorem}

\begin{proof}
Proceed as in the proof of the theorem above and spread out the family until it is defined over a curve. To calculate the dimension of the intersections of $\mathcal Z$ with the strata of $\mathcal Y$, the map to $\spec R$ is not relevant, so we directly apply Theorem~\ref{thm: transversality-criterion} in the previous section and conclude. 
\end{proof}

\section{Tropical degenerations and algebraic transversality}\label{sec: flat-limits}

Let $(X|D)$ be a simple normal crossing pair such that the all intersections of irreducible components of $D$ are connected.  We further assume that $X$ is proper.
In this section, we first introduce the precise notion of expansion of $X$ that we consider in this paper.  We then examine how to use tropical data to construct algebraically transverse flat limits for families of subschemes of $X$, along the lines of work of Tevelev~\cite{Tev07} and Ulrisch~\cite{U15}.  This will motivate our construction of the stack of expansions in the next section, and provide the main ingredient in our proof of properness. Throughout this section, we specialize our discussion to subschemes of dimension at most $1$.

\subsection{The plan} Given a simple normal crossings pair $(X|D)$, we have recalled in the section above how a \textit{conical subdivision} of its tropicalization $\Sigma_X$ determines an open subset in a proper birational modification of $X$. A conical subdivision of $\Sigma_X\times\RR_{\geq 0}$, by the same dictionary, will determine a \textit{expansion}, a special type of degeneration, of $X$ along its boundary $D$. It leads to a family defined over $\A^1$. 

A conical subdivision 
\[
\widetilde{\Sigma_X\times\RR_{\geq 0}}\to \Sigma_X\times\RR_{\geq 0}
\]
can be visualized as a {\it polyhedral} subdivision of $\Sigma_X$: take the hight $1$ slice of the subdivision under the projection to $\RR_{\geq 0}$. The result is a union of polyhedra in $\Sigma_X$ glued along faces -- a polyhedral complex. The cone over this polyhedral subdivision recovers the original conical subdivision. 

As we have seen in the previous section, the tropicalization of a $1$-dimensional subscheme of $X$, as defined in the previous section, determines a polyhedral subdivision and therefore an expansion. 

We capture the possible tropicalizations of subschemes by the notion of a \textit{$1$-complex}, and use them to prove flat limit algorithms which will eventually establish the properness and separatedness of our yet-to-be-proposed moduli spaces. In order to build these moduli, we will form the parameter space for such $1$-complexes, and use logarithmic geometry to turn this into a parameterizing stack for expansions. The logarithmic Hilbert scheme will be built on top of this stack of expansions.

\subsection{Graphical preliminaries}\label{sec: graphical-prelims} Let $\underline G$ be a finite graph, possibly disconnected but without loops or parallel edges. We enhance it with two additional pieces of data. The first is a finite set of \textit{rays}, formally given by a finite set $R(\underline G)$ equipped with a map to the vertex set
\[
r\colon R(\underline G)\to V(\underline G).
\]
The second is the \textit{metrization} of the edge set, given by the edge length function
\[
\ell\colon E(\underline G)\to \RR_{>0}.
\]

\begin{definition}
An \textit{abstract $1$-complex} is a triple $(\underline G, r\colon R(\underline G)\to V(\underline G),\ell\colon E(\underline G)\to \RR_{>0})$ consisting of a finite graph, a collection of rays, and an edge length on the edges.
\end{definition}

These data give rise simultaneously to a metric space and a polyhedral complex, both enhancing the topological realization of $\underline G$. The topological realization of $\underline G$ is a CW complex and we endow an edge $E$ with a metric by identifying it with an interval in $\RR$ of length $\ell(E)$. For each element $ h\in R(\underline G)$ we glue on a copy of the metric space $\RR_{\geq 0}$ to the point $r(h)$. As we are now free to think of each edge or half edge as being either a polyhedron or a metric space, the result is a space $G$ that is simultaneously a metric space and a polyhedral complex of dimension at most $1$. In particular, it makes sense to talk about real-valued continuous piecewise affine functions on $G$.


Let $\Sigma$ be a smooth cone complex such that the intersection of any two cones is a face of each. Let $|\Sigma|$ be the associated topological space. Note that $\Sigma$ embeds canonically as a subcomplex of a standard orthant via $\Sigma\hookrightarrow \RR_{\geq 0}^{\Sigma^{(1)}}$.

A \textit{piecewise affine map $F:G\to\Sigma$} is the data of a continuous map on the underlying topological spaces
\[
|G|\to \underline |\Sigma|,
\]
such that every face of $G$ maps to a cone of $\Sigma$, and such that the map is integer affine upon restriction to each face.  

We can measure the \textit{slope along an edge of $G$}, well-defined up to sign, as follows. Each non-contracted edge $E$ maps to a cone $\sigma$, and thus maps onto a line $L_E$ in $\sigma^{\mathsf{gp}}$. We refer to the expansion factor of the induced map $F:E\to L_E$ as the \textit{slope of $F$ along the edge $E$}. 

\begin{definition}
An \textit{embedded $1$-complex} in $\Sigma$ is a piecewise affine map, written,
\[
G\to \Sigma
\]
that is injective on underlying topological spaces subject to the following conditions: (i) the slope along all edges $E$ in $G$ is equal to $1$, and (ii) the image of each ray of $G$ is parallel to a one-dimensional face of the cone containing it.  
\end{definition}

This latter condition will be satisfied for all the embedded $1$-complexes that occur for us, since it will be forced by algebraic transversality of subschemes in expansions. The reader can also drop condition, as these would correspond to components of the stack of expansions that are not relevant for our moduli spaces.

When it is clear from context that $G$ has been embedded, we refer to it simply as a \textit{$1$-complex} and will use the notation $G$ to refer to this embedded object. 

\subsection{Target geometry} Let $\Sigma_X$ be the cone complex of $(X|D)$. An embedded $1$-complex $G\subset \Sigma_X$ gives rise to a class of target geometries as follows. Place the embedded $1$-complex $G$ at height $1$ inside $\Sigma_X\times\RR_{\geq 0}$ and let $\mathsf C(G)$ be the cone over it. 

\begin{lemma}
The cone $\mathsf C(G)$ over $G$ is a cone complex embedded in $\Sigma_X\times\RR_{\geq 0}$
\end{lemma}

\begin{proof}
The cone over each face in $G$ certainly forms a cone, so we check that this collection of cones meet along faces of each. Following the argument in~\cite[Theorem~3.4]{BGS11} this check is nontrivial in the fiber over $0$ in $\Sigma_X\times\RR_{\geq 0}$. Since $G$ has dimension either $0$ or $1$, the cones in the $0$ fiber are either rays starting from the origin or the origin itself. Since two such rays either coincide or intersect only at the origin, this implies the statement. 
\end{proof}

By applying the construction in Section~\ref{sec: subdivisions} and Remark~\ref{rem: divisorial} to the subdivision
\[
\mathsf C(G)\hookrightarrow \Sigma_X\times\RR_{\geq 0},
\]
we obtain a target expansion
\[
\mathcal Y_G\to X.
\]
It is typically not proper. 

\begin{remark}
The construction of such target expansions using subdivisions goes back at least to Mumford's work on degenerations of abelian varieties~\cite{Mum72}. Its first appearance in enumerative geometry is in work of Nishinou and Siebert~\cite{NS06}. 
\end{remark}

Every affine toric variety carries a canonical logarithmic structure. Consider a toric monoid $P$ with associated toric variety $U_P$ and a closed point $u$ in the closed torus orbit. By pulling back the toric logarithmic structure to $u$, we obtain the \textit{$P$-logarithmic point}. It is denoted $\spec \CC_P$. The \textit{standard logarithmic point} is $\spec \CC_\NN$ and is the pullback to $0$ of the toric logarithmic structure on $\A^1$.

\begin{definition}[{\it Expansions and families}]
A \textit{rough expansion} of $X$ over $\spec \CC_\NN$ is a logarithmic scheme $\mathcal Y\to X\times \spec \CC_\NN$ which is the fiber over $0$ in the modification of $(X|D)\times\A^1$ induced by a subdivision
\[
\Delta\hookrightarrow \Sigma_X\times\RR_{\geq 0}.
\]
A rough expansion is called an \textit{expansion} if, in addition, the following two conditions are satisfied 
\begin{enumerate}[(E1)]
\item The scheme $\mathcal Y$ is reduced.
\item The subdivision is given by the cone over an embedded $1$-complex in $\Sigma_X$. Equivalently, at every point of $\mathcal Y$, the relative characteristic monoid of $\mathcal Y\to \spec \CC_\NN$ is free of rank at most $1$.
\end{enumerate}
If $S$ is a fine and saturated logarithmic scheme, a morphism $\mathcal Y/S\to X\times S/S$ is called an {\it expansion of $X$ over $S$} if all pullbacks $\spec \CC_\NN\to S$ are expansions and if the morphism $\mathcal{Y} \to S$ is flat and logarithmically smooth.
\end{definition}

We make a few comments about the definition. The first concerns (E2). We deal exclusively with subschemes of dimension at most $1$ in this paper, and the transversality condition that we will impose on these subschemes will ensure that the meet the strata of their ambient spaces in a dimensionally transverse fashion. Once a rough expansion is constructed, the codimension $2$ strata become irrelevant, and by removing these, one is led to more efficient moduli spaces. The second concerns the definition of an expansion over general bases $S$. For experts in logarithmic geometry, we note that one can replace the condition ``flat and logarithmically smooth'' with ``logarithmically smooth and integral''. Similarly, the reducedness of the fibers can be replaced with ``saturated''. Finally, we note that although the expansions themselves tend to be non-proper, we always have a map from $\mathcal Y$ to the original target $X$, which has been assumed to be projective throughout. Moreover, we will only ever consider sheaves on $\mathcal Y$ with proper support, so for example, the support always defines a cycle in $X$ by pushforward. 

A $1$-complex and the corresponding geometric expansion are shown in Figure~\ref{fig: a-1-complex}. 

\begin{figure}
\tikzset{every picture/.style={line width=0.75pt}} 

\tikzset{every picture/.style={line width=0.75pt}} 

\tikzset{every picture/.style={line width=0.75pt}} 

\begin{tikzpicture}[x=0.75pt,y=0.75pt,yscale=-1,xscale=1]

\draw  [dash pattern={on 3.75pt off 3pt on 7.5pt off 1.5pt}]  (83,207) -- (83.49,60) ;
\draw [shift={(83.5,58)}, rotate = 90.19] [color={rgb, 255:red, 0; green, 0; blue, 0 }  ][line width=0.75]    (10.93,-3.29) .. controls (6.95,-1.4) and (3.31,-0.3) .. (0,0) .. controls (3.31,0.3) and (6.95,1.4) .. (10.93,3.29)   ;
\draw  [dash pattern={on 3.75pt off 3pt on 7.5pt off 1.5pt}]  (83,207) -- (240.5,206.01) ;
\draw [shift={(242.5,206)}, rotate = 179.64] [color={rgb, 255:red, 0; green, 0; blue, 0 }  ][line width=0.75]    (10.93,-3.29) .. controls (6.95,-1.4) and (3.31,-0.3) .. (0,0) .. controls (3.31,0.3) and (6.95,1.4) .. (10.93,3.29)   ;
\draw    (133.5,165) -- (83,207) ;
\draw    (133.5,165) -- (134.48,62) ;
\draw [shift={(134.5,60)}, rotate = 90.55] [color={rgb, 255:red, 0; green, 0; blue, 0 }  ][line width=0.75]    (10.93,-3.29) .. controls (6.95,-1.4) and (3.31,-0.3) .. (0,0) .. controls (3.31,0.3) and (6.95,1.4) .. (10.93,3.29)   ;
\draw    (133.5,165) -- (220.5,163.04) ;
\draw [shift={(222.5,163)}, rotate = 178.71] [color={rgb, 255:red, 0; green, 0; blue, 0 }  ][line width=0.75]    (10.93,-3.29) .. controls (6.95,-1.4) and (3.31,-0.3) .. (0,0) .. controls (3.31,0.3) and (6.95,1.4) .. (10.93,3.29)   ;
\draw    (416.11,106.98) .. controls (416.7,104.46) and (418.07,103.44) .. (420.22,103.91) .. controls (422.73,104.12) and (424.04,103.15) .. (424.16,101) .. controls (424.63,98.59) and (426.06,97.54) .. (428.46,97.87) .. controls (430.81,98.24) and (432.18,97.26) .. (432.57,94.95) .. controls (432.95,92.66) and (434.26,91.76) .. (436.5,92.25) .. controls (439,92.59) and (440.41,91.66) .. (440.74,89.46) .. controls (441.39,87.09) and (442.89,86.17) .. (445.24,86.7) .. controls (447.23,87.47) and (448.67,86.67) .. (449.56,84.28) .. controls (450.22,82.07) and (451.61,81.38) .. (453.73,82.21) .. controls (456,83.03) and (457.62,82.36) .. (458.59,80.19) .. controls (459.68,78.06) and (461.26,77.56) .. (463.34,78.69) .. controls (465.25,79.96) and (466.81,79.63) .. (468.04,77.72) .. controls (469.45,75.87) and (471.16,75.71) .. (473.15,77.26) .. controls (474.69,78.93) and (476.29,78.97) .. (477.94,77.4) .. controls (479.79,75.93) and (481.43,76.16) .. (482.86,78.09) .. controls (484.15,80.08) and (485.86,80.49) .. (487.98,79.33) .. controls (489.97,78.2) and (491.45,78.68) .. (492.43,80.79) .. controls (493.65,83.04) and (495.2,83.66) .. (497.09,82.64) .. controls (499.4,81.84) and (501.03,82.59) .. (501.99,84.89) .. controls (502.6,87.05) and (503.97,87.75) .. (506.11,86.98) -- (506.11,86.98) ;
\draw    (496.11,166.98) .. controls (496.35,164.56) and (497.66,163.42) .. (500.04,163.55) .. controls (502.38,163.72) and (503.59,162.65) .. (503.68,160.34) .. controls (503.98,157.83) and (505.23,156.7) .. (507.44,156.95) .. controls (509.87,156.97) and (511.12,155.79) .. (511.2,153.42) .. controls (511.17,151.11) and (512.29,150) .. (514.54,150.09) .. controls (516.97,149.94) and (518.12,148.68) .. (517.99,146.31) .. controls (517.84,143.86) and (518.86,142.55) .. (521.07,142.38) .. controls (523.4,141.86) and (524.3,140.37) .. (523.76,137.92) .. controls (522.85,135.9) and (523.45,134.39) .. (525.56,133.4) .. controls (527.6,132.06) and (527.91,130.48) .. (526.5,128.67) .. controls (524.93,126.88) and (524.96,125.18) .. (526.58,123.57) .. controls (528.07,121.66) and (527.88,120.04) .. (525.99,118.73) .. controls (523.98,117.23) and (523.62,115.59) .. (524.9,113.82) .. controls (526.01,111.59) and (525.52,109.94) .. (523.45,108.87) .. controls (521.34,107.81) and (520.79,106.17) .. (521.78,103.96) .. controls (522.83,101.99) and (522.23,100.39) .. (519.98,99.16) .. controls (517.84,98.27) and (517.23,96.73) .. (518.15,94.56) .. controls (519.03,92.31) and (518.44,90.87) .. (516.38,90.25) -- (516.38,90.25) ;
\draw    (486.11,166.98) .. controls (485.54,164.81) and (485.57,163.85) .. (486.2,164.1) .. controls (483.97,163.11) and (482.37,163.66) .. (481.4,165.75) .. controls (480.17,167.84) and (478.56,168.26) .. (476.59,167.01) .. controls (474.7,165.64) and (473.06,165.85) .. (471.67,167.66) .. controls (469.94,169.35) and (468.21,169.31) .. (466.49,167.52) .. controls (465.2,165.63) and (463.57,165.3) .. (461.58,166.52) .. controls (459.5,167.57) and (458.02,166.98) .. (457.13,164.77) .. controls (456.48,162.52) and (454.97,161.62) .. (452.6,162.05) .. controls (450.24,162.34) and (448.94,161.28) .. (448.71,158.88) .. controls (448.56,156.53) and (447.32,155.42) .. (444.99,155.57) .. controls (442.68,155.71) and (441.51,154.63) .. (441.49,152.33) .. controls (441.36,149.9) and (440.09,148.67) .. (437.69,148.66) .. controls (435.48,148.8) and (434.39,147.69) .. (434.43,145.34) .. controls (434.4,142.89) and (433.25,141.66) .. (430.99,141.65) .. controls (428.6,141.47) and (427.51,140.21) .. (427.71,137.88) .. controls (427.9,135.48) and (426.83,134.14) .. (424.5,133.85) .. controls (422.2,133.54) and (421.26,132.22) .. (421.67,129.87) .. controls (422.18,127.56) and (421.32,126.15) .. (419.11,125.63) .. controls (416.82,124.78) and (416.13,123.25) .. (417.02,121.02) -- (416.11,116.98) ;
\draw    (493.84,180.73) .. controls (361,316) and (317,148) .. (403.3,114.18) ;
\draw   (405.84,103.7) .. controls (409.38,101.71) and (413.98,103.18) .. (416.11,106.98) .. controls (418.25,110.78) and (417.11,115.47) .. (413.57,117.45) .. controls (410.03,119.44) and (405.43,117.97) .. (403.3,114.18) .. controls (401.17,110.38) and (402.31,105.69) .. (405.84,103.7) -- cycle ;
\draw   (486.11,166.98) .. controls (489.65,164.99) and (494.25,166.46) .. (496.38,170.25) .. controls (498.52,174.05) and (497.38,178.74) .. (493.84,180.73) .. controls (490.3,182.72) and (485.7,181.25) .. (483.57,177.45) .. controls (481.44,173.66) and (482.57,168.97) .. (486.11,166.98) -- cycle ;
\draw   (508.66,76.5) .. controls (512.19,74.52) and (516.79,75.98) .. (518.92,79.78) .. controls (521.06,83.58) and (519.92,88.27) .. (516.38,90.25) .. controls (512.84,92.24) and (508.25,90.78) .. (506.11,86.98) .. controls (503.98,83.18) and (505.12,78.49) .. (508.66,76.5) -- cycle ;
\draw  [color={rgb, 255:red, 208; green, 2; blue, 27 }  ,draw opacity=1 ][line width=1.5] [line join = round][line cap = round] (439,72.5) .. controls (454.93,88.43) and (436.73,108.69) .. (432,120.5) .. controls (430.79,123.52) and (423.94,132.44) .. (426,134.5) .. controls (426.46,134.96) and (430.3,135.73) .. (431,135.5) .. controls (432.8,134.9) and (436.82,129.68) .. (438,128.5) .. controls (442.59,123.91) and (449.97,121.91) .. (456,119.5) .. controls (459.24,118.2) and (467.14,114.82) .. (469,121.5) .. controls (471.98,132.24) and (466.62,141.56) .. (462,148.5) .. controls (459.2,152.7) and (456.2,155.89) .. (455,159.5) .. controls (454.68,160.45) and (454.06,162.15) .. (455,162.5) .. controls (460.45,164.54) and (465.9,160.92) .. (470,157.5) .. controls (474.84,153.47) and (481.51,144.25) .. (489,143.5) .. controls (497.21,142.68) and (518,147.67) .. (518,154.5) ;
\draw  [color={rgb, 255:red, 208; green, 2; blue, 27 }  ,draw opacity=1 ][line width=1.5] [line join = round][line cap = round] (388,158.5) .. controls (388,156.1) and (394.26,152.81) .. (396,150.5) .. controls (400.7,144.23) and (412.52,139.74) .. (419,136.5) .. controls (419.38,136.31) and (426.38,133.47) .. (427,134.5) .. controls (429.86,139.27) and (424.89,140.83) .. (424,143.5) .. controls (421,152.49) and (415.79,168.29) .. (423,175.5) .. controls (434.55,187.05) and (445.57,171.93) .. (450,167.5) .. controls (452.02,165.48) and (452.53,165.43) .. (454,162.5) .. controls (454.33,161.83) and (455.47,160.97) .. (456,161.5) .. controls (462.67,168.17) and (454.19,182.43) .. (456,191.5) .. controls (456.75,195.24) and (463,197.18) .. (463,197.5) ;
\draw    (372,198) .. controls (370.73,196.07) and (371.04,194.42) .. (372.94,193.07) .. controls (374.77,191.51) and (374.91,189.8) .. (373.36,187.95) .. controls (371.74,186.18) and (371.77,184.54) .. (373.44,183.05) .. controls (375.09,181.23) and (375.07,179.55) .. (373.38,178.02) .. controls (371.69,176.19) and (371.68,174.48) .. (373.33,172.87) .. controls (375.01,171.29) and (375.04,169.67) .. (373.42,168.02) .. controls (371.85,166.27) and (371.96,164.64) .. (373.76,163.11) .. controls (375.61,161.68) and (375.85,160.02) .. (374.47,158.12) .. controls (373.2,156.07) and (373.59,154.39) .. (375.65,153.09) .. controls (377.76,151.96) and (378.3,150.41) .. (377.27,148.45) .. controls (376.52,146.03) and (377.25,144.47) .. (379.47,143.77) .. controls (381.71,143.22) and (382.58,141.8) .. (382.07,139.51) .. controls (381.76,137.08) and (382.72,135.8) .. (384.94,135.67) .. controls (387.41,135.37) and (388.55,134.09) .. (388.38,131.84) .. controls (388.37,129.49) and (389.56,128.36) .. (391.95,128.43) .. controls (394.33,128.58) and (395.69,127.44) .. (396.03,125.03) .. controls (396.15,122.84) and (397.49,121.85) .. (400.05,122.06) .. controls (402.22,122.6) and (403.49,121.76) .. (403.86,119.53) -- (408,117) ;
\draw    (391,213) .. controls (391.93,210.71) and (393.51,210.1) .. (395.75,211.19) .. controls (397.66,212.6) and (399.38,212.45) .. (400.91,210.74) .. controls (402.5,209.13) and (404.17,209.23) .. (405.9,211.04) .. controls (407.31,212.88) and (408.97,213.09) .. (410.86,211.66) .. controls (412.84,210.27) and (414.45,210.5) .. (415.69,212.37) .. controls (417.29,214.28) and (418.98,214.51) .. (420.75,213.06) .. controls (422.5,211.58) and (424.1,211.74) .. (425.53,213.55) .. controls (427.07,215.32) and (428.71,215.41) .. (430.45,213.8) .. controls (432.08,212.13) and (433.76,212.09) .. (435.49,213.69) .. controls (437.38,215.2) and (439.09,215.02) .. (440.63,213.13) .. controls (442.01,211.18) and (443.57,210.85) .. (445.32,212.14) .. controls (447.59,213.21) and (449.17,212.71) .. (450.06,210.62) .. controls (451.13,208.37) and (452.72,207.66) .. (454.82,208.49) .. controls (457.07,209.15) and (458.49,208.32) .. (459.07,206.01) .. controls (459.5,203.68) and (460.92,202.65) .. (463.31,202.92) .. controls (465.82,202.98) and (467.06,201.9) .. (467.01,199.69) .. controls (466.88,197.48) and (468.1,196.22) .. (470.69,195.91) .. controls (473.01,195.8) and (474.06,194.57) .. (473.83,192.22) .. controls (473.54,189.87) and (474.57,188.49) .. (476.94,188.08) .. controls (479.35,187.51) and (480.21,186.23) .. (479.51,184.26) .. controls (479.1,181.77) and (479.95,180.38) .. (482.05,180.11) -- (483.57,177.45) ;

\draw (255,212.4) node [anchor=north west][inner sep=0.75pt]    {$\Sigma $};
\draw (141,142.4) node [anchor=north west][inner sep=0.75pt]    {$v_{1}$};
\draw (71,212.4) node [anchor=north west][inner sep=0.75pt]    {$v_{0}$};
\draw (489,203.4) node [anchor=north west][inner sep=0.75pt]    {$Y_{0}$};
\draw (531,130.4) node [anchor=north west][inner sep=0.75pt]    {$Y_{1} = \mathbb P^2\setminus \textrm{$\{$3 points$\}$}$};

\end{tikzpicture}
\vspace{-0.5in}
\caption{On the left we sketch a $1$-complex in the cone complex $\mathbb R_{\geq 0}^2$. The cone complex $\Sigma_X$ is $\RR_{\geq 0}^2$ drawn with a dashed arrow while the $1$-complex is solid. On the right is a cartoon of the corresponding expansion. It is obtained by performing deformation to the normal cone of a codimension $2$ stratum and then passing to an open. The wavy lines indicate the divisors where the logarithmic structure is nontrivial. The holes indicate codimension $2$ strata that have been removed from the expansion. The red curve in the middle is a subscheme, of the type that we will soon introduce.}\label{fig: a-1-complex}
\end{figure}

\begin{remark}

The definition has a peculiar feature when we consider $X$ itself. Specifically, $X$ is a rough expansion of itself, but it is only an expansion of itself when it contains no codimension $2$ strata. The complement of the codimension $2$ strata in $X$ is an expansion of $X$, as is the interior of $X$.  The reason for imposing this condition is that we will later require our families of subschemes to have nonempty intersection with all logarithmic strata, which is not satisfied for a $1$-dimensional transverse subscheme unless there are no codimension $2$ strata.  This definition yields a cleaner universal property when we consider the valuative criterion of properness.

\end{remark}

\subsection{Flat limit algorithms: existence} We let $X^\circ$ be the complement of $D$ in $X$. Let $C$ be a smooth curve with a distinguished point $0$, whose complement is denoted $C^\circ$.

\begin{proposition}\label{prop: dimensional-transversality}
Consider a flat family of subschemes $\mathcal Z^\circ\subset X^\circ\times C^\circ$ over $C^\circ$. There exists a rough expansion $\mathcal Y\to X\times C$ over $C$ such that the closure $\mathcal Z$ of $\mathcal Z^\circ$ in $\mathcal Y$ has the following properties:
\begin{enumerate}[(D1)]
\item The scheme $\mathcal Z$ is proper and flat over $C$.
\item The scheme $\mathcal Z$ has non-empty intersection with each stratum of $\mathcal Y$. 
\item The scheme $\mathcal Z$ intersects all the strata of $\mathcal Y$ in the expected dimension. 
\end{enumerate}
Moreover, after replacing $C$ with a ramified base change, the degeneration of $X$ can be guaranteed to have reduced special fiber, i.e. the rough expansion can be chosen to be an expansion. 
\end{proposition}
We will refer to a family satisfying these properties as as \textit{dimensionally transverse} family. A sketch picture of dimensional transversality can be seen on the right in Figure~\ref{fig: a-1-complex}. Notice in particular that in the main component $Y_0$ of the expansion shown there, the subscheme does not intersect the logarithmic divisors. Correspondingly, the $1$-complex on the left of the image does not include the dashed axes. 

\begin{proof}
Let 
\[
\spec \CC[\![t]\!] \to C
\]
be the spectrum of the completed local ring of $C$ at $0$. Pull back the subscheme family $\mathcal Z^\circ$ to the generic point of this valuation ring to obtain a subscheme $\mathcal Z^\circ_\eta$ of $X^\circ$ over $\CC(\!(t)\!)$ . After passing to a valued field extension with value group $\RR$, apply the tropicalization map to obtain
\[
\trop(\mathcal Z^\circ)\subset \Sigma_X.
\]
By the structure result for tropicalizations recorded in Theorem~\ref{thm: bieri-groves}, this is the support of a polyhedral complex of dimension at most $1$ in $\Sigma$. Choose a polyhedral structure on this set and call it $G$. In practice, this amounts to expressing the tropicalization as a union of vertices, closed bounded edges, and unbounded rays. This results in a polyhedral complex embedded in $\Sigma_X$. Let $\mathsf C({G})$ be the cone over $G$ in $\Sigma_X\times\RR_{\geq 0}$. We obtain an associated target family
\[
\mathcal Y_G\to C. 
\]
By construction, the tropicalization of $\mathcal Z^\circ$ when viewed as a single subscheme of the total space $X^\circ\times C^\circ$ is equal to $\mathsf C({G})$. The theorems in Section~\ref{sec: trop-via-val} imply that the closure is proper. Since the base is a smooth curve, the closure is also flat over $C$. 

We examine the strata intersections. Given a stratum $W$ of $\mathcal Y_G$ dual to a cone $\sigma$ in $\mathsf{C}(G)$, choose a point $v\in \sigma^\circ$. By definition, this is the image of a $K$-valued point of $\mathcal Z^\circ$ under the tropicalization map. Since the family of subschemes is proper this extends to an $R$-valued point where $R$ is the valuation ring of $K$.  Since the tropicalization of this $K$-valued point lies in the interior of $\sigma$, the closed point maps to $W$, so it follows that $\mathcal Z$ intersects all strata. 

Finally, we show that we can obtain a degeneration with reduced special fiber after a base change. By using the toric dictionary, we observe that the special fiber of $\mathcal Y_G\to C$ is reduced if and only if the vertices of $G$ are lattice points in $\Sigma_X$. This can be engineered by using Kawamata's cyclic covering trick, explained in~\cite{AK00}. The toroidal procedure is carried out by replacing the integral lattice in $\Sigma_C$ with a finite index sublattice, thereby ensuring that the fiber over the new primitive generator of $\Sigma_C$ is a dilation of $G$ which has integral vertices. The main result of~\cite[Section~5]{AK00} is that this produces the requisite ramified base change. The statements concerning strata intersections are unaffected by the base change, and we conclude the result. 
\end{proof}

\subsection{Flat limit algorithms: uniqueness} The tropical limit algorithm in the previous section inherits a uniqueness property that should be thought of as close to a universal closedness result for the forthcoming moduli problem, where by ``close to'' we mean that we will need to eventually strengthen the notion of transversality. 

Let $K$ be the discretely valued field associated to the valuation at $0$ in $C$ and let $R$ be its valuation ring. 

\begin{proposition}\label{prop: uniqueness-dimensional-case}
Let $\mathcal Z_\eta^\circ$ be a flat family of subschemes of $X^\circ$ over $\spec K$. Assume that the closure of $\mathcal Z_\eta^\circ$ in $X\times \spec K$ is dimensionally transverse. Then there exists a canonical triple $(R',\mathcal Y',\mathcal Z')$ comprising of a ramified base change $R\subset R'$ with fraction field $K'$ and expansion of $X$ 
\[
\mathcal Y'\to \spec R'
\] 
such that the closure $\mathcal Z'$ of $\mathcal Z_\eta^\circ\otimes_K K'$ in $\mathcal Y$ is dimensionally transverse. Moreover, the triple satisfies the following uniqueness property: \\

\noindent
$(\star)$ \ \ \ For any other choice $(R'',\mathcal Y'',\mathcal Z'')$ satisfying these requirements, there exists a unique toroidal birational morphism
\[
\mathcal Y''\to \mathcal Y''_1
\]
over $\spec R''$ with subscheme $\mathcal Z''_1$, and a ramified covering $\spec R''\to \spec R$, such that $\mathcal Z''_1\subset \mathcal Y''_1$ is obtained from $\mathcal Z'\subset \mathcal Y'$ by base change. 
\end{proposition}

We deduce the uniqueness results from two simple combinatorial observations. 

\begin{lemma}
Let $\iota: G\hookrightarrow \Sigma_X$ be an embedded $1$-complex. The underlying set $|G|$ of $G$ carries a unique minimal polyhedral structure $\overline G$ such that the map $\iota$ descends to an embedding $\overline G\hookrightarrow \Sigma_X$. 
\end{lemma}

The term minimal here is used in the sense that any other polyhedral structure is obtained from the putatutive unique minimal one by subdividing along edges. Note that a morphism of polyhedral complexes is required to map vertices and edges of $G$ to into cones of $\Sigma_X$.

\begin{proof}
Recall that $G$ comes with a distinguished vertex set and $|G|$ has the structure of a metric space independent of the chosen vertex set. The non-bivalent vertices of $G$ are necessarily contained in the vertex set of any polyhedral structure on $|G|$. Let $x\in G$ be a $2$-valent vertex. Call $v$ \textit{inessential} if both of the following conditions hold.\footnote{Intuitively, $x$ is inessential if the map $\iota$ doesn't \textit{bend} at $x$; thus, the map $\iota$ descends to a polyhedral map after coarsening the polyhedral structure to exclude $x$ from the vertex set.}
\begin{enumerate}[(1)]
\item There exists an open neighborhood $U_x$ of $x$ that is completely contained in the relative interior $\sigma^\circ$ of a cone $\sigma\in\Sigma_X$.
\item Every point in this neighborhood lies on the same line in the vector space $\sigma^{\mathsf{gp}}$.
\end{enumerate}
A vertex that is not inessential is essential. Let $\overline G$ be the polyhedral complex obtained from the metric space $|G|$ by declaring the vertices of $\overline G$ to be the essential vertices of $G$. The morphism $\iota$ descends to an embedding $\overline G\to \Sigma_X$. Conversely, any polyhedral structure must contain the points of $|G|$ that are not inessential to ensure that $\iota$ is a morphism of polyhedral complexes. It follows that $\iota:\overline G\to \Sigma_X$ is minimal. 
\end{proof}

Passing to the cone over $\overline G$, we obtain a uniqueness property concerning the dilations of $\overline G$ obtained by uniformly scaling the edge lengths. 

\begin{lemma}
Let $G\hookrightarrow \Sigma_X$ be an embedded $1$-complex. Let $\mathsf C(\overline G)\hookrightarrow \Sigma_X\times\RR_{\geq 0}$ be the cone over the minimal polyhedral structure $\overline G$ of $G$. There exists a minimum positive integer $b$ such that all vertices in the fiber of $\mathsf C(\overline G)$ over $b$ lie in the lattice of $\Sigma_X$. 
\end{lemma}

\begin{proof}
We take $b$ to be the least common multiple of the denominators appearing in the coordinates of all vertices of $G$, with respect to the coordinates on $\Sigma_X$. 
\end{proof}

\subsubsection*{Proof of Proposition.} The uniqueness is a translation of the combinatorial observations above, as we now explain. Let $\mathcal Z_\eta^\circ$ be a subscheme as in the proposition. The algorithm for finding limits has the following steps. We calculate the tropicalization of $\mathcal Z_\eta^\circ$ and obtain a canonical $1$-complex $\overline G\hookrightarrow \Sigma_X$. We then pass to the cone over this complex $\mathsf C(\overline G)$ in $\Sigma_X\times \RR_{\geq 0}$. Finally, we find the minimum positive integer $b$ in $\RR_{\geq 0}$ over which the fiber in $\mathsf C(\overline G)$ has integral vertices. We perform the order $b$ cyclic ramified base change to obtain the requisite triple $(R',\mathcal Y',\mathcal Z')$. 

Now consider another set $(R'',\mathcal Y'',\mathcal Z'')$ satisfying these properties. Let $\mathsf C(G'')\to \RR_{\geq 0}$ be the fan associated to this expansion. Tropicalization is invariant under taking valued field extensions; since the subscheme $\mathcal Z''$ nontrivially intersects all the strata of the expansion, we can conclude using  Theorem~\ref{thm: divisorial-trop} that the support of the tropicalization is equal to $\mathsf C(G'')$ where $G''$ is the height $1$ slice of the cone. If we replace $G''$ with its minimal polyhedral structure $\overline G''$, we obtain a refinement of cone complexes $\mathsf C(G'')\to \mathsf C(\overline G'')$. The latter cone complex gives rise to an expansion of $X$ in which the closure of $\mathcal Z_\eta^\circ$ is still dimensionally transverse. 

For the base change, by using the preceding lemma we let $\overline G'$ be the minimal dilation of $\overline G$ whose vertices are all integral. The cones over $\overline G'$ and $\overline G''$ coincide, so $\overline G''$ must be a dilation of $\overline G'$. It follows that the morphism 
\[
\mathsf C(\overline G'')\to \RR_{\geq 0}
\]
is obtained from 
\[
\mathsf C(\overline G')\to \RR_{\geq 0}
\]
by passing to a finite index sublattice in the integral structure of the base. The result follows.
\qed

Before proceeding to algebraic transversality, we offer two examples of procedure above. We begin with a $0$-dimensional example.

\begin{example}
Let $X$ be the $\A^2$ equipped with its toric logarithmic structure. Consider the subscheme $\mathcal Z$ given by the following union of two reduced points defined over the field $\CC(\!(t)\!)$:
\[
\mathcal Z = \{(t^2,t^3)\} \cup \{(t^4,t^5)\}\subset \A^2.
\]
The flat limit is a non-reduced point supported at the origin in $\A^2$. The tropicalization of $\mathcal Z$ consists of two points
\[
\trop(\mathcal Z) = \{(2,3)\}\cup \{(4,5)\} \subset \RR^2_{\geq 0}.
\]
According to the algorithm above, associated expansion is obtained by performing a toric modification of $\A^2\times \A^1$, adding rays passing through points $(2,3,1)$ and $(4,5,1)$ in the fan $\RR^2_{\geq 0}\times \RR_{\geq 0}$. The resulting expansion consists of two components, both isomorphic to $\mathbb G_m^2$. The new flat limit in the expanded family is the union of two reduced points, with one in each of these two components. 
\end{example}

We now consider a $1$-dimensional example. 

\begin{example}
Let $X$ be $\PP^2$ equipped with its toric logarithmic structure and homogeneous coordinates $X$, $Y$, and $Z$. Consider the subscheme $\mathcal Z$ given by a line defined over the field $\CC(\!(t)\!)$:
\[
\mathcal Z = \mathbb V(t^{-3}X+t^{-2}Y+Z) \subset \mathbb P^2
\]
The tropicalization is easily computed by hand as the subset of $\RR^2$ obtained by the union of $3$ rays $\rho_1,\rho_2,\rho_3$ at the point $(-3,-2)$.  The ray $\rho_1$ is parallel to the negative $x$-axis, the ray $\rho_2$ is parallel to the negative $y$-axis, and $\rho_3$ is parallel to the first-quadrant diagonal. It is the translate of the $1$-skeleton of the fan of $\PP^2$ to the point $(-2,-3)$. The associated $1$-complex $G$ has $2$-vertices: one vertex at the point $(-3,-2)$ which is trivalent and we call $V_1$, and another at the point $(-1,0)$ where the set $|G|$ intersects the $1$-skeleton of the fan of $\PP^2$, which we call $V_2$. See Figure~\ref{fig: example-tropical-line}

\begin{figure}

\tikzset{every picture/.style={line width=0.75pt}} 

\begin{tikzpicture}[x=0.75pt,y=0.75pt,yscale=-1,xscale=1]
\draw [ball color=black] (190,180) circle (0.5mm);
\draw [ball color=black] (220,180) circle (0.5mm);
\draw  [dash pattern={on 3.75pt off 3pt on 7.5pt off 1.5pt}]  (250,160) -- (250,52) ;
\draw [shift={(250,50)}, rotate = 90] [color={rgb, 255:red, 0; green, 0; blue, 0 }  ][line width=0.75]    (10.93,-3.29) .. controls (6.95,-1.4) and (3.31,-0.3) .. (0,0) .. controls (3.31,0.3) and (6.95,1.4) .. (10.93,3.29)   ;
\draw  [dash pattern={on 3.75pt off 3pt on 7.5pt off 1.5pt}]  (250,160) -- (348,160) ;
\draw [shift={(350,160)}, rotate = 180] [color={rgb, 255:red, 0; green, 0; blue, 0 }  ][line width=0.75]    (10.93,-3.29) .. controls (6.95,-1.4) and (3.31,-0.3) .. (0,0) .. controls (3.31,0.3) and (6.95,1.4) .. (10.93,3.29)   ;
\draw  [dash pattern={on 3.75pt off 3pt on 7.5pt off 1.5pt}]  (250,160) -- (161.66,218.89) ;
\draw [shift={(160,220)}, rotate = 326.31] [color={rgb, 255:red, 0; green, 0; blue, 0 }  ][line width=0.75]    (10.93,-3.29) .. controls (6.95,-1.4) and (3.31,-0.3) .. (0,0) .. controls (3.31,0.3) and (6.95,1.4) .. (10.93,3.29)   ;
\draw    (190,180) -- (190,72) ;
\draw [shift={(190,70)}, rotate = 90] [color={rgb, 255:red, 0; green, 0; blue, 0 }  ][line width=0.75]    (10.93,-3.29) .. controls (6.95,-1.4) and (3.31,-0.3) .. (0,0) .. controls (3.31,0.3) and (6.95,1.4) .. (10.93,3.29)   ;
\draw    (190,180) -- (288,180) ;
\draw [shift={(290,180)}, rotate = 180] [color={rgb, 255:red, 0; green, 0; blue, 0 }  ][line width=0.75]    (10.93,-3.29) .. controls (6.95,-1.4) and (3.31,-0.3) .. (0,0) .. controls (3.31,0.3) and (6.95,1.4) .. (10.93,3.29)   ;
\draw    (190,180) -- (101.66,238.89) ;
\draw [shift={(100,240)}, rotate = 326.31] [color={rgb, 255:red, 0; green, 0; blue, 0 }  ][line width=0.75]    (10.93,-3.29) .. controls (6.95,-1.4) and (3.31,-0.3) .. (0,0) .. controls (3.31,0.3) and (6.95,1.4) .. (10.93,3.29)   ;

\draw (151,160.4) node [anchor=north west][inner sep=0.75pt]    {$V_{1}$};
\draw (227,190.4) node [anchor=north west][inner sep=0.75pt]    {$V_{2}$};
\draw (155,110.4) node [anchor=north west][inner sep=0.75pt]    {$G$};
\end{tikzpicture}
\caption{The $1$-complex depicted in solid arrows with the fan of $\mathbb P^2$ in dashed arrows.}\label{fig: example-tropical-line}
\end{figure}

The associated expansion consists of two components corresponding to the two vertices above. The component $X_1$ corresponding to $V_1$ is a copy of $\mathbb P^2$ minus its three fixed points. The component $X_2$ corresponding to $V_2$ is a copy of $\mathbb P^1\times\mathbb G_m$. A straightforward calculation shows that the limiting subscheme is a line in $X_1$ union a fiber in $X_2$. 
\end{example}

\subsection{Algebraic transversality from dimensional transversality}\label{sec: strong-transversality} We have constructed limits of transverse subschemes of $X$ inside expansions of $X$ along $D$ that are dimensionally transverse to the strata of the expansion. We require a stronger form of transversality. 

\begin{definition}[Algebraic transversality]
Let $\mathcal Y$ be an expansion of $X$ and let $\mathcal Z\subset \mathcal Y$ be a subscheme with ideal sheaf $\mathcal I_{\mathcal Z}$. Then $\mathcal Z$ is said to be \textit{algebraically transverse} if it is dimensionally transverse and for every closed divisor stratum $\mathcal S\subset \mathcal Y$, the induced map
\[
\mathcal I_{\mathcal Z}\otimes_{\mathcal O_{\mathcal Y}} \mathcal O_{\mathcal S}\to \mathcal O_{\mathcal Y}\otimes_{\mathcal O_{\mathcal Y}} \mathcal O_{\mathcal S}
\]
is injective. 
\end{definition}

We call this {\it algebraic} transversality because it depends on more than just the dimension of the intersection with the strata, which is only set theoretic. On the other hand, it is weaker than genuine geometric transversality, because it allows the intersection of the subscheme with a divisor to be non-reduced.\footnote{In an earlier version of this paper, which was in circulation from 2020 until 2023, we had used the terminology ``strongly transverse'' instead of ``algebraically transverse''. We thank the referee, as well as numerous friends and colleagues of their objections to this. }

Let us give a few different ways to think about this condition. We maintain the notation above. 

\begin{proposition}
The subscheme $\mathcal Z\subset\mathcal Y$ is algebraically transverse if and only all of the following conditions hold:
\begin{enumerate}[(i)]
\item $\mathcal Z$ intersects every stratum of $\mathcal Y$, 
 \item there are no embedded points or components of $\mathcal Z$ that are contained in the codimension $1$ strata of $\mathcal Y$, and 
 \item if $\mathcal S$ is any double divisor contained in irreducible components $Y_i$ and $Y_j$ of $\mathcal Y$, the restrictions of the subscheme to $Y_i$ and $Y_j$ have the same intersection with $S$. 
\end{enumerate}
\end{proposition}

\begin{proof}
Assume $\mathcal Z\subset\mathcal Y$ is algebraically transverse. Since this implies dimensional transversality by definition, it is clear that (i) holds, see Proposition~\ref{prop: dimensional-transversality}. Now, the injectivity condition in the definition of algebraic transversality can be checked formally local on $\mathcal Y$. If we look in a neighborhood $U$ of a double divisor $\mathcal S$, then $\mathcal Y$ can be identified with $\mathbb A^{\mathsf{dim} X-1}\times \mathcal N$, where $\mathcal N$ is a nodal curve, i.e. given by $\{uv = 0\}$. The stated injectivity condition is equivalent to flatness of $\mathcal Z_U\to \mathcal N$. The condition (ii) now follows. Finally, condition (iii) follows by pulling back the flat morphism $\mathcal Z_U\to \mathcal N$ to each of the two branches $\{u=0\}$ and $\{v=0\}$ of the node, and then to the node $\{u = v = 0\}$ itself. Since the result is independent of which branch we restrict to, (iii) follows. 

Conversely, suppose the conditions (i)--(iii) hold. We can again restrict to a formal neighborhood $U = \mathbb A^{\mathsf{dim} X-1}\times \mathcal N$, and examine the resulting subscheme to $\mathcal Z_U$. We claim that under these hypotheses, the map $\mathcal Z_U\to \mathcal N$ is flat. Since $\mathcal N$ is reduced, we can deduce this from the valuative criterion for flatness~\cite[Corollaire 4.2.10]{RG71}, i.e. we can check flatness after pulling back to the spectrum $B$ of a DVR. Working over the base $B$, flatness of $\mathcal Z_B\to B$ is equivalent to the condition that the total space is the closure of the generic fiber. This is clearly implied by (i)--(iii)
\end{proof} 

\begin{remark}
The proposition allows one to check algebraic transversality in practice. Given a $1$-dimensional subscheme $\mathcal Z\hookrightarrow \mathcal Y$ in an expansion of $X$, our definition of dimensional transversality already includes condition (i). Therefore, in order to guarantee algebraic transversality, one needs to simply check that there are no embedded points on the divisorial strata, and that the induced subscheme on the double divisor from of its branches is the same. 
\end{remark}

A more conceptual characterization comes from the proof of the proposition above. Recall that $\Sigma_X$ is a union of faces in the standard orthant fan $\RR_{\geq 0}^d$. Let $\mathsf A_\Sigma$ be the associated union of orbits in the Artin stack $[\A^d/\mathbb G_m^d]$. The expansion $\mathcal Y\to X$ is pulled back from an associated expansion $\mathsf A_{\mathcal Y}$. We stress that this is a gluing together of Artin fans along divisors.

\begin{proposition}
A subscheme $\mathcal Z\hookrightarrow \mathcal Y$ is algebraically transverse if and only if the composite map $\mathcal Z\to \mathsf A_{\mathcal Y}$ is both flat and surjective. 
\end{proposition}

\begin{proof}
We can pass to an open neighborhood in $\mathcal Y$, and as in the previous proposition, algebraic transversality becomes equivalent to (i) flatness of $\mathcal Z_U$ over the nodal curve in the local model $U = \mathbb A^{\mathsf{dim} X-1}\times \mathcal N$, plus (ii) the condition that all strata have non-empty intersection with $\mathcal Z$. The map from $U$ to the Artin fan $ \mathsf A_{\mathcal Y}$ is certainly flat and surjective, so the forward implication is clear. The converse follows immediately from the characterization of algebraic transversality in the  previous proposition.
\end{proof}

Via the proposition above, algebraic transversality is equivalent to the \textit{logarithmic flatness} of $\mathcal Z$, when equipped with the pullback logarithmic structure from the expansion, together with the condition all strata have non-empty intersection with $\mathcal Z$. The perspective coming from logarithmic flatness play a key role in Kennedy-Hunt's work on the general logarithmic Hilbert scheme~\cite{KH22}. 

Finally, we note that by basic properties of tensor products, algebraic transversality is equivalent to the condition that the higher Tor functors of $\mathcal O_Z$ with $\mathcal O_S$ vanish. This is how transversality is stated by Li--Wu, who refer to the condition as {\it normality to the divisor}, see~\cite{LiWu15}.

\begin{proposition}\label{thm: uniqueness-strong-case}
Let $\mathcal Z_\eta^\circ$ be a flat family of subschemes of $X^\circ$ over $\spec K$ whose closure in $X$ is algebraically transverse. Then there exists a canonical triple $(R',\mathcal Y',\mathcal Z')$ comprised of a ramified base change $R\subset R'$ with fraction field $K'$ and expansion of $X$ 
\[
\mathcal Y'\to \spec R'
\] 
such that the closure $\mathcal Z'$ of $\mathcal Z_\eta^\circ\otimes_K K'$ in $\mathcal Y$ is algebraically transverse, satisfying the following uniqueness property: \\

\noindent
$(\star)$ \ \ \ For any other choice $(R'',\mathcal Y'',\mathcal Z'')$ satisfying these requirements, there exists a unique toroidal birational morphism
\[
\mathcal Y''\to \mathcal Y''_1
\]
over $\spec R''$ with subscheme $\mathcal Z''_1$, and a ramified covering $\spec R''\to \spec R$, such that $\mathcal Z''_1\subset \mathcal Y''_1$ is obtained from $\mathcal Z'\subset \mathcal Y'$ by base change. 
\end{proposition}

\begin{proof}
The proposition is established in the following sequence of lemmas.
\end{proof}

\subsubsection{A proof via the work of Li--Wu}
Given the dimensional transversality statement that we have already established, one can deduce the proposition above formally from the results of Li--Wu. We give a sketch of this, in case the reader wants to simply skip the rest of this section without a loss of continuity. Using the arguments of the previous section on dimensional transversality,  given a family of algebraically transverse subschemes over a valued field $\spec K$, a limit can be found over a ramified base change $\spec R'$ of its valuation ring, where the intersections of all locally closed strata with the flat limit of the subscheme are either empty or have the expected dimension. By removing the codimension $2$ strata from the degeneration, we obtain a degeneration with only double points, such that the closure of the general fiber of subschemes forms a proper and flat family. In other words, we find a dimensionally transverse limit.

We can now appeal to the results of~\cite{LiWu15} concerning degenerating subschemes in double point degenerations. The results of Li--Wu are not explicitly stated for non-proper total spaces, but the properness is only used to establish {\it some} proper and flat limit of the generic fiber. This is already achieved by the first step. The methods of ~\cite{LiWu15} apply with minor changes to give the result above. We leave the details of this to an interested reader.

\subsubsection{A proof via Gr\"obner theory} We take a different approach, inspired by Gr\"obner theory considerations. The approach has a slightly more constructive nature. 

The strategy is to reduce to a theorem of Tevelev on subvarieties of toric varieties~\cite{Tev07}.  Specifically, if $X$ is a toric variety with dense torus $T$, Tevelev proves that for any subscheme $Z\hookrightarrow X$, there exists a toric blowup $X'\to X$, with $X'$ smooth, such that the strict transform $Z'\hookrightarrow X'$ is what he calls {\it tropical}~\cite[Theorem~1.2]{Tev07}. This means that the induced map $Z\to [X'/T]$ is flat. When this theorem is applied to a flat family of subschemes in $X$ of relative dimension $1$, over $\A^1$, the output is precisely the algebraically transverse family that we seek. 

We will reduce our more general situation to the toric one. The reduction is essentially straightforward; once dimensional transversality has been achieved, the additional birational modifications required to produce an algebraically transverse family can be analyzed affine locally. The details follow. 

We will use the following lemma in the course of our proof, which allows us to take an algebraically transverse limit and contract unnecessary blowups. 

\begin{lemma}[Contraction lemma]
Let $\mathcal Y\to \spec R$ be a expansion of $X$ whose special fiber includes two irreducible components $Y_1$ and $Y_2$ that meet transversely along a smooth divisor $D$. Let $\mathcal Y'\to\mathcal Y$ be the blowup at $D$ with exceptional $E$. Assume that the special fiber of $\mathcal Y'$ is reduced . 

Suppose $\mathcal Z'\hookrightarrow \mathcal Y'$ is a flat and proper family of algebraically transverse subschemes of dimension $1$, with generic fiber $\mathcal Z'_\eta$, and such that $\mathcal Z'\cap E$ is the pullback of a subscheme along $E\to D$.

Let $\mathcal Z$ be the closure of the generic fiber $\mathcal Z_\eta$ in $\mathcal Y$. Then $\mathcal Z$ is a flat family of algebraically transverse subschemes.
\end{lemma}

\begin{proof}
Throughout the proof, we add the subscript $0$ to indicate special fibers. Suppose $\mathcal Z$ were not algebraically transverse. By the characterization of algebraic transversality, this means that the special fiber $\mathcal Z_0\hookrightarrow\mathcal Y_0$ has embedded points contained in $D$. Since $\mathcal Z$ is the closure of its generic fiber, it is flat over $\spec R$. It follows that the special fibers of $\mathcal Z'_0$ and $\mathcal Z_0$ have the same holomorphic Euler characteristic. However, this now gives a numerical contradiction. Since $\mathcal Z'_0$ is algebraically transverse, its holomorphic Euler characteristic is given by
\[
\chi_{\mathsf{hol}}(\mathcal Z'_0) = \chi_{\mathsf{hol}}(\mathcal Z'_0\cap Y_1)+\chi_{\mathsf{hol}}(\mathcal Z'_0\cap Y_2)+\chi_{\mathsf{hol}}(\mathcal Z'_0\cap E)- \chi_{\mathsf{hol}}(\mathcal Z'_0\cap E\cap Y_1)-\chi_{\mathsf{hol}}(\mathcal Z'_0\cap E\cap Y_2).
\]
But we also have
\[
\chi_{\mathsf{hol}}(\mathcal Z'_0\cap E\cap Y_i) = \chi_{\mathsf{hol}}(\mathcal Z'_0\cap E),
\]
since $\mathcal Z'_0\cap E$ is a $\mathbb P^1$-bundle over $\mathcal Z'_0\cap E\cap Y_i$.

Now, the intersection of $\mathcal Z'$ with $(Y_1\cup Y_2)\setminus E$ maps isomorphically onto its image in the complement of $D$ in $\mathcal Y_0$. On the one hand, the closure of this image cannot contain the embedded points above, and is a closed subscheme, properly contained in $\mathcal Z_0$. On the other hand, the Euler characteristic of this closure is equal to that of $\mathcal Z'_0$ above. Since the embedded points contribute positively to the holomorphic Euler characteristic of $\mathcal Z_0$, we arrive at a contradiction.
\end{proof}

We move on to the main proof. We first identify the subdivision that is required to produce an algebraically transverse limit, and then prove that it has the properties required. We require the notion of an \textit{initial degeneration} from tropical geometry. In the notation of the proposition, let $\mathcal Z_\eta^\circ$ be a flat family of subschemes of $X^\circ$ over $\spec K$. Then we have
\[
\trop(\mathcal Z_\eta)\subset \Sigma_X
\]
which is the support of a $1$-dimensional polyhedral complex embedded in $\Sigma_X$. Given a rational point $w$ in $\Sigma_X$, we consider a non-proper degeneration of $X$ as follows. Consider the linear map
\[
\RR_{\geq 0}\to \Sigma_X\times\RR_{\geq 0}
\]
mapping isomorphically onto the ray joining the origin to $(w,1)$. This determines an associated flat degeneration 
\[
\mathcal Y_w\to \spec(R)
\]
whose general fiber is the interior $X^\circ$ of $X$, and whose special fiber is a torus torsor over a locally closed stratum of $X$, and precisely, the locally closed stratum corresponding to the cone that contains $w$ in its interior. We denote the special fiber by $X_w$. We note in particular, that up to isomorphism, the scheme $X_w$ is constant when $w$ varies in the interior of each cone of $\Sigma_X$. 

\begin{definition}
The \textit{initial degeneration} of $\mathcal Z_\eta$ at a point $w$ of $\Sigma_X$ is the special fiber of the closure of $\mathcal Z_\eta\cap X^\circ$ in $\mathcal Y_w$, viewed as a subscheme of $X_w$. It will be denoted $\mathsf{in}_w(\mathcal Z_\eta)$. 
\end{definition}

Recall from the previous section that $\trop(Z_\eta)$ has a minimal polyhedral structure, and this determines a degeneration $\mathcal Y_\Gamma$. Let $w$ be a point on an edge $e$ of this minimal polyhedral structure. Then $X_w$ is a $\mathbb G_m$-torsor over a codimension $1$ stratum $X_e$, in the special fiber of the associated degeneration. We will say that the initial degeneration $\mathsf{in}_w(\mathcal Z_\eta)$ is \textit{tubular} if it is the schematic preimage of a $0$-dimensional subscheme under the projection $X_w\to X_e$. 

The following basic finiteness theorem holds. 

\begin{proposition}
The subset of points $w$ in $\trop(\mathcal Z_\eta)$ with the property that $\mathsf{in}_w(\mathcal Z_\eta)$ is not tubular is finite. 
\end{proposition}

\begin{proof}
In the notation of the preceding paragraph, if we are given a finite set of points in the set $\trop(Z_\eta)$, there is a unique coarsest polyhedral structure on the tropicalization that includes these points as vertices. It of course refines the minimal polyhedral structure on $\trop(Z_\eta)$. We refer to the vertices of this new polyhedral structure that do not appear in the minimal one as {\it extra} vertices. They should be viewed as lying in the interior of an edge in the minimal structure. 

By the paragraph above and the dictionary explained in the background section, each such choice of finite set produces a model of $X$ over the DVR. We assume it has reduced special fiber by performing an appropriate base change. The closure of $Z_\eta$ in this model is a flat family of subschemes. 

Let $w$ be an extra vertex. The component it is dual to is a $\mathbb P^1$-bundle. Since the Hilbert polynomial of the subscheme is constant in the family, for all but finitely many choices of $w$ in $\trop(Z_\eta)$, the curve class of the subscheme in this component is a fiber class. Similarly, since the holomorphic Euler characteristic is fixed, we find that $\mathsf{in}_w(\mathcal Z_\eta)$ must be tubular for all but finitely many choices of $w$. 
\end{proof}

We are now in a position to construct an algebraically transverse limit.

\begin{construction}[The Gr\"obner limit]
Let $\mathscr G$ denote the unique minimal polyhedral structure on $\trop(\mathcal Z_\eta)$ whose vertices include all points $w$ in $\trop(\mathcal Z_\eta)$ whose initial degeneration is {\it not} tubular. Consider the associated rough expansion $\mathcal V_{\mathscr G}\to \spec(R)$. Perform a ramified base change $R\subset R'$ of order equal to the least common multiple of the denominators appearing in the coordinates of the vertices of $\mathscr G$. Let $\mathcal Y_{\mathscr G}\to \spec(R')$ be the associated expansion. Let $\mathcal Z\to \spec(R')$ be the new flat limit, obtained as the closure of $\mathcal Z_\eta$ after extension of scalars. 
\end{construction}

We now show that this Gr\"obner limit is algebraically transverse. 

\begin{lemma}\label{lem: contraction-lemma}
The Gr\"obner limit family $\mathcal Z\to \spec(R')$ is a family of algebraically transverse subschemes of $\mathcal Y_{\mathscr G}$. 
\end{lemma}

\begin{proof}
Consider the Gr\"obner family of subschemes $\mathcal Z\hookrightarrow \mathcal Y_{\mathscr G}$ over $\spec(R')$ as above. The family is certainly dimensionally transverse, as we have already shown, and we assume that the generic fiber is algebraically transverse. Algebraic transversality is therefore equivalent to the condition that there are no embedded points on the double locus or the divisorial strata in the special fiber. We treat the case of the double locus, the other case is essentially identical. 

The condition may be checked Zariski locally on the subscheme $\mathcal Z$, and depends only on the map from $\mathcal Z$ to the Artin fan of the target $\mathcal Y_{\mathscr G}$. By shrinking to a neighborhood $U$ around the double divisor, and denoting the subscheme by $\mathcal Z_U$, each divisor may be assumed to be principal. Choose generators for the coordinate ring of $U$ that include the equations for these principal divisors, to obtain a maps
\[
\mathcal Z_U\to \mathbb A^n_R\to \mathbb A^2_R,
\]
where the second arrow is given by projection onto the first two coordinates. Note that with respect to the logarithmic structure on $\mathbb A^2_R$ the tropicalization of $\mathcal Z_U$ is a single edge, and we henceforth replace ${\mathscr G}$ to be this edge. Moreover, the logarithmic structure induced on $\mathcal Z_U$ by pullback from $\mathcal Y_{\mathscr G}$ is the pullback of the logarithmic structure coming from the coordinate axes of $\mathbb A^2$. Our task is check flatness of the induced map
\[
\mathcal Z_U\to \left[\A^2/\mathbb G_m^2\right]. 
\]
We now reduce to the toric case. Choose $n-2$ generic hyperplanes and add them to the logarithmic structure of $\mathbb A^n$ above; this logarithmic structure is now toric. This gives rise to a factorization of the above map as
\begin{equation}\label{eq: two-arrows}
\mathcal Z_U\to \left[\A^n/\mathbb G_m^n\right]\to \left[\A^2/\mathbb G_m^2\right].
\end{equation}
The second map is certainly flat, so we need only check flatness for the first. 

With the subscheme now embedded in a toric variety, with its toric logarithmic structure, we may now appeal to the methods of Gr\"obner theory and tropical geometry, in the sense of~\cite[Section~12]{Gub13} or~\cite[Section~6.4]{MS14}. Since the affine hyperplanes are generic, the closures of the initial degenerations of $\mathcal Z_U$ are constant on the faces of the tropicalization of $\mathcal Z_U$. In order to see this, we note that the tropicalization of $\mathcal Z_U$ in the new embedding is obtained from ${\mathscr G}$ by adding unbounded rays to each of its two vertices. Note that since the affine hyperplanes are generic, the initial degenerations at the points of these unbounded rays are tubular, and are constant along these new rays. 

We apply the result of~\cite[Theorem~12.3]{Gub13}. This constructs (non-uniquely) a Gr\"obner complex ${\mathscr G}'$ on the underlying set of ${\mathscr G}$, which is an explicit polyhedral structure on the tropicalization of $\mathcal Z_U$. The associated degeneration $\mathcal Y_{{\mathscr G}'}$ of the generic fiber $\mathcal Y_{{\mathscr G},\eta}$ has the property that (i) the initial degenerations are constant along faces, and (ii) the strict transform of $\mathcal Z_U$ in the associated degeneration is algebraically transverse.\footnote{Note that in the tropical geometry literature, compactifications and degenerations with algebraic transversality are called ``flat tropical compactifications'' in~\cite[Theorem~6.4.17]{MS14} or merely ``tropical compactifications'' in~\cite{Tev07,Gub13}.}

The degenerations $\mathcal Y_{{\mathscr G}'}$ and $\mathcal Y_{\mathscr G}$ can only differ by further subdivision of edges in the tropicalization; call these new $2$-valent vertices \textit{extra vertices}. By algebraic transversality of the closure of $\mathcal Z_U$ in $\mathcal Y_{{\mathscr G}'}$, if we restrict the subscheme to the components corresponding to the extra vertices, the resulting subscheme is pulled back along the blowup $\mathcal Y_{{\mathscr G}'}$ and $\mathcal Y_{\mathscr G}$. We now apply Lemma~\ref{lem: contraction-lemma} to erase all the extra vertices, and see that the closure in $\mathcal Y_{\mathscr G}$ is already algebraically transverse. This is equivalent to the flatness of the first arrow in Eq~(\ref{eq: two-arrows}), and the proof of the lemma is complete. 
\end{proof}

We now establish the uniqueness property that we need. 

\begin{lemma}
The Gr\"obner limit family $\mathcal Z\to \spec(R')$ satisfies the uniqueness property $(\star)$ of Proposition~\ref{thm: uniqueness-strong-case}. 
\end{lemma}

\begin{proof}
Consider another family $\mathcal Z'\to \spec(R'')$ that is algebraically transverse obtained from a polyhedral structure on $\trop(Z_\eta)$. The first claim is that if a point $w$ on $\trop(Z_\eta)$ has non-tubular initial degeneration, then it must appear in the vertex set of any algebraically transverse family. To see this, first observe that if $\mathcal Z'$ is an algebraically transversely family of subschemes in $\mathcal Y$, then if we blowup $\mathcal Y$ further, then the strict and total transforms of $\mathcal Z'$ must coincide. However, the strict transform is precisely the flat limit of $\mathcal Z_\eta$ in the new family. 

With this in mind, we consider a polyhedral structure $\Lambda$ for which there is a point $w$ not contained in the vertex set of such that the initial degeneration at $w$ is not tubular. Form a new polyhedral structure $\Lambda'$ by introducing this point $w$ as part of the vertex set. Examine the total transform of $\mathcal Z$ under the blowup $\mathcal Y_{\Lambda'}\to \mathcal Y_\Lambda$ restricted to the locally closed stratum corresponding to this vertex $w$. By algebraic transversality, it must coincide with the initial degeneration at $w$, but the total transform is necessarily tubular, since it is the preimage of the intersection of $\mathcal Z$ with a double divisor. We arrive at a contradiction. 

The reducedness condition for an expansion $\mathcal Y_{\Lambda'}$ is that every vertex of $\Lambda'$ must be integral. There is a unique minimal set of vertices -- precisely those with non-tubular initial degenerations -- that must be included in \textit{every} algebraically transverse model. It follows that the order of every base change for the Gr\"obner limit divides the order of the base change for every other algebraically transverse model. The uniqueness statement is a consequence.
\end{proof}

\begin{remark}\label{rem:notubes}
We have the following important consequence of the algorithm for producing algebraically transverse limits.  If $G \hookrightarrow \Sigma_X$ denotes the $1$-complex associated to the minimal algebraically transverse limit, $G$ may involve refining the minimal polyhedral structure $\overline G$ by adding bivalent vertices to subdivide edges.  However, if we look at the irreducible components $Y_i$ corresponding to these bivalent vertices, the embedded subscheme $Z_i \subset Y_i$ is stable, i.e. not fixed by the natural $\mathbb{G}_m$-action on $Y_i$.
\end{remark}

With the results of this section as motivation, we define {\it a subscheme of $X$ relative to $D$}. 

\begin{definition}[Relative subschemes]
Let $S$ be a logarithmic scheme. A \textit{subscheme of $X$ relative to $D$  over $S$} is an expansion $\mathcal Y/S$ of $X$ over $S$ together with a flat family of algebraically transverse subschemes $\mathcal Z\subset\mathcal Y$. 
\end{definition}

Proposition \ref{thm: uniqueness-strong-case} suggests that algebraically transverse subschemes will satisfy the valuative criterion for properness.
Our goal is now to construct moduli for these objects -- first for the expansions themselves which happens in the next section, and then for the algebraically transverse subschemes which happens in the one after.

\section{Moduli of target expansions}\label{sec: target-moduli}

The purpose of this section is to construct moduli spaces for the expansions in the previous section, in analogy with the stack of expansions of a smooth pair;  see~\cite[Section~6.1]{ACMUW} for an exposition of the latter.   
Our approach is to use the dictionary between combinatorial moduli spaces and Artin fans, which we review in the first subsection.  Using this, we then study the combinatorial moduli space parametrizing embedded $1$-complexes inside $\Sigma_X$, and show it can be given the structure of a cone space, in the sense of~\cite{CCUW}.   

Two important subtleties arise in this process.  First, as discussed in the introduction, the choice of cone decomposition on the space of embedded $1$-complexes is not unique, and requires an auxiliary combinatorial choice.  Second, in order to endow the universal family with the structure of a cone space, further subdivision is required, which translates into allowing additional codimension-$1$ bubbling in the geometric expansion.  See Remarks~\ref{rem:nonunique} and~\ref{rem:twofamilies}.

\subsection{Artin fans and cone stacks} The stack $\mathsf{Exp}(X|D)$ is similar in nature to Olsson's stack $\mathsf{LOG}$ of logarithmic structures and are instances of the Artin fans that are extracted from $\mathsf{LOG}$, see~\cite{AW,Ols03}. 

An \textit{Artin cone} is a global toric quotient stack:
\[
\mathsf{A}_P = \spec \CC[P]/\spec \CC[P^{\mathsf{gp}}],
\]
where $P$ is a toric monoid. These stacks are endowed with a logarithmic structure in the smooth topology, from the toric logarithmic structure on $\spec \CC[P]$. Observe that $\mathsf{A}_P$ recovers $P$ from its logarithmic structure, so these data naturally determine each other. 

\begin{definition}
An \textit{Artin fan} is a logarithmic algebraic stack $\mathsf{A}$ that has a strict \'etale cover by a disjoint union of Artin cones. 
\end{definition}

In practice, this means that we are interested in algebraic stacks that arise by gluing stacks of the form $\mathsf A_P$ along open substacks corresponding to torus invariant subvarieties of $\spec \CC[P]$. Note that we do not impose the faithful monodromy condition on Artin fans, but our constructions are explicit rather than conceptual, and in practice all of our Artin fans will be monodromy free.

Under mild assumptions, every logarithmic structure on a scheme arises from a morphism to an Artin fan~\cite[Proposition~3.2.1]{ACMUW}. 
In practice, it is convenient pass to an equivalent but explicitly combinatorial $2$-category, for which we follow~\cite[Sections~2.1--2.2]{CCUW}.

\begin{definition}[{Cone spaces}]\label{def: cone-space}
A \textit{cone space} is a collection $\mathcal C = \{\sigma_\alpha\}$ of rational polyhedral cones together with a collection of face morphisms $\mathcal F$ such that the collection $\mathcal F$ is closed under composition, the identity map of every cone $\sigma_\alpha$ lies in $\mathcal F$, and every face of a cone is the image of exactly one face morphism in $\mathcal F$. 
\end{definition}

A \textit{cone complex} is obtained when the further assumption is placed that there is at most one morphism between any two cones. A \textit{cone stack} can be defined similarly, as a fibered category over the category of rational polyhedral cones satisfying natural conditions. We will not need to work with cone stacks in a serious way, so we refer to the reader to~\cite[Section~2]{CCUW} for a definition. Their results allow us to systematically work with Artin fans using combinatorics.  In fact, any cone space can be subdivided into a cone complex, so even cone spaces can be avoided for our purposes.  However, it can be convenient to allow them to construct minimal choices of stacks of target expansions; for example the stack of expansions in~\cite{LiWu15} arises from a cone space.


\begin{theorem}[{\cite[Theorem~3]{CCUW}}]\label{thm: CCUW}
The $2$-categories of Artin fans and of cone stacks are equivalent. 
\end{theorem}

Under this equivalence, an Artin cone $\mathsf{A}_P$ is carried to the cone $\Hom(P,\RR_{\geq 0})$. If $\Sigma$ is a fan embedded in a vector space and $W_{\Sigma}$ is the resulting toric variety with dense torus $\mathbb G_\Sigma$, the equivalence carries the global quotient Artin fan $[W_{\Sigma}/\mathbb G_\Sigma]$ to the cone complex $\Sigma$. 

\subsection{Embedded complexes} Section~\ref{sec: flat-limits} constructed expansions of $X$ along $D$ that accommodate limits of transverse subschemes. The construction was presented over valuation rings with logarithmic structure. In order to globalize this we construct moduli for the tropical data over higher dimensional cones, and then glue these cones to form a cone space. By the categorical equivalence in Theorem~\ref{thm: CCUW}, we will have produced $\mathsf{Exp}(X|D)$ as an Artin stack. 

Recall that $\Sigma$ is a cone complex with smooth cones such that the intersection of any two cones is a face of each, and we view it as a subcomplex of the standard orthant by the embedding
\[
\Sigma\hookrightarrow \RR_{\geq 0}^{\Sigma^{(1)}}.
\]
We will make use of the Euclidean geometry of this ambient vector space. {

In Section~\ref{sec: graphical-prelims} we defined an \textit{abstract $1$-complex} to be a possibly disconnected metric graph $G$ together with polyhedral complex structure. An {\it embedded $1$-complex in $\Sigma$}, or \textit{$1$-complex} for short, was an injective piecewise affine map of polyhedral complexes
\[
G\hookrightarrow \Sigma
\]
given by piecewise affine functions of slope $1$. Let $|{\mathsf{T}(\Sigma)}|$ be the set of isomorphism classes of $1$-complexes $G$ equipped with an embedding
\[
G\hookrightarrow \Sigma.
\]
}
For each embedded $1$-complex, we can dilate the embedded complex by any positive real scalar, giving rise to an copy of $\RR_{\geq 0}$ in $|{\mathsf{T}(\Sigma)}|$. We will endow this set with the structure of a cone space, whose rays will be a subset of the rays obtained in this fashion. 

\subsection{Construction of cones of embedded graphs}\label{sec: constructing-T}   By the conventions in this text, a \textit{graph} refers to a possibly disconnected finite graph without loop or parallel edges, together with a finite collection of rays placed at the vertices. As is standard in Gromov--Witten theory, we visualize the rays as ``legs''. 

In order to put a cone space structure on $|{\mathsf{T}(\Sigma)}|$, we first study spaces of maps from graphs to $\Sigma$ and then identify maps with the same image.  For now, we suppress discussion of the integral structure, and fix the correct integral lattice afterwards.

\begin{definition}
A \textit{combinatorial $1$-complex $G$ in $\Sigma$} is a graph $\underline G$ with vertex, ray, and edge sets $V(\underline G)$, $R(G)$, and $E(\underline G)$, equipped with the following labeling data:
\begin{enumerate}[(C1)]
\item For each vertex $v\in V(\underline G)$ a cone $\sigma_v$ in $\Sigma$.
\item For each flag consisting of a vertex $v$ incident to an edge or leg $e\in E(\underline G)\cup R(G)$, a choice of cone $\sigma_e$ together with a nonzero primitive integral vector in the lattice of $\sigma_e$ (referred to as a \textit{edge direction})
\end{enumerate}
subject to the compatibility condition that if $v$ is an incident vertex of an edge or ray $e$, the cone $\sigma_v$ is a face of $\sigma_e$, including possibly $\sigma_e$ itself. 
\end{definition}
We define a
\emph{linear path} in $G$ to be a path (consisting of edges and/or rays) where every edge/ray has the same direction vector.   Let $P(G)$ denote the set of linear paths in $G$.

An isomorphism of combinatorial  $1$-complexes is an isomorphism of the underlying graphs that is compatible with the labels. Fix a representative for each isomorphism class. We consider embeddings of a combinatorial $1$-complex into $\Sigma$. Let $|\Sigma|$ denote the support of the cone complex. Define the set $\mathbb X_G$ to be the set of functions
\[
f: V(\underline G)\to |\Sigma|
\]
such that
\begin{enumerate}
\item For each $v\in V(\underline G)$, the image $f(v)$ lies in $\sigma_v$.
\item For each $e\in E(\underline G)$ with adjacent vertices $v$ and $W$, the line segment between $f(v)$ and $f(w)$ has direction vector equal to the one labeling $e$. 
\end{enumerate}

Each such function $f\in \mathbb X_G$ determines a polyhedral complex in $\underline\Sigma$. It can be constructed as follows. Take the collection of points obtained as images of vertices of $\underline G$ under $f$. Given an edge $e$ between $v$ and $w$ introduce a segment between $f(v)$ and $f(w)$, whose edge direction is necessarily the one given by $e$. For each ray, glue an unbounded ray with the corresponding base point and edge direction dictated by $G$. This topological space has a minimal polyhedral complex structure whose vertices are of two types: (1) images of vertices of $\underline G$ under $f$ and (2) points where edges or rays intersect.

We refer to this embedded complex as the \textit{image}\footnote{The terminology is mildly inaccurate, since there may be no polyhedral map from the source of $f$ to its image as a polyhedral complex without first subdividing the source. It is however the smallest polyhedral structure on the topological image with the property that images of vertices under $f$ are vertices.}, and denote it by $\mathsf{im}(f)$. We restrict to those $G$ such that there exists at least one function $f\in \mathbb X_G$ such that $\mathsf{im}(f)$ is isomorphic to the complex $|\underline G|$, the latter with its canonical polyhedral structure.

\begin{lemma}
The set $\mathbb X_G$ has the structure of a cone.
\end{lemma}

\begin{proof}
The set of all functions $f: V(\underline G)\to |\Sigma|$ satisfying only the first of the two conditions above is given by the product of $\sigma_v$ ranging over all vertices $v$. Inside this product, we consider the locus where the second condition is imposed. Consider an edge $E$ with incident vertices $v$ and $w$. Consider the line containing $f(V)$ with edge direction given by $e$. The intersection point of this line with the cone $\sigma_w$ is linear in the coordinates of $f(v)$. It follows that the locus where the second condition above is satisfied for a given edge $e$ is linear, and that $\mathbb X_G$ is cut out from the product of all cones $\sigma_v$ by a finite collection of linear equations.  This
gives us a cone structure on $\mathbb X_G$. 
\end{proof}

As mentioned above, each point $f\in \mathbb X_G$ determines an embedded $1$-complex $\mathsf{im}(f)$. Our primary interest is in embedded complexes, so we would like to replace the space $\mathbb X_G$ with the corresponding set of images. However, the topological type of the image is not constant even for points in the relative interior of a cone in $\mathbb X_G$. A key example is provided by two skew edges in a $3$-dimensional cone that intersect in their interiors. We record the example explicitly to illustrate the behavior. 

\begin{example}
Consider the family of embedded $1$-complexes in $\mathbb R^3_{\geq 0}$ parameterized by a cone $\sigma$ which we identify with $\RR_{\geq 0}^2$ and furnish with coordinates $(s,t)$. At a point $(s,t)$ the associated $1$-complex is a union of two components $\rho_1(s,t)$ and $\rho_2(s,t)$. At the point $(s,t)$ the ray $\rho_1(s,t)$ is the unbounded ray connecting $(0,0,s)$ and $(1,1,s)$. At the point $(s,t)$ the segment $\rho_2(s,t)$ is the line segment that connects $(t,0,0)$ and $(0,t,t)$. The two segments intersect along a codimension $1$ subspace in the interior of the base cone $\sigma$. The total spaces of $\rho_1(s,t)$ and $\rho_2(s,t)$, taken indivisually, form $2$-dimensional cones in the $4$-dimensional space $\RR^3_{\geq 0}\times\sigma$. However, their union is not a cone complex, as their intersection is not a face of either. For each fixed value of $(s,t)$ there is a minimal polyhedral structure on the union $\rho_1(s,t)\cup\rho_2(s,t)$. The combinatorial type, and in particular the face poset, of this fiber can change in the interior of the cone $\sigma$. The figure below displays a slice of this picture, see Figure~\ref{fig: the-famous-X}

\begin{figure}

\tikzset{every picture/.style={line width=0.75pt}} 

\begin{tikzpicture}[scale = 0.8, x=0.75pt,y=0.75pt,yscale=-1,xscale=1]

\draw    (100,100) -- (140,140) ;
\draw [shift={(100,100)}, rotate = 45] [color={rgb, 255:red, 0; green, 0; blue, 0 }  ][fill={rgb, 255:red, 0; green, 0; blue, 0 }  ][line width=0.75]      (0, 0) circle [x radius= 3.35, y radius= 3.35]   ;
\draw    (100,190) -- (200,110) ;
\draw [shift={(200,110)}, rotate = 321.34] [color={rgb, 255:red, 0; green, 0; blue, 0 }  ][fill={rgb, 255:red, 0; green, 0; blue, 0 }  ][line width=0.75]      (0, 0) circle [x radius= 3.35, y radius= 3.35]   ;
\draw [shift={(100,190)}, rotate = 321.34] [color={rgb, 255:red, 0; green, 0; blue, 0 }  ][fill={rgb, 255:red, 0; green, 0; blue, 0 }  ][line width=0.75]      (0, 0) circle [x radius= 3.35, y radius= 3.35]   ;
\draw    (200,200) -- (160,160) ;
\draw [shift={(200,200)}, rotate = 225] [color={rgb, 255:red, 0; green, 0; blue, 0 }  ][fill={rgb, 255:red, 0; green, 0; blue, 0 }  ][line width=0.75]      (0, 0) circle [x radius= 3.35, y radius= 3.35]   ;
\draw    (230,160) .. controls (269.2,130.6) and (289.19,187.64) .. (327.62,161.69) ;
\draw [shift={(330,160)}, rotate = 143.13] [fill={rgb, 255:red, 0; green, 0; blue, 0 }  ][line width=0.08]  [draw opacity=0] (10.72,-5.15) -- (0,0) -- (10.72,5.15) -- (7.12,0) -- cycle    ;
\draw    (390,110) -- (440,160) ;
\draw [shift={(440,160)}, rotate = 45] [color={rgb, 255:red, 0; green, 0; blue, 0 }  ][fill={rgb, 255:red, 0; green, 0; blue, 0 }  ][line width=0.75]      (0, 0) circle [x radius= 3.35, y radius= 3.35]   ;
\draw [shift={(390,110)}, rotate = 45] [color={rgb, 255:red, 0; green, 0; blue, 0 }  ][fill={rgb, 255:red, 0; green, 0; blue, 0 }  ][line width=0.75]      (0, 0) circle [x radius= 3.35, y radius= 3.35]   ;
\draw    (390,200) -- (490,120) ;
\draw [shift={(490,120)}, rotate = 321.34] [color={rgb, 255:red, 0; green, 0; blue, 0 }  ][fill={rgb, 255:red, 0; green, 0; blue, 0 }  ][line width=0.75]      (0, 0) circle [x radius= 3.35, y radius= 3.35]   ;
\draw [shift={(390,200)}, rotate = 321.34] [color={rgb, 255:red, 0; green, 0; blue, 0 }  ][fill={rgb, 255:red, 0; green, 0; blue, 0 }  ][line width=0.75]      (0, 0) circle [x radius= 3.35, y radius= 3.35]   ;
\draw    (490,210) -- (440,160) ;
\draw [shift={(490,210)}, rotate = 225] [color={rgb, 255:red, 0; green, 0; blue, 0 }  ][fill={rgb, 255:red, 0; green, 0; blue, 0 }  ][line width=0.75]      (0, 0) circle [x radius= 3.35, y radius= 3.35]   ;
\draw    (570,130) -- (570,198) ;
\draw [shift={(570,200)}, rotate = 270] [color={rgb, 255:red, 0; green, 0; blue, 0 }  ][line width=0.75]    (10.93,-3.29) .. controls (6.95,-1.4) and (3.31,-0.3) .. (0,0) .. controls (3.31,0.3) and (6.95,1.4) .. (10.93,3.29)   ;
\draw    (80,260) -- (450,260) ;
\draw [shift={(450,260)}, rotate = 0] [color={rgb, 255:red, 0; green, 0; blue, 0 }  ][fill={rgb, 255:red, 0; green, 0; blue, 0 }  ][line width=0.75]      (0, 0) circle [x radius= 3.35, y radius= 3.35]   ;
\draw    (450,260) -- (500,260) ;

\draw (111,92.4) node [anchor=north west][inner sep=0.75pt]    {$e_{1}$};
\draw (111,190.4) node [anchor=north west][inner sep=0.75pt]    {$e_{2}$};
\draw (539,88.4) node [anchor=north west][inner sep=0.75pt]    {$\mathbb{R}^{3} \times \mathbb{R}_{\geq 0}$};
\draw (558,250.4) node [anchor=north west][inner sep=0.75pt]    {$\mathbb{R}_{\geq 0}$};
\draw (131,282.4) node [anchor=north west][inner sep=0.75pt]    {$t\neq 0$};
\draw (431,282.4) node [anchor=north west][inner sep=0.75pt]    {$t=0$};

\end{tikzpicture}
\caption{The figure shows a family of embedded $1$-complexes in $\mathbb R^3$ over a $1$-dimensional base. The $1$-complex on the left corresponds to a generic point but at $t = 0$ the two skew edges cross forming a new vertex. As a consequence, the total space has two $2$-dimensional cones meeting at a point in their interior. In order to obtain a total space that, after taking cones, is a cone complex, vertices have to be added to $e_1$ and $e_2$.}\label{fig: the-famous-X}
\end{figure}

\end{example}

The topological changing in the interior of a cone, as in the above example, can be thought of as a non-flatness of the family of $1$-complexes. The following definition is meant to capture this phenomenon of varying topological type. The topological type can also change due to specialization, e.g. a vertex moving from a cone to a face. But it is the changes of topological type in the interior that is the exotic phenomenon, as compared to earlier moduli problems~\cite{ACP,CCUW,GS13}.

\begin{definition}\label{def: surjection}
Let $G$ and $H$ be combinatorial $1$-complexes in $\Sigma_X$. A \textit{surjection} from $G$ to $H$ is given by the a pair of maps
\begin{eqnarray*}
\tau: V(\underline G)&\to& V(\underline H)\\
\upsilon: E(\underline G)\sqcup R(G)&\to& P(\underline H)
\end{eqnarray*}
subject to the following conditions for:
\begin{enumerate}
\item If $v$ and $w$ are the vertices adjacent to $e$, then $\tau(v)$ and $\tau(w)$ are the endpoints of the path $\upsilon(E)$. 

\item Let $e$ be an edge of $\underline G$. Orient the edges of $\nu(e)$ to form a directed path and choose the orientation of $E$ compatibly, i.e. if $e$ is directed from endpoint $v$ to endpoint $w$ then $\nu(e)$ is oriented from $\tau(v)$ to $\tau(w)$. Then, 

If $e$ is an edge of $\underline G$, then the edge directions for each edge in $\upsilon(e)$ coincide with that of $e$.

\item For each vertex of $\underline G$ the cone $\sigma_{\tau(v)}$ is equal to a face of $\sigma_v$. 
\item Each edge or ray in $H$ is contained in a path which is in the image of $\upsilon$.
\end{enumerate}
\end{definition}

Each point $f\in \mathbb X_G$ determines a surjection of $G$, as follows. Given a point $f\in \mathbb X_G$, we consider graph associated to the image $\mathsf{im}(f)$, considered with its natural structure as a combinatorial $1$-complex. A basic finiteness result holds.

\begin{lemma}\label{lem: T-1}
Let $G$ and $\mathbb X_G$ be as above. The set of surjections of $G$ associated to points of $\mathbb X_G$ is finite.
\end{lemma}

\begin{proof}
We may view the cone $\mathbb X_G$ as a cone of tropical maps with type $G$ to $\Sigma$, in the sense of~\cite[Remark~1.21]{GS13}. It comes equipped with a moduli diagram
\[
\begin{tikzcd}
\mathbf G\arrow{d}\arrow{r}& \Sigma\times\mathbb X_G \\
\mathbb X_G.&
\end{tikzcd}
\]
The fibers of the vertical map are the metrizations of $G$ given by endowing each edge $e$ with length equal to the distance in $\Sigma$ between the images of its endpoints. Since the horizontal map is linear on cones, the image of $\mathbf G$ in $\Sigma\times\mathbb X_G$ can be given the structure of a cone complex. Choose any such cone complex structure $\mathsf{im}(\mathbf G)$. After subdividing this image further, and replacing $\mathbb X_G$ by a subdivision, the induced map
\[
\widetilde {\mathbf G}\to \widetilde {\mathbb X}_G
\]
is flat in the combinatorial sense -- every cone of the source surjects onto a cone of the image. This follows from~\cite[Section~4]{AK00}. After this subdivision, the combinatorial type of the $1$-complex in the fibers in the relative interior of any cell is constant. It follows that there are only finitely many surjections that appear from taking image, as claimed. 
\end{proof}

Let $G$ be as above, and let $H$ be a surjection of $G$ associated to a point in $\mathbb X_G$. Note that we will only consider surjections associated to points of cones $\mathbb X_G$. Let $\mathbb X_H$ be the cone associated to the combinatorial $1$-complex $H$. There is an associated inclusion
\[
\mathbb X_H\hookrightarrow \mathbb X_G.
\]
Specifically, a point of $\mathbb X_H$ in particular determines a function on $V(\underline G)$ by using the map on vertex sets $V(\underline G)\to V(\underline H)$. 

\begin{lemma}\label{lem: common-surjections}
Let $G\to H$ and $G\to K$ be two surjections of $G$ obtained from points of $\mathbb X_G$ by the image construction. Consider a point 
\[
f\in \mathbb X_H\cap \mathbb X_K\subset \mathbb X_G.
\]
The surjection $G\to J$ determined by $f$ is a common surjection of $H$ and $K$. 
\end{lemma}

\begin{proof}
The lemma follows immediately from the definitions by taking $J$ to be the image of $f$. 
\end{proof}

Fix a combinatorial $1$-complex $G$ and let $\mathsf{Aut}_G$ be the automorphism group. There is a natural action of $\mathsf{Aut}_G$ on the cone $\mathbb X_G$. Our goal will be construct an automorphism-equivariant subdivision of $\mathbb X_G$ such that the subsets $\mathbb X_H$ for surjections $H$ of $G$ become subcomplexes. The following lemma will be of use to us. 

\begin{lemma}\label{lem: equivariant-strictness}
Let $\mathbf C$ be a cone and let $\mathbf F\subset \mathbf C$ be an embedded cone complex. Let $\Gamma$ be a finite group acting on $\mathbf C$ and such that $\mathbf F$ is $\Gamma$-stable. There exists a $\Gamma$-equivariant complete subdivision $\widetilde{\mathbf C}$ of $\mathbf C$ such that $\mathbf F$ is a union of faces of $\widetilde{\mathbf C}$. If $\mathbf F$ is a smooth cone complex, then $\widetilde {\mathbf C}$ can be chosen to be smooth. 
\end{lemma}

\begin{proof}
Let $N_\RR$ denote the associated vector space of the cone $\mathbf C$. We consider $\mathbf F$ and $\mathbf C$ as subsets of this vector space. The existence of equivariant completions for toric varieties guarantees that there exists a complete fan $\Delta$ in $N_\RR$ that contains $\mathbf F$ as a subcomplex~\cite{EI06}. Consider the common refinement of $\Delta$ and $\mathbf C$:
\[
\mathbf C' = \{\sigma_1\cap\sigma_2: \sigma_1\in \Delta \ \ \sigma_2\in \mathbf C\}.
\]
This is a cone complex structure on the intersection of the supports of the two fans, that is on $|\mathbf C|$, see~\cite[Section~2.3]{MS14}. It clearly contains $\mathbf F$ as a subcomplex. The subdivision is not $\Gamma$ equivariant. We fix this by averaging over the group. Given $\gamma\in\Gamma$, the set of translates of cones in $\mathbf C'$ form a cone complex with support $|\mathbf C|$. The common refinement over all group elements gives rise to the subdivision $\widetilde{\mathbf C}$ as claimed in the lemma. Since $\mathbf F$ is $\Gamma$-stable, this last common refinement step does not subdivide it. The final statement on smoothness follows from the fact that toric resolution of singularities for any fan can be performed without subdividing smooth cones~\cite[Theorem~11.1.9]{CLS11} and the fact that these subdivisions can be made $\Gamma$ equivariant~\cite[Section~2]{AW97}.
\end{proof}

Note that this procedure typically involves choices, and is not canonical. Our application is to the cones associated to combinatorial one complexes.

\begin{lemma}\label{lem: T-4}
For each combinatorial $1$-complex $G$, there exists an $\mathsf{Aut}_G$-equivariant subdivision $\mathbb Y_G\to \mathbb X_G$ such that for each surjection $G\to H$ induced by a point of $\mathbb X_G$, the induced map
\[
\mathbb Y_H\hookrightarrow \mathbb Y_G
\]
is an inclusion of a subcomplex. If $\mathbb Y_H$ is smooth, then $\mathbb Y_G$ can be chosen to be smooth.
\end{lemma}

\begin{proof}
Given $G$, the set of all surjections of $G$ forms a partially ordered set, since the notion is stable under composition. Choose any total order extending this partial order. At the minimal elements of this partial order, there is no subdivision necessary, and the union of the cones associated to these types forms a cone complex in $\mathbb X_G$. Proceeding in the chosen order, we repeatedly apply Lemma~\ref{lem: equivariant-strictness}. Given $H$ and its surjections $J_1,\ldots J_k$, we observe that by Lemma~\ref{lem: common-surjections}, the union of the inductively chosen $\mathbb Y_{J_i}$ is a cone complex. We may therefore apply the preceding lemma to obtain an $\mathsf{Aut}_H$-equivariant subdivision $\mathbb Y_H$ of $\mathbb X_H$ such that each previously constructed $\mathbb Y_{J_i}$ is a union of faces. The lemma follows. The smoothness statement follows from the corresponding statement in the preceding lemma. 
\end{proof}

\subsection{Tropical moduli of expansions}\label{sec:tropical-moduli} We now construct the moduli space of expansions via a colimit procedure. Fix a set of subdivisions $\mathbb Y_G$ as guaranteed by the lemma above, one for each choice of combinatorial $1$-complex $G$. Consider the cone complex given by the disjoint union $\sqcup \{\mathbb Y_G\}$ taken over all such $G$. We want to identify two cones $\sigma$ and $\sigma'$ if they are identified via an inclusion $\mathbb Y_H\hookrightarrow \mathbb Y_G$ induced by a surjection or identified via an automorphism of $G$.
Consider the equivalence relation on cones generated by these identifications.In what follows, we will take the quotient by this equivalence relation. The quotient exists in the category of generalized cone complexes\footnote{A generalized cone complex can be thought of as the coarse moduli space of a cone stack. We are interested here only in cone spaces, so the formalisms of~\cite{ACP} and~\cite{CCUW} coincide.}, defined by Abramovich--Caporaso--Payne~\cite[Section~2.6]{ACP}. In showing that such quotients exist in their category, those authors also explain that the quotient also carries a natural cone space (in fact, even a cone complex) structure after subdivision, obtained as follows.  We subdivide each $\mathbb Y_G$ by taking barycentric subdivision of all the cones.   Under the barycentric subdivision procedure above, if an element of $\mathsf{Aut}_G$ stabilizes a cone of $\mathbb Y_G$, then it fixes it pointwise.  The cones of the quotient are precisely the equivalence 
classes considered above, generated by inclusions as $G$ varies or by automorphisms of $G$. Note that the notion of a graph surjection captures both simple specialization phenomena, such as contracting edges of $G$, as well as the more exotic topological type changes in the interior.

\begin{definition}
A \textit{moduli space of tropical expansions} $T$ is the quotient of the disjoint union of cone complexes $\mathbb Y_G$ over all combinatorial $1$-complexes $G$, with cone structure as described in the preceding paragraph. The underlying set will be denoted $|T|$. 
\end{definition}

The quotient as a set is precisely the set of embedded $1$-complexes in $\Sigma$. This also endows it, for example, with the structure of a topological space. We find this topological structure useful in thinking about the problem, but do not need it any technical sense here. This latter structure is examined more thoroughly in~\cite{KH22}.

By toric resolution of singularities, we can guarantee that the cones of $T$ are smooth. From this point forth, we only work with those $T$ that are smooth. In parallel constructions in Gromov--Witten theory, it is natural to allow toric singularities in analogous spaces~\cite{AW,GS13}. We choose not to do this here, however this is largely for convenience in the applications that we currently foresee~\cite{MR22}.

\begin{remark}\label{rem:nonunique}
Note that the cone spaces $T$ constructed above are non-unique.  The ambiguity arises from the inductive construction of the subdivision $\mathbb Y_G$ as a subdivision of $\mathbb X_G$, for each graph $G$. The construction is analogous to, and derived from, the result on equivariant completions of toric varieties and on completions of fans~\cite{EI06}. In this sense, the root cause of the non-uniqueness of our spaces is the same as the non-uniqueness seen in the compactification problems for tori. In particular, dropping the smoothness discussed above does \textit{not} alleviate the non-uniqueness.
\end{remark}

\subsection{Universal family and integral structure}\label{sec:universal}

Given a moduli space $T$ of tropical expansions, we want to equip it with a universal
expansion.  More precisely, we will construct a diagram of cone spaces
\[
\begin{tikzcd}
{\Upsilon\subset \Sigma\times T} \arrow{d}\arrow{r} &\Sigma\\
T,&
\end{tikzcd}
\]
where ${\Upsilon\subset \Sigma\times T}$ is a subdivision and the map $\Upsilon\to T$ is flat in the combinatorial sense: i.e. cones map surjectively onto cones.

In order to construct this diagram, we observe that each $\mathbb X_G$ comes equipped with a universal graph $\mathbf G$ and a map to $\Sigma\times \mathbb X_G$. In order to construct $\Upsilon$ we simply replicate every step in the construction of $T$ above, modified as follows.

For each $G$, we consider the augmented graph $G_\ast$ consisting of the disjoint union of $G$ with a new distinguished vertex $\ast$ (with no edges connecting it to vertices of $G$).
If we denote $s$ a choice of vertex, edge, or ray of $G$,
we consider the set $\mathbb U_G(s)$ consisting of a function from $G_\ast$ to $\Sigma$ whose restriction to $G$ is in $\mathbb X_G$ and such that $f(\ast)$ lies on the image of $f(s)$.

Each element of $\mathbb U_G(s)$ determines a $1$-complex in $\Sigma$ equipped with distinguished \textit{basepoint} lying on $s$; this basepoint is the image $f(\ast)$.  The set $\mathbb U_G(s)$ forms a cone, since the choice for the basepoint is a cone, namely the cone formed by the edge or vertex $s$, as $f$ varies in $\mathbb X_G$.  As we consider all possible choices of $s$, these cones $\mathbb U_G(s)$ form a cone complex $\mathbb U_G$ equipped with an action of $\mathsf{Aut}_G$ and a morphism
\[
\iota_G: \mathbb U_G\to \mathbb X_G \times \Sigma.
\]
Given a surjection $G \to H$, an edge $s$ of $G$, and an edge $t$ of $H$ which is contained in the path $\upsilon(s)$, we have 
\[
\mathbb U_H(t) \hookrightarrow \mathbb U_G(s).
\]
We have similar maps when $s$ or $t$ are rays or vertices.  Notice given a surjection, there could be many distinct embeddings of $\mathbb U_H(t)$ into $\mathbb U_G$.

Following our previous approach, Lemmas~\ref{lem: T-1}, \ref{lem: common-surjections}, \ref{lem: equivariant-strictness}, and~\ref{lem: T-4} each apply in this setting of pointed graphs. They give rise to another system of subdivisions $\widetilde{\mathbb U_G}$ of $\mathbb U_G$, which are compatible with relabeling, the pullback of the subdivision $\mathbb Y_G$, and the inclusions of $\widetilde{ \mathbb U_H(t)}$. If we apply the map $\iota_G$ to this subdivision, it is injective on every cone, and the images of any two cones intersect along a unique face of each. Therefore, there is a natural cone structure on its image $\mathbb V_G \subset  \mathbb Y_G \times \Sigma$, along with natural maps $\mathbb V_H \rightarrow \mathbb V_G$ associated to surjections.  We pass to the quotient as before
and declare the colimit to be $\Upsilon$.  

Finally, we fix the integral structure on $T$ and $\Upsilon$ as follows.  
Given a rational point $P$ of $T$, we say it is \emph{integral} if every vertex of the embedded $1$-complex $\Upsilon_P \subset \Sigma$ lies on an integral lattice point of $\Sigma.$  Note that the process is not circular: the lattice structure on the target $\Sigma$ is unambiguous\footnote{Said differently, a $1$-complex is integral if all its vertices are integral. A point of $T$ is integral if the vertices of the universal $1$-complex above it is integral}. Similarly, a rational point $Q \in \Upsilon$ is integral if its image in $T\times \Sigma$ is integral.  One can check that this definition of integral points in each cone $\sigma$ of $T$ is compatible with the usual notion of integral structure, i.e. it is the restriction of a lattice in the ambient rational vector space.  It follows from our definition that the map $\Upsilon \rightarrow T$ is flat and reduced.


\begin{remark}\label{rem:twofamilies}

Given any point $P$ of $T$, we obtain two $1$-complexes associated to this point.  
First, since we can identify points of the topological realization of $T$ with the points $|T|$, we have the $1$-complex $G_P$ associated to this point.  Second, we can take the fiber 
$\Upsilon_P$ of the map
$\Upsilon \rightarrow T$.  
It follows from the construction that there is a map
\[ \Upsilon_P \rightarrow G_P
\]
which is a refinement of polyhedral sets, obtained by adding $2$-valent vertices along edges.
We will refer to the vertices of $\Upsilon_P$ that are not present in $G_P$ as  \textit{tube vertices}, and these will play an important role in our definition of stability.

As with our construction of $T$, a combinatorial choice is required in the construction of $\Upsilon$.  Consequently, given $P$, while the $1$-complex $G_P$ is canonically given, the subdivision $\Upsilon_P$ depends on this choice.  Moreover, the integral structure on $T$ also depends on this additional choice.
\end{remark}



\begin{remark}[Relative compactification]\label{rem: rel-comp}
Fix a moduli space of tropical expansions $T$ and a finite type subspace $T_{\beta,n}$. We abuse notation mildly, allowing $\Upsilon$ to denote the universal embedded $1$-complex in $T_{\beta,n}\times\Sigma_X$. The subdivisions we have constructed for the universal target expansions are typically non-proper. For appropriate choices in the construction of $T$ it is possible to ensure that $\Upsilon$ can be completed to complete, and in fact, projective, subdivision of $T_{\beta,n}\times\Sigma_X$. Indeed, after itreated barycentric subdivision $T_{\beta,n}$ can be embedded in a vector space~\cite[Section~4.6]{ACMW}. Results on equivariant completions for fans of toric varieties and toric Chow's lemma give rise to the completion~\cite{EI06}. Toroidal semistable reduction guarantees flatness after a further sequence of blowups of $T_{\beta,n}$. 
\end{remark}

\subsection{Common refinements}\label{sec:common-refinements}
The different choices involved in constructing $T$ above lead to different cone spaces, but they only differ by subdivision. We record this for future use. Let $\Lambda$ denote the set of all conical structures on the set $|T|$ that arise from the construction above. Given $\lambda\in \Lambda$, let $T^\lambda$ denote the associated moduli space of tropical expansions.

\begin{proposition}\label{prop: common-refinements}
Any two conical structures $\lambda,\mu \in \Lambda$ share a common refinement. Precisely, there exists $\varphi\in \Lambda$ and morphisms
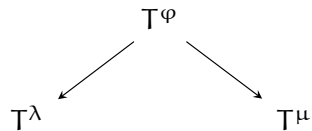
\[
\begin{tikzcd}
&T^\varphi\arrow{dl}\arrow{dr}&\\
T^\lambda & & T^\mu
\end{tikzcd}
\]
such that each of the two arrows is a composition of a complete subdivision and passage to a finite index sublattice. 
\end{proposition}

\begin{proof}
We reinspect the procedure described above. For each combinatorial $1$-complex $G$ and surjection $G\to H$, we choose a subdivision of $\mathbb X_G$ such that the cone complex $\mathbb Y_H$ is a union of faces of this subdivision. In turn, there are choices involved in constructing $\mathbb Y_H$ from $\mathbb X_H$, given by studying surjections $H\to K$ and so on. The moduli spaces $T$ are then formed by taking the colimit of the quotients of these cones by the appropriate automorphism groups. However, we note that two cone complexes on the same underlying space always have a common refinement, whose cones are given by intersections of the cones in the two complexes. Given cone complexes $\mathbb Y_G^\lambda$ and $\mathbb Y_G^\mu$, as their support is the cone $\mathbb X_G$ we can construct their common refinement $\mathbb Y_G^\varphi$. The common refinement yields a compatible system of subdivisions for the different combinatorial $1$-complexes in the following sense. Given a graph surjection $G\to H$ and an inclusion
\[
\mathbb X_H\hookrightarrow \mathbb X_G
\]
we may take the common refinement of $\mathbb X_H$ with each of these three cone complex structures on $\mathbb X_G$ to obtain $\mathbb Y_G^\varphi$ to obtain $\mathbb Y_H^\lambda$, $\mathbb Y_H^\mu$, and $\mathbb Y_H^\varphi$. A definition chase reveals that $\mathbb Y_H^\varphi$ is the common refinement of  $\mathbb Y_H^\lambda$ and $\mathbb Y_H^\mu$. Further, the common refinement of a $\mathsf{Aut}_G$-equivariant subdivisions of $\mathbb X_G$ remains equivariant. It follows that the preceding colimit construction can be run with the quotients $\mathbb Y_G^\varphi/\!/\mathsf{Aut}_G$. By construction, this cone space refines both $T^\lambda$ and $T^\mu$.
\end{proof}

\begin{remark}
There is a broader context for the construction here that we have skirted around in order to keep the discussion concrete. Given a cone complex $\Sigma$, the integer points in its topological realization is canonically identified with the $\NN$-valued points of the functor that $\Sigma$ defines on monoids. Precisely, $\Sigma$ defines the functor
\begin{eqnarray*}
F_\Sigma: \mathsf{Monoids}&\to & \mathsf{Sets}\\
P&\mapsto & \Hom(P^\vee,\Sigma).
\end{eqnarray*}
However, we notice that the $\NN$-valued points of $\Sigma$ and any complete subdivision of $\Sigma$ coincide. More generally, the functors coincide on all \textit{valuative monoids}. In this sense, although the construction of $T$ above is not canonical, the associated functor on \textit{valuative monoids} is canonical. One may view the discussion above as reverse engineering this. We take the integral points set $|T|$ to be the $\NN$-valued points of a putative set of cones $T^\lambda$ on all monoids, which we construct. Functors on valuative monoids (resp. valuative logarithmic schemes) often arise as inverse limits of cone complexes (resp. logarithmic schemes) along subdivisions (resp. logarithmic modifications). However, a functor defined directly on the valuative category needs to satisfy additional properties in order to arise in this fashion. Once these are satisfied, the functor in question, in our case $|T|$ is determined as a polyhedral complex up to refinement. Rather than working with the inverse limit directly, we work with the entire inverse system, making compatible choices to produce morphisms where necessary. There are, at this point, relatively few genuinely valuative approaches to moduli problems. The logarithmic Picard group is a notable exception~\cite{MW18}.
\end{remark}

\subsection{Geometric moduli}\label{sec:geometric-moduli}

We now pass to Artin fans to produce our moduli stack of target expansions.
Consider $\Sigma_X$ the cone complex for $(X|D)$.

Given a moduli space of tropical expansions $T$ for $\Sigma_X$ and its universal family $\Upsilon$, we set $\mathsf{Exp}(X|D)$ to be the Artin fan associated to the cone space $T$ via Theorem~\ref{thm: CCUW}.  It carries a natural logarithmic structure and is smooth if we choose $T$ to be smooth. The points of $\mathsf{Exp}(X|D)$ correspond to cones of $T$.

To produce a universal expansion over $\mathsf{Exp}(X|D)$, we apply the Artin fan construction to the inclusion $\Upsilon \hookrightarrow \Sigma \times T$ to obtain
\[
\mathsf{A}_{\Upsilon} \rightarrow \mathsf{A}_{\Sigma} \times \mathsf{Exp}(X|D).
\]
Using the natural map $X \rightarrow \mathsf{A}_\Sigma$, we set
\[
\mathcal Y := 
\mathsf{A}_{\Upsilon} \times_{ \mathsf{A}_{\Sigma} \times \mathsf{Exp}(X|D)} (X \times\mathsf{Exp}(X|D)).
\]
and consider the diagram
\[
\begin{tikzcd}
\mathcal Y  \arrow{d}\arrow{r} &X \\
\mathsf{Exp}(X|D).&
\end{tikzcd}
\]
Since we arranged for the morphism of cone complexes $\Upsilon \rightarrow T$ to be flat and reduced, $\mathcal Y$ is a family of expansions of $X$ over 
$\mathsf{Exp}(X|D)$.  

Recall that, for every point $P$ of $T$, we have a distinguished set of \emph{tube vertices} which are $2$-valent vertices that appear in the subdivision $\Upsilon_P \rightarrow G_P$.   Given a cone $\sigma$ of $T$, the combinatorics of the universal complex does not change in the interior of $\sigma$, so the tube vertices of $\Upsilon$ are well-defined over $\sigma$ by choosing any point $P$ in its interior.  
If $\mathcal Y_\sigma$ denotes the geometric expansion corresponding to this point of 
$\mathsf{Exp}(X|D)$, the irreducible components corresponding to these tube vertices will be called \textit{tube components}.  These are components that are not needed for applying the valuative criterion, but appear when producing families over the entire stack of degenerations. They play a significant role in the next section.

\section{Moduli of ideal sheaves}\label{moduli-of-sheaves}


In this section, we define stable ideal sheaves on $X$ relative to $D$ and show that, for fixed numerical invariants, their families are parameterized by a proper Deligne-Mumford stack $\mathsf{DT}(X|D)$.  For expository clarity we will assume that $X$ has dimension $3$. Removing this assumption is straightforward but notationally cumbersome. 

Let $\mathsf{Exp}(X|D)$ be a stack of target expansions\footnote{We suppress the choice from the notation for brevity. } of $X$ along $D$ and let $T(\Sigma_X)$ the associated tropical moduli stack. We will denote it by $T$ when there is no chance of confusion.

\subsection{Tube components and tube subschemes} The expansions constructed in the previous section contain certain distinguished components. In order to obtain well-behaved moduli spaces, we restrict subschemes inside such components. 

Let $G\subset \Sigma_X$ be an embedded $1$-complex. A \textit{linear $2$-valent vertex} is a $2$-valent vertex $v$ such that  two conditions hold (i) $v$ and both of its outgoing edges lie in the interior of a single cone of $\Sigma_X$ and (ii) the directions of the two flags are opposite, i.e. a neighborhood of $v$ lies on a line inside the cone $\sigma_v$

Fix a moduli space of $1$-complexes $T$ equipped with a cone structure as in the previous section, as well as a cone complex on the associated family $\Upsilon\to T$. Recall that the underlying set $|T|$ is in natural bijection with the set of embedded $1$-complexes in $\Sigma_X$. 

Choose a point $p$ of $T$. There are two $1$-complexes that one can associate to this point. First, let $G_p$ be the $1$-complex obtained by examining the fiber of $p$ in $\Upsilon$. Second, let $G^\star_p$ be the embedded $1$-complex associated to the point $p\in |T|$. There is a natural refinement $c: G_p\to G_p^\star$ which erases some -- but not all -- the linear $2$-valent vertices. We maintain this notation for the next definition. 

\begin{definition}[Tube vertex]
Let $v$ be a linear $2$-valent vertex of $G_p$. The vertex $v$ is called a \textit{tube vertex} if its image $c(v)$ under the refinement map above is not a vertex of $G^\star_p$. 
\end{definition}

\begin{remark}
Tube vertices are a subset of linear $2$-valent vertices, and we have defined a tube vertex by forcing the linear $2$-valency condition for clarity. It should be noted that vertices that are not linear and $2$-valent could never arise from a subdivision, so in fact the linear $2$-valency condition could be removed from the preceding definition without change. 
\end{remark}

The notion geometrizes as follows. Given a point of $\mathsf{Exp}(X|D)$, there is an associated expansion $\mathcal Y_p$ there is an associated cone of $T$. Given any point $p$ in the interior of this cone, we obtain a $1$-complex $G_p$. A component of $\mathcal Y_p$ is called a \textit{tube component} if it corresponds to a tube vertex in $G_p$. 

Note that each tube component is a $\mathbb P^1$-bundle $\mathbb P$ over a surface $S$.  

Given a $\mathbb P^1$- bundle $\mathbb P$ over a surface $S$, a subscheme of $\mathbb P$ is a \textit{tube subscheme} if it is the schematic preimage of a zero-dimensional subscheme in $S$.\footnote{These are the analogues of ``trivial bubbles'' in relative and expanded Gromov--Witten theory.} We say that a component of an expansion is \textit{linear $2$-valent} if the corresponding vertex in its $1$-complex is linear $2$-valent. Linear $2$-valent components components are all $\mathbb P^1$-bundles over a surface $S$.

\noindent
{\bf Tube terminology: a summary.} There are three different objects that have the adjective ``tube''. A tube vertex is a distinguished linear $2$-valent vertex of a $1$-complex. A tube component is a component of an expansion that corresponds to such a tube vertex. And finally, a tube subscheme is a subscheme of a linear $2$-valent component of an expansion (so in particular a $\mathbb P^1$-bundle) that is pulled back along the bundle. A priori, there is nothing that stops a non-tube, linear $2$-valent component from containing a tube subscheme. We will explicitly stop this from happening in our moduli problem, which we now come to.

\subsection{Stable ideal sheaves}  We now state our main definition.

\begin{definition}
A \textit{stable ideal sheaf on $X$ relative $D$} over a closed point $p$ is a morphism $$p\to\mathsf{Exp}(X|D)$$ with associated expansion $\mathcal Y_p$, and a proper, algebraically transverse subscheme $\mathcal Z_p\subset\mathcal Y_p$, which satisfies the following stability condition:

\underline{\it DT stability}: \ \ the subscheme $\mathcal Z_P$ restricted to a linear $2$-valent component of $\mathcal Y_P$ is a tube subscheme if and only if that component is a tube component.

A \textit{family of stable ideal sheaves on $X$ relative to $D$ over a  scheme $S$} is a morphism $S\to\mathsf{Exp}(X|D)$ with associated expansion $\mathcal Y_S$ and a flat family of algebraically transverse proper subschemes $\mathcal Z_S\hookrightarrow \mathcal Y_S$ such that DT stability is satisfied in each geometric fiber. 
\end{definition}

Recall that algebraic transversality implies that $Z$ meets every stratum of $\mathcal Y_p$.  Also, notice that a family of stable ideal sheaves over a scheme $S$ induces a logarithmic structure on $S$ by pulling back the natural structure on the Artin fan $\mathsf{Exp}(X|D)$. 

\begin{remark}
The DT stability condition stems from the fact that while $|T|$ genuinely parameterizes $1$-complexes, it is not a cone space or even a cone stack. On the other hand, given a cone space structure on $T$ and its universal family $\Upsilon$, there may be points of $T$ such that the universal family may be identical. 

The DT stability condition itself can be understood as follows. Suppose we have two points $P$ and $Q$ in $T$ such that the fiber in $\Upsilon$ gives the same embedded $1$-complex, and let us denote this $1$-complex by $\Theta$. However, $P$ and $Q$ are also points of $|T|$ which determine \textit{distinct} $1$-complexes. It follows that the fibers $\Upsilon_P$ and $\Upsilon_Q$, which are both identified with $\Theta$, are obtained from two distinct $1$-complexes via subdivision. Necessarily then, the choices of points $P$ and $Q$ determine two distinct ways to coarsen $\Theta$ by erasing \textit{certain} but \textit{not all} $2$-valent vertices. A vertex of $\Theta$ that is erased according to the point $P$ will be a tube vertex for $\Upsilon_P$, and will determine a tube component. However, such a vertex need not be erased according to the point $Q$. Passing to the corresponding geometric picture, the $1$-complex $\Theta$ determines an expansion, and some of the components in this expansion may be $\mathbb P^1$-bundle components. But a given such component might be a tube component in the expansion corresponding to $P$ but not in the one corresponding to $Q$, and vice versa. If the component is a tube, the expansion can be blown down by erasing the vertex. The DT stability condition then insists that the subscheme is pulled back along this blow down. Figure~\ref{fig: tube-picture} depicts the situation. 
\end{remark}

\begin{figure}

\tikzset{every picture/.style={line width=0.75pt}} 

\begin{tikzpicture}[x=0.75pt,y=0.75pt,yscale=-1,xscale=1,scale=0.8]

\draw   (288,760.95) .. controls (288,724.18) and (326.54,694.36) .. (374.09,694.36) .. controls (421.63,694.36) and (460.17,724.18) .. (460.17,760.95) .. controls (460.17,797.73) and (421.63,827.55) .. (374.09,827.55) .. controls (326.54,827.55) and (288,797.73) .. (288,760.95) -- cycle ;
\draw    (314.52,717.55) .. controls (316.19,715.88) and (317.85,715.88) .. (319.52,717.55) .. controls (321.19,719.22) and (322.85,719.22) .. (324.52,717.55) .. controls (326.19,715.88) and (327.85,715.88) .. (329.52,717.55) .. controls (331.19,719.22) and (332.85,719.22) .. (334.52,717.55) .. controls (336.19,715.88) and (337.85,715.88) .. (339.52,717.55) .. controls (341.19,719.22) and (342.85,719.22) .. (344.52,717.55) .. controls (346.19,715.88) and (347.85,715.88) .. (349.52,717.55) .. controls (351.19,719.22) and (352.85,719.22) .. (354.52,717.55) .. controls (356.19,715.88) and (357.85,715.88) .. (359.52,717.55) .. controls (361.19,719.22) and (362.85,719.22) .. (364.52,717.55) .. controls (366.19,715.88) and (367.85,715.88) .. (369.52,717.55) .. controls (371.19,719.22) and (372.85,719.22) .. (374.52,717.55) .. controls (376.19,715.88) and (377.85,715.88) .. (379.52,717.55) .. controls (381.19,719.22) and (382.85,719.22) .. (384.52,717.55) .. controls (386.19,715.88) and (387.85,715.88) .. (389.52,717.55) .. controls (391.19,719.22) and (392.85,719.22) .. (394.52,717.55) .. controls (396.19,715.88) and (397.85,715.88) .. (399.52,717.55) .. controls (401.19,719.22) and (402.85,719.22) .. (404.52,717.55) .. controls (406.19,715.88) and (407.85,715.88) .. (409.52,717.55) .. controls (411.19,719.22) and (412.85,719.22) .. (414.52,717.55) .. controls (416.19,715.88) and (417.85,715.88) .. (419.52,717.55) .. controls (421.19,719.22) and (422.85,719.22) .. (424.52,717.55) .. controls (426.19,715.88) and (427.85,715.88) .. (429.52,717.55) .. controls (431.19,719.22) and (432.85,719.22) .. (434.52,717.55) .. controls (436.19,715.88) and (437.85,715.88) .. (439.52,717.55) .. controls (441.19,719.22) and (442.85,719.22) .. (444.52,717.55) .. controls (446.19,715.88) and (447.85,715.88) .. (449.52,717.55) .. controls (451.19,719.22) and (452.85,719.22) .. (454.52,717.55) .. controls (456.19,715.88) and (457.85,715.88) .. (459.52,717.55) .. controls (461.19,719.22) and (462.85,719.22) .. (464.52,717.55) .. controls (466.19,715.88) and (467.85,715.88) .. (469.52,717.55) .. controls (471.19,719.22) and (472.85,719.22) .. (474.52,717.55) .. controls (476.19,715.88) and (477.85,715.88) .. (479.52,717.55) .. controls (481.19,719.22) and (482.85,719.22) .. (484.52,717.55) .. controls (486.19,715.88) and (487.85,715.88) .. (489.52,717.55) .. controls (491.19,719.22) and (492.85,719.22) .. (494.52,717.55) .. controls (496.19,715.88) and (497.85,715.88) .. (499.52,717.55) .. controls (501.19,719.22) and (502.85,719.22) .. (504.52,717.55) -- (505.83,717.55) -- (505.83,717.55) ;
\draw    (404,820) .. controls (402.41,818.26) and (402.48,816.59) .. (404.22,815) .. controls (405.95,813.41) and (406.02,811.74) .. (404.43,810.01) .. controls (402.84,808.27) and (402.91,806.6) .. (404.65,805.01) .. controls (406.39,803.42) and (406.46,801.76) .. (404.87,800.02) .. controls (403.28,798.29) and (403.35,796.62) .. (405.08,795.02) .. controls (406.82,793.43) and (406.89,791.77) .. (405.3,790.03) .. controls (403.71,788.29) and (403.78,786.62) .. (405.52,785.03) .. controls (407.25,783.44) and (407.32,781.77) .. (405.73,780.04) .. controls (404.14,778.3) and (404.21,776.63) .. (405.95,775.04) .. controls (407.69,773.45) and (407.76,771.79) .. (406.17,770.05) .. controls (404.58,768.32) and (404.65,766.65) .. (406.38,765.05) .. controls (408.12,763.46) and (408.19,761.8) .. (406.6,760.06) .. controls (405.01,758.32) and (405.08,756.65) .. (406.82,755.06) .. controls (408.55,753.47) and (408.62,751.8) .. (407.03,750.07) .. controls (405.44,748.33) and (405.51,746.66) .. (407.25,745.07) .. controls (408.99,743.48) and (409.06,741.82) .. (407.47,740.08) .. controls (405.88,738.35) and (405.95,736.68) .. (407.68,735.08) .. controls (409.42,733.49) and (409.49,731.82) .. (407.9,730.08) .. controls (406.31,728.34) and (406.38,726.68) .. (408.12,725.09) .. controls (409.85,723.49) and (409.92,721.82) .. (408.33,720.09) .. controls (406.74,718.35) and (406.81,716.69) .. (408.55,715.1) .. controls (410.29,713.51) and (410.36,711.84) .. (408.77,710.1) .. controls (407.18,708.37) and (407.25,706.7) .. (408.98,705.11) .. controls (410.72,703.52) and (410.79,701.85) .. (409.2,700.11) .. controls (407.61,698.37) and (407.68,696.71) .. (409.42,695.12) .. controls (411.15,693.52) and (411.22,691.85) .. (409.63,690.12) .. controls (408.04,688.38) and (408.11,686.72) .. (409.85,685.13) .. controls (411.59,683.54) and (411.66,681.87) .. (410.07,680.13) .. controls (408.48,678.4) and (408.55,676.73) .. (410.28,675.14) .. controls (412.02,673.55) and (412.09,671.88) .. (410.5,670.14) .. controls (408.91,668.4) and (408.98,666.74) .. (410.72,665.15) .. controls (412.45,663.55) and (412.52,661.88) .. (410.93,660.15) .. controls (409.34,658.41) and (409.41,656.75) .. (411.15,655.16) .. controls (412.89,653.57) and (412.96,651.9) .. (411.37,650.16) .. controls (409.78,648.43) and (409.85,646.76) .. (411.58,645.16) .. controls (413.32,643.57) and (413.39,641.91) .. (411.8,640.17) .. controls (410.21,638.43) and (410.28,636.76) .. (412.02,635.17) .. controls (413.75,633.58) and (413.82,631.91) .. (412.23,630.18) -- (412.35,627.55) -- (412.35,627.55) ;
\draw    (28,818) -- (28,670) ;
\draw [shift={(28,668)}, rotate = 90] [color={rgb, 255:red, 0; green, 0; blue, 0 }  ][line width=0.75]    (10.93,-3.29) .. controls (6.95,-1.4) and (3.31,-0.3) .. (0,0) .. controls (3.31,0.3) and (6.95,1.4) .. (10.93,3.29)   ;
\draw    (28,818) -- (176,818) ;
\draw [shift={(178,818)}, rotate = 180] [color={rgb, 255:red, 0; green, 0; blue, 0 }  ][line width=0.75]    (10.93,-3.29) .. controls (6.95,-1.4) and (3.31,-0.3) .. (0,0) .. controls (3.31,0.3) and (6.95,1.4) .. (10.93,3.29)   ;
\draw [shift={(28,818)}, rotate = 0] [color={rgb, 255:red, 0; green, 0; blue, 0 }  ][fill={rgb, 255:red, 0; green, 0; blue, 0 }  ][line width=0.75]      (0, 0) circle [x radius= 3.35, y radius= 3.35]   ;
\draw    (28.54,758) -- (108.56,757.12) ;
\draw [shift={(108.56,757.12)}, rotate = 359.37] [color={rgb, 255:red, 0; green, 0; blue, 0 }  ][fill={rgb, 255:red, 0; green, 0; blue, 0 }  ][line width=0.75]      (0, 0) circle [x radius= 3.35, y radius= 3.35]   ;
\draw [shift={(28.54,758)}, rotate = 359.37] [color={rgb, 255:red, 0; green, 0; blue, 0 }  ][fill={rgb, 255:red, 0; green, 0; blue, 0 }  ][line width=0.75]      (0, 0) circle [x radius= 3.35, y radius= 3.35]   ;
\draw    (108.01,817.11) -- (108.56,757.12) ;
\draw [shift={(108.56,757.12)}, rotate = 270.52] [color={rgb, 255:red, 0; green, 0; blue, 0 }  ][fill={rgb, 255:red, 0; green, 0; blue, 0 }  ][line width=0.75]      (0, 0) circle [x radius= 3.35, y radius= 3.35]   ;
\draw [shift={(108.01,817.11)}, rotate = 270.52] [color={rgb, 255:red, 0; green, 0; blue, 0 }  ][fill={rgb, 255:red, 0; green, 0; blue, 0 }  ][line width=0.75]      (0, 0) circle [x radius= 3.35, y radius= 3.35]   ;
\draw    (108.56,757.12) -- (108.01,680) ;
\draw [shift={(108,678)}, rotate = 89.6] [color={rgb, 255:red, 0; green, 0; blue, 0 }  ][line width=0.75]    (10.93,-3.29) .. controls (6.95,-1.4) and (3.31,-0.3) .. (0,0) .. controls (3.31,0.3) and (6.95,1.4) .. (10.93,3.29)   ;
\draw    (108.56,757.12) -- (166,757.97) ;
\draw [shift={(168,758)}, rotate = 180.85] [color={rgb, 255:red, 0; green, 0; blue, 0 }  ][line width=0.75]    (10.93,-3.29) .. controls (6.95,-1.4) and (3.31,-0.3) .. (0,0) .. controls (3.31,0.3) and (6.95,1.4) .. (10.93,3.29)   ;
\draw    (316.7,627.55) .. controls (318.37,625.88) and (320.03,625.88) .. (321.7,627.55) .. controls (323.37,629.22) and (325.03,629.22) .. (326.7,627.55) .. controls (328.37,625.88) and (330.03,625.88) .. (331.7,627.55) .. controls (333.37,629.22) and (335.03,629.22) .. (336.7,627.55) .. controls (338.37,625.88) and (340.03,625.88) .. (341.7,627.55) .. controls (343.37,629.22) and (345.03,629.22) .. (346.7,627.55) .. controls (348.37,625.88) and (350.03,625.88) .. (351.7,627.55) .. controls (353.37,629.22) and (355.03,629.22) .. (356.7,627.55) .. controls (358.37,625.88) and (360.03,625.88) .. (361.7,627.55) -- (364.52,627.55) -- (364.52,627.55) .. controls (366.19,625.88) and (367.85,625.88) .. (369.52,627.55) .. controls (371.19,629.22) and (372.85,629.22) .. (374.52,627.55) .. controls (376.19,625.88) and (377.85,625.88) .. (379.52,627.55) .. controls (381.19,629.22) and (382.85,629.22) .. (384.52,627.55) .. controls (386.19,625.88) and (387.85,625.88) .. (389.52,627.55) .. controls (391.19,629.22) and (392.85,629.22) .. (394.52,627.55) .. controls (396.19,625.88) and (397.85,625.88) .. (399.52,627.55) .. controls (401.19,629.22) and (402.85,629.22) .. (404.52,627.55) .. controls (406.19,625.88) and (407.85,625.88) .. (409.52,627.55) .. controls (411.19,629.22) and (412.85,629.22) .. (414.52,627.55) .. controls (416.19,625.88) and (417.85,625.88) .. (419.52,627.55) .. controls (421.19,629.22) and (422.85,629.22) .. (424.52,627.55) .. controls (426.19,625.88) and (427.85,625.88) .. (429.52,627.55) .. controls (431.19,629.22) and (432.85,629.22) .. (434.52,627.55) .. controls (436.19,625.88) and (437.85,625.88) .. (439.52,627.55) .. controls (441.19,629.22) and (442.85,629.22) .. (444.52,627.55) .. controls (446.19,625.88) and (447.85,625.88) .. (449.52,627.55) .. controls (451.19,629.22) and (452.85,629.22) .. (454.52,627.55) .. controls (456.19,625.88) and (457.85,625.88) .. (459.52,627.55) .. controls (461.19,629.22) and (462.85,629.22) .. (464.52,627.55) .. controls (466.19,625.88) and (467.85,625.88) .. (469.52,627.55) .. controls (471.19,629.22) and (472.85,629.22) .. (474.52,627.55) .. controls (476.19,625.88) and (477.85,625.88) .. (479.52,627.55) .. controls (481.19,629.22) and (482.85,629.22) .. (484.52,627.55) .. controls (486.19,625.88) and (487.85,625.88) .. (489.52,627.55) .. controls (491.19,629.22) and (492.85,629.22) .. (494.52,627.55) .. controls (496.19,625.88) and (497.85,625.88) .. (499.52,627.55) .. controls (501.19,629.22) and (502.85,629.22) .. (504.52,627.55) -- (508,627.55) -- (508,627.55) ;
\draw    (504,820) .. controls (502.37,818.3) and (502.4,816.63) .. (504.1,815) .. controls (505.8,813.37) and (505.84,811.7) .. (504.21,810) .. controls (502.58,808.3) and (502.61,806.63) .. (504.31,805) .. controls (506.01,803.37) and (506.05,801.7) .. (504.42,800) .. controls (502.79,798.3) and (502.82,796.64) .. (504.52,795.01) .. controls (506.22,793.38) and (506.25,791.71) .. (504.62,790.01) .. controls (502.99,788.31) and (503.03,786.64) .. (504.73,785.01) .. controls (506.43,783.38) and (506.46,781.71) .. (504.83,780.01) .. controls (503.2,778.31) and (503.24,776.64) .. (504.94,775.01) .. controls (506.64,773.38) and (506.67,771.71) .. (505.04,770.01) .. controls (503.41,768.31) and (503.44,766.64) .. (505.14,765.01) .. controls (506.84,763.38) and (506.88,761.71) .. (505.25,760.01) .. controls (503.62,758.31) and (503.65,756.64) .. (505.35,755.01) .. controls (507.05,753.38) and (507.08,751.72) .. (505.45,750.02) .. controls (503.82,748.32) and (503.86,746.65) .. (505.56,745.02) .. controls (507.26,743.39) and (507.29,741.72) .. (505.66,740.02) .. controls (504.03,738.32) and (504.07,736.65) .. (505.77,735.02) .. controls (507.47,733.39) and (507.5,731.72) .. (505.87,730.02) .. controls (504.24,728.32) and (504.27,726.65) .. (505.97,725.02) .. controls (507.67,723.39) and (507.71,721.72) .. (506.08,720.02) .. controls (504.45,718.32) and (504.48,716.65) .. (506.18,715.02) .. controls (507.88,713.39) and (507.92,711.72) .. (506.29,710.02) .. controls (504.66,708.32) and (504.69,706.65) .. (506.39,705.02) .. controls (508.09,703.39) and (508.12,701.73) .. (506.49,700.03) .. controls (504.86,698.33) and (504.9,696.66) .. (506.6,695.03) .. controls (508.3,693.4) and (508.33,691.73) .. (506.7,690.03) .. controls (505.07,688.33) and (505.11,686.66) .. (506.81,685.03) .. controls (508.51,683.4) and (508.54,681.73) .. (506.91,680.03) .. controls (505.28,678.33) and (505.31,676.66) .. (507.01,675.03) .. controls (508.71,673.4) and (508.75,671.73) .. (507.12,670.03) .. controls (505.49,668.33) and (505.52,666.66) .. (507.22,665.03) .. controls (508.92,663.4) and (508.95,661.73) .. (507.32,660.03) .. controls (505.69,658.33) and (505.73,656.67) .. (507.43,655.04) .. controls (509.13,653.41) and (509.16,651.74) .. (507.53,650.04) .. controls (505.9,648.34) and (505.94,646.67) .. (507.64,645.04) .. controls (509.34,643.41) and (509.37,641.74) .. (507.74,640.04) .. controls (506.11,638.34) and (506.14,636.67) .. (507.84,635.04) .. controls (509.54,633.41) and (509.58,631.74) .. (507.95,630.04) -- (508,627.55) -- (508,627.55) ;
\draw    (108.28,787.12) ;
\draw [shift={(108.28,787.12)}, rotate = 0] [color={rgb, 255:red, 0; green, 0; blue, 0 }  ][fill={rgb, 255:red, 0; green, 0; blue, 0 }  ][line width=0.75]      (0, 0) circle [x radius= 3.35, y radius= 3.35]   ;
\draw    (414,770) .. controls (415.67,768.33) and (417.33,768.33) .. (419,770) .. controls (420.67,771.67) and (422.33,771.67) .. (424,770) .. controls (425.67,768.33) and (427.33,768.33) .. (429,770) .. controls (430.67,771.67) and (432.33,771.67) .. (434,770) .. controls (435.67,768.33) and (437.33,768.33) .. (439,770) .. controls (440.67,771.67) and (442.33,771.67) .. (444,770) .. controls (445.67,768.33) and (447.33,768.33) .. (449,770) .. controls (450.67,771.67) and (452.33,771.67) .. (454,770) .. controls (455.67,768.33) and (457.33,768.33) .. (459,770) .. controls (460.67,771.67) and (462.33,771.67) .. (464,770) .. controls (465.67,768.33) and (467.33,768.33) .. (469,770) .. controls (470.67,771.67) and (472.33,771.67) .. (474,770) .. controls (475.67,768.33) and (477.33,768.33) .. (479,770) .. controls (480.67,771.67) and (482.33,771.67) .. (484,770) .. controls (485.67,768.33) and (487.33,768.33) .. (489,770) .. controls (490.67,771.67) and (492.33,771.67) .. (494,770) .. controls (495.67,768.33) and (497.33,768.33) .. (499,770) .. controls (500.67,771.67) and (502.33,771.67) .. (504,770) -- (504,770) ;
\draw  [dash pattern={on 0.75pt off 0.75pt}]  (424,770) .. controls (422.33,768.33) and (422.33,766.67) .. (424,765) .. controls (425.67,763.33) and (425.67,761.67) .. (424,760) .. controls (422.33,758.33) and (422.33,756.67) .. (424,755) .. controls (425.67,753.33) and (425.67,751.67) .. (424,750) .. controls (422.33,748.33) and (422.33,746.67) .. (424,745) .. controls (425.67,743.33) and (425.67,741.67) .. (424,740) .. controls (422.33,738.33) and (422.33,736.67) .. (424,735) .. controls (425.67,733.33) and (425.67,731.67) .. (424,730) .. controls (422.33,728.33) and (422.33,726.67) .. (424,725) .. controls (425.67,723.33) and (425.67,721.67) .. (424,720) -- (424,720) ;
\draw [color={rgb, 255:red, 0; green, 0; blue, 0 }  ,draw opacity=1 ] [dash pattern={on 0.75pt off 0.75pt}]  (434,770) .. controls (432.33,768.33) and (432.33,766.67) .. (434,765) .. controls (435.67,763.33) and (435.67,761.67) .. (434,760) .. controls (432.33,758.33) and (432.33,756.67) .. (434,755) .. controls (435.67,753.33) and (435.67,751.67) .. (434,750) .. controls (432.33,748.33) and (432.33,746.67) .. (434,745) .. controls (435.67,743.33) and (435.67,741.67) .. (434,740) .. controls (432.33,738.33) and (432.33,736.67) .. (434,735) .. controls (435.67,733.33) and (435.67,731.67) .. (434,730) .. controls (432.33,728.33) and (432.33,726.67) .. (434,725) .. controls (435.67,723.33) and (435.67,721.67) .. (434,720) -- (434,720) ;
\draw [color={rgb, 255:red, 0; green, 0; blue, 0 }  ,draw opacity=1 ] [dash pattern={on 0.75pt off 0.75pt}]  (444,770) .. controls (442.33,768.33) and (442.33,766.67) .. (444,765) .. controls (445.67,763.33) and (445.67,761.67) .. (444,760) .. controls (442.33,758.33) and (442.33,756.67) .. (444,755) .. controls (445.67,753.33) and (445.67,751.67) .. (444,750) .. controls (442.33,748.33) and (442.33,746.67) .. (444,745) .. controls (445.67,743.33) and (445.67,741.67) .. (444,740) .. controls (442.33,738.33) and (442.33,736.67) .. (444,735) .. controls (445.67,733.33) and (445.67,731.67) .. (444,730) .. controls (442.33,728.33) and (442.33,726.67) .. (444,725) .. controls (445.67,723.33) and (445.67,721.67) .. (444,720) -- (444,720) ;
\draw  [dash pattern={on 0.75pt off 0.75pt}]  (454,770) .. controls (452.33,768.33) and (452.33,766.67) .. (454,765) .. controls (455.67,763.33) and (455.67,761.67) .. (454,760) .. controls (452.33,758.33) and (452.33,756.67) .. (454,755) .. controls (455.67,753.33) and (455.67,751.67) .. (454,750) .. controls (452.33,748.33) and (452.33,746.67) .. (454,745) .. controls (455.67,743.33) and (455.67,741.67) .. (454,740) .. controls (452.33,738.33) and (452.33,736.67) .. (454,735) .. controls (455.67,733.33) and (455.67,731.67) .. (454,730) .. controls (452.33,728.33) and (452.33,726.67) .. (454,725) .. controls (455.67,723.33) and (455.67,721.67) .. (454,720) -- (454,720) ;
\draw  [dash pattern={on 0.75pt off 0.75pt}]  (464,770) .. controls (462.33,768.33) and (462.33,766.67) .. (464,765) .. controls (465.67,763.33) and (465.67,761.67) .. (464,760) .. controls (462.33,758.33) and (462.33,756.67) .. (464,755) .. controls (465.67,753.33) and (465.67,751.67) .. (464,750) .. controls (462.33,748.33) and (462.33,746.67) .. (464,745) .. controls (465.67,743.33) and (465.67,741.67) .. (464,740) .. controls (462.33,738.33) and (462.33,736.67) .. (464,735) .. controls (465.67,733.33) and (465.67,731.67) .. (464,730) .. controls (462.33,728.33) and (462.33,726.67) .. (464,725) .. controls (465.67,723.33) and (465.67,721.67) .. (464,720) -- (464,720) ;
\draw [color={rgb, 255:red, 0; green, 0; blue, 0 }  ,draw opacity=1 ] [dash pattern={on 0.75pt off 0.75pt}]  (474,770) .. controls (472.33,768.33) and (472.33,766.67) .. (474,765) .. controls (475.67,763.33) and (475.67,761.67) .. (474,760) .. controls (472.33,758.33) and (472.33,756.67) .. (474,755) .. controls (475.67,753.33) and (475.67,751.67) .. (474,750) .. controls (472.33,748.33) and (472.33,746.67) .. (474,745) .. controls (475.67,743.33) and (475.67,741.67) .. (474,740) .. controls (472.33,738.33) and (472.33,736.67) .. (474,735) .. controls (475.67,733.33) and (475.67,731.67) .. (474,730) .. controls (472.33,728.33) and (472.33,726.67) .. (474,725) .. controls (475.67,723.33) and (475.67,721.67) .. (474,720) -- (474,720) ;
\draw [color={rgb, 255:red, 208; green, 2; blue, 27 }  ,draw opacity=1 ]   (482.5,770) .. controls (480.83,768.33) and (480.83,766.67) .. (482.5,765) .. controls (484.17,763.33) and (484.17,761.67) .. (482.5,760) .. controls (480.83,758.33) and (480.83,756.67) .. (482.5,755) .. controls (484.17,753.33) and (484.17,751.67) .. (482.5,750) .. controls (480.83,748.33) and (480.83,746.67) .. (482.5,745) .. controls (484.17,743.33) and (484.17,741.67) .. (482.5,740) .. controls (480.83,738.33) and (480.83,736.67) .. (482.5,735) .. controls (484.17,733.33) and (484.17,731.67) .. (482.5,730) .. controls (480.83,728.33) and (480.83,726.67) .. (482.5,725) .. controls (484.17,723.33) and (484.17,721.67) .. (482.5,720) -- (482.5,720)(485.5,770) .. controls (483.83,768.33) and (483.83,766.67) .. (485.5,765) .. controls (487.17,763.33) and (487.17,761.67) .. (485.5,760) .. controls (483.83,758.33) and (483.83,756.67) .. (485.5,755) .. controls (487.17,753.33) and (487.17,751.67) .. (485.5,750) .. controls (483.83,748.33) and (483.83,746.67) .. (485.5,745) .. controls (487.17,743.33) and (487.17,741.67) .. (485.5,740) .. controls (483.83,738.33) and (483.83,736.67) .. (485.5,735) .. controls (487.17,733.33) and (487.17,731.67) .. (485.5,730) .. controls (483.83,728.33) and (483.83,726.67) .. (485.5,725) .. controls (487.17,723.33) and (487.17,721.67) .. (485.5,720) -- (485.5,720) ;
\draw  [dash pattern={on 0.75pt off 0.75pt}]  (494,770) .. controls (492.33,768.33) and (492.33,766.67) .. (494,765) .. controls (495.67,763.33) and (495.67,761.67) .. (494,760) .. controls (492.33,758.33) and (492.33,756.67) .. (494,755) .. controls (495.67,753.33) and (495.67,751.67) .. (494,750) .. controls (492.33,748.33) and (492.33,746.67) .. (494,745) .. controls (495.67,743.33) and (495.67,741.67) .. (494,740) .. controls (492.33,738.33) and (492.33,736.67) .. (494,735) .. controls (495.67,733.33) and (495.67,731.67) .. (494,730) .. controls (492.33,728.33) and (492.33,726.67) .. (494,725) .. controls (495.67,723.33) and (495.67,721.67) .. (494,720) -- (494,720) ;
\draw  [dash pattern={on 0.75pt off 0.75pt}]  (504,770) .. controls (502.33,768.33) and (502.33,766.67) .. (504,765) .. controls (505.67,763.33) and (505.67,761.67) .. (504,760) .. controls (502.33,758.33) and (502.33,756.67) .. (504,755) .. controls (505.67,753.33) and (505.67,751.67) .. (504,750) .. controls (502.33,748.33) and (502.33,746.67) .. (504,745) .. controls (505.67,743.33) and (505.67,741.67) .. (504,740) .. controls (502.33,738.33) and (502.33,736.67) .. (504,735) .. controls (505.67,733.33) and (505.67,731.67) .. (504,730) .. controls (502.33,728.33) and (502.33,726.67) .. (504,725) .. controls (505.67,723.33) and (505.67,721.67) .. (504,720) -- (504,720) ;
\draw  [color={rgb, 255:red, 208; green, 2; blue, 27 }  ,draw opacity=1 ] (326.32,780.04) .. controls (349.72,660.05) and (372.96,840.06) .. (396.35,720.07) ;
\draw  [color={rgb, 255:red, 208; green, 2; blue, 27 }  ,draw opacity=1 ] (333.85,626.83) .. controls (358.9,810.94) and (371.29,535.94) .. (396.34,720.05) ;
\draw  [color={rgb, 255:red, 208; green, 2; blue, 27 }  ,draw opacity=1 ] (324,790) .. controls (328.08,779.75) and (331.98,770) .. (336.5,770) .. controls (341.02,770) and (344.92,779.75) .. (349,790) .. controls (353.08,800.25) and (356.98,810) .. (361.5,810) .. controls (366.02,810) and (369.92,800.25) .. (374,790) .. controls (378.08,779.75) and (381.98,770) .. (386.5,770) .. controls (391.02,770) and (394.92,779.75) .. (399,790) .. controls (400.68,794.22) and (402.33,798.35) .. (404,801.76) ;
\draw  [color={rgb, 255:red, 208; green, 2; blue, 27 }  ,draw opacity=1 ] (406.5,795) .. controls (416.98,782.19) and (426.98,770) .. (431.5,770) .. controls (436.02,770) and (433.83,782.19) .. (431.5,795) .. controls (429.17,807.81) and (426.97,820) .. (431.49,820) .. controls (436.02,820) and (446.02,807.81) .. (456.5,795) .. controls (466.98,782.19) and (476.98,770) .. (481.5,770) .. controls (486.02,770) and (483.83,782.19) .. (481.5,795) .. controls (479.61,805.4) and (477.8,815.39) .. (479.61,818.78) ;
\draw [color={rgb, 255:red, 208; green, 2; blue, 27 }  ,draw opacity=1 ]   (334,670) .. controls (374,640) and (464,700) .. (504,670) ;
\draw [color={rgb, 255:red, 208; green, 2; blue, 27 }  ,draw opacity=1 ]   (483.1,718.8) .. controls (491.25,712.69) and (494.31,706.49) .. (494.31,700.18) .. controls (494.31,691.69) and (489.03,683.13) .. (483.85,674.55) .. controls (478.35,665.45) and (472.95,656.33) .. (472.95,647.36) .. controls (472.95,640.3) and (476.16,633.25) .. (485.36,626.35)(484.9,721.2) .. controls (494.1,714.3) and (497.31,707.25) .. (497.31,700.18) .. controls (497.31,691.21) and (491.91,682.1) .. (486.41,673) .. controls (481.23,664.42) and (475.95,655.85) .. (475.95,647.36) .. controls (475.95,641.05) and (479.01,634.86) .. (487.16,628.75) ;
\draw [color={rgb, 255:red, 208; green, 2; blue, 27 }  ,draw opacity=1 ]   (340,660) ;
\draw [shift={(340,660)}, rotate = 0] [color={rgb, 255:red, 208; green, 2; blue, 27 }  ,draw opacity=1 ][fill={rgb, 255:red, 208; green, 2; blue, 27 }  ,fill opacity=1 ][line width=0.75]      (0, 0) circle [x radius= 3.35, y radius= 3.35]   ;
\draw [color={rgb, 255:red, 208; green, 2; blue, 27 }  ,draw opacity=1 ]   (370,760) ;
\draw [shift={(370,760)}, rotate = 0] [color={rgb, 255:red, 208; green, 2; blue, 27 }  ,draw opacity=1 ][fill={rgb, 255:red, 208; green, 2; blue, 27 }  ,fill opacity=1 ][line width=0.75]      (0, 0) circle [x radius= 3.35, y radius= 3.35]   ;
\draw [color={rgb, 255:red, 208; green, 2; blue, 27 }  ,draw opacity=1 ]   (470,780) ;
\draw [shift={(470,780)}, rotate = 0] [color={rgb, 255:red, 208; green, 2; blue, 27 }  ,draw opacity=1 ][fill={rgb, 255:red, 208; green, 2; blue, 27 }  ,fill opacity=1 ][line width=0.75]      (0, 0) circle [x radius= 3.35, y radius= 3.35]   ;

\draw (125,781) node [anchor=north west][inner sep=0.75pt]   [align=left] {\textsf{Tube vertex $v$}};
\draw (515,742) node [anchor=north west][inner sep=0.75pt]   [align=left] {\textsf{Tube component $\mathcal Y_v$}};
\end{tikzpicture}
\caption{An $1$-complex with a tube vertex on the left and the corresponding tube component on the right. The subscheme depicted in that tube component is a tube subscheme -- it is attempting to depict the preimage of a $0$-dimensional subscheme of length $2$, along the projection. On the figure on the right, the circle in the bottom left indicates the main component of the expansion, while the rest of the components are bundles over strata. The tube component contains a tube subscheme, and the picture is meant to depict a subscheme satisfying DT stability.}\label{fig: tube-picture}
\end{figure}

\subsection{Geometry of the space of relative ideal sheaves} Let $\mathsf{DT}(X|D)$ denote the fibered category over schemes of families of stable ideal sheaves on $X$ relative to $D$. We prove that the moduli problem is algebraic with finite stabilizers.

\begin{theorem}
For each fixed stack $\mathsf{Exp}(X|D)$, the fibered category $\mathsf{DT}(X|D)$ is an algebraic stack.
\end{theorem}

The stack comes with a map to $\mathsf{Exp}(X|D)$, and we equip it with the pullback logarithmic structure. 

\begin{proof}
The proof is broken up into four steps; we prove algebraicity, first for the stack of expansions, then for the relative Hilbert scheme on the universal family of the stack of expansions, then for the substack of algebraically transverse maps, and finally for the substack of transverse maps satisfying the DT stability condition.  

First, the stack $\mathsf{Exp}(X|D)$ is an Artin fan, and therefore is algebraic. 
Second, since the map $\mathcal Y \rightarrow \mathsf{Exp}(X|D)$ is separated,
it follows from ~\cite[\href{https://stacks.math.columbia.edu/tag/0D01}{Tag 0D01}]{stacks-project} that the Hilbert functor $\mathsf{Hilb}(\mathcal Y/\mathsf{Exp}(X|D))$ is representable by an algebraic space over $\mathsf{Exp}(X|D)$.  

We next consider algebraic transversality.  Denote by $\mathsf{Hilb}^\dagger(\mathcal Y/\mathsf{Exp}(X|D))$ the open subspace parameterizing algebraically transverse subschemes, without the DT stability condition. 
Suppose we are given a family of expansions $\mathcal Y_S \rightarrow S$ and a family of subschemes $Z_S \subset Y_S$.   Since algebraic transversality is characterized by the vanishing of higher Tor's, it is a constructible condition. Therefore it suffices to show it is stable under generization, so we assume $S$ is a DVR.  Recall that any expansion is naturally stratified, since it arises as the special fiber of a logarithmically smooth degeneration. Given any expansion over $S$, the closure of any stratum lying over the generic fiber contains strata of the closed fiber.  Therefore, by flatness, every stratum of the generic fiber must intersect $Z$ non-trivially.  Similarly, by upper semi-continuity, the Tor-vanishing condition for divisor strata of the generic fiber follows from the same condition on the closed fiber (by choosing a divisor stratum of the closed fiber in the closure). 

Finally, we deduce openness of the DT stability condition from the following lemmas.
\end{proof}

\begin{lemma}
The DT stability condition is open and therefore $DT(X|D)\subset \mathsf{Hilb}(\mathcal Y/\mathsf{Exp}(X|D))$ is an open algebraic substack.
\end{lemma}

\begin{proof}
As before, we need to check that the condition of DT stability is preserved by generization. Let $B$ be the spectrum of a discrete valuation ring with generic point $\eta$ and closed point $b$ and consider a morphism $B\to \mathsf{Hilb}^\dagger(\mathcal Y/\mathsf{Exp}(X|D))$, to the locus in the relative Hilbert scheme parameterizing algebraically transverse subschemes. Assume that DT stability is satisfied at $0$. We will prove that it is satisfied at the generic point. 

Examine the map $B\to \mathsf{Exp}(X|D)$. The generic and closed point each determine strata of the latter space. Let $F_\eta$ and $F_b$ be the corresponding cones in $T$; note that $F_\eta$ is a face of $F_b$. Each component determines a $1$-complex embedded in $\Sigma_X$ and by slight abuse of notation, we let $G_\eta$ and $G_b$ be the combinatorial types of the associated graphs and note that there is a specialization map\footnote{The reader may wish to have in mind that $G_\eta$ typically simpler than $G_b$, in the sense of having fewer vertices.} from $G_b$ to $G_\eta$. To see this specialization, we may choose particular metric structures on $G_b$ and $G_\eta$ such that there is a morphism from a polyhedral complex $\mathcal G\to \RR_{\geq 0}$, with $\mathcal G\subset \Sigma_X\times \RR_{\geq 0}$ whose fibers over $0$ and $1$ are $G_0$ and $G_1$ respectively. The combinatorial types of $G_1$ and $G_b$ agree, and the combinatorial types of $G_0$ and $G_\eta$ also agree.The specialization map is obtained by taking the cone over a face in $G_1$ and intersecting with the $0$ fiber, and then making the natural identifications with $G_b$ and $G_\eta$.

Let $\mathcal Z$ be the family of subschemes over $B$ with generic and special fibers $\mathcal Z_\eta$ and $\mathcal Z_b$. By the DT stability condition we see that at the closed point $b$, the subscheme $Z_b$ is a tube along the tube vertices of $G_b$. We wish to show that the same is true at $\eta$. 

We claim that, if $v$ is a tube vertex of $G_\eta$, then all vertices $w$ of $G_b$ that specialize to $v$ are also tube vertices. Indeed, tube vertices arise by adding $2$-valent vertices to subdivide the minimal graph structure on $G_b$ and then specializing these vertices to obtain a subdivision of the minimal structure on $G_\eta$. It follows that the set of vertices specializing to a tube vertex must all be tube vertices. 

Let $v$ be a tube vertex of $G_\eta$. The star around $v$ in $G_\eta$ can be identified with the fan of $\mathbb P^1$. Let $w_1,\ldots, w_k$ be the vertices of $G_b$ specializing to $v$. Since tube vertices are always $2$-valent, the star of $v$ in the total space $\mathcal G$ can be identified with the cone over a polyhedral subdivision of $\RR$ with $k$ vertices. The toric dictionary provides us with the following structure for the component $Y_v$ dual to $v$ over $B$. We have morphisms
\[
Y_v\to S\to B,
\]
where $S$ is a smooth surface over $B$ and $Y_v\to S$ is a proper family of semistable genus $0$ curves that is generically a $\mathbb P^1$ bundle -- that is, $X_\eta\to S_\eta$ is a $\mathbb P^1$-bundle. The surface $S$ is obtained as a locally closed stratum in a logarithmic modification of $X$, and in particular, is typically non-proper. Let $Y_b$ be the union of the irreducible components corresponding to the vertices $w_i$. 

After replacing our degeneration with this local model, we are left to check that if $Z_b$ is a tube subscheme on every component of $Y_b$, then $\mathcal Z_\eta$ is also a tube subscheme. This is established by the lemma below, which then concludes the proof.
\end{proof}

\begin{lemma}
As above let 
\[
Y_v\to S\to B,
\]
be a generically smooth semistable genus $0$ curve fibration over a family of smooth surfaces over $B$. Let $Z_b$ and $Z_\eta$ be the special and generic subschemes obtained by restriction to the local model as above.  If $Z_b\subset Y_b$ is the preimage of a zero-dimensional subscheme  $W_b \subset S_b$, then there exists a subscheme $W\subset S$ such that $Z$ is the preimage of $W$.
\end{lemma}

\begin{proof}
By applying the criteria for properness and transversality from Theorems~\ref{thm: properness-criterion} and~\ref{thm: transversality-criterion}, we see that the family of subschemes obtained by the construction above to the component above is flat and proper. Properness and upper-semicontinuity of dimension for the family of subschemes guarantees that the image of $Z_\eta$ in $S_\eta$ is zero-dimensional. Since the special fiber $Z_b$ is the preimage of a zero dimensional subscheme along a nodal curve fibration, the special fiber has no embedded points. The Cohen-Macaulay condition is open, so the generic fiber also has no embedded points~\cite[\href{https://stacks.math.columbia.edu/tag/0E0H}{Tag 0E0H}]{stacks-project}. 

We examine the Hilbert polynomial of the subscheme. By flatness of the family of subschemes, the fiber degree of the subscheme is constant. We therefore conclude that the subscheme $Z_\eta$ is a Cohen--Macaulay thickening of a union of fibers in a $\mathbb P^1$-bundle over $S_\eta$. For each such fiber, the thickening gives rise to a morphism from $\mathbb P^1$ to the Hilbert scheme of points $\mathsf{Hilb}(S_\eta,d)$ on the surface $S_\eta$. The minimal possible holomorphic Euler characteristic of this thickening occurs if and only if the subscheme is a tube. Indeed, if the subscheme is not a tube, then the morphism $\mathbb P^1\to \mathsf{Hilb}(S_\eta,d)$ is non-constant. The difference of the Euler characteristic and the degree of this moduli map is constant, from which the claim follows. 
\end{proof}


\subsection{Boundedness}\label{sec: boundedness} Fix a curve class $\beta$ and holomorphic Euler characteristic $n$ and consider the moduli space $\mathsf{DT}_{\beta,n}(X|D)$ of ideal sheaves on $X$ relative to $D$ with these \textit{numerical data}. Specifically, we consider subschemes for which the pushforward of the fundamental cycle 
to $X$ is $\beta$.  We establish boundedness of the moduli problem. 

\begin{theorem}
There exists a finite type scheme $S$ and a map
\[
S\to \mathsf{DT}_{\beta,n}(X|D)
\]
which exhausts the geometric points of the target. 
\end{theorem}

\begin{proof}
We prove this in two lemmas. First, in Lemma~\ref{lem: comb-boundedness}, we bound the number of tropical types that can arise by examining points of $\mathsf{DT}_{\beta,n}(X|D)$. Second, in Lemma~\ref{lem: relative-boundedness}, we
exhaust the locus of transverse subschemes in a fixed degenerate target that give a fixed tropicalization. \end{proof}

Let $G$ and $G'$ be embedded $1$-complexes in $\Sigma_X$. We say that $G$ and $G'$ \textit{have the same type} if they lie in the interior of the same cone of $T(\Sigma_X)$. 

\begin{lemma}\label{lem: comb-boundedness}
There are finitely many combinatorial types of embedded $1$-complexes in $\Sigma_X$ that arise from ideal sheaves in $X$ relative to $D$ with fixed curve class $\beta$ and Euler characteristic $n$. That is, the moduli map
\[
DT_{\beta,n}(X|D)\to \mathsf{Exp}(X|D)
\] 
factors through a finite type substack of the target. 
\end{lemma}

\begin{proof} The blueprint for this result on the Gromov--Witten side is~\cite[Section~3.1]{GS13}. The proof bounds the possible slopes of edges of $1$-complexes embedded in $\Sigma_X$ that can arise from subschemes with given discrete data. This is possible by means of the balancing condition, which we formulate below, together with known finiteness results for balanced tropical curves. 

\subsection*{\sc Step I. Embedded tropical curves} Let ${\mathcal Z}\subset \mathcal Y$ be an algebraically transverse subscheme of an expansion of $X$ along $D$, defined over a point. This point is equipped with a logarithmic structure pulled back from $\mathsf{Exp}(X|D)$. It has characteristic monoid $P$. After pulling back the family to a standard logarithmic point via a sharp map $P\to \NN$, we obtain a $1$-complex $\mathbb G_{\mathcal Z}$ embedded in $\Sigma_X$. We want to bound the combinatorial possibilities for these $1$-complexes. The first step is to observe that $1$-complexes coming from such ${\mathcal Z}\subset\mathcal Y$ can be decorated with additional information.

First, decorate each vertex vertex $v$ with the homology class of the associated subscheme $Z_v$ in the homology of $X$. To be precise, we note that there is a proper cycle $Z_v$ in the component dual to $v$. Its pushforward in $X$ defines a class in $H_2(X)$. Denote this homology class at a vertex $v$ by $\beta_v$. In particular, note that vertices such that the subscheme $Z_v$ lies entirely in fiber of the projection $\mathcal Y\to X$ are decorated with $0$. 

Next, let $v$ be a vertex of $\mathbb G_{\mathcal Z}$ and $e$ an edge or ray incident to $v$. Note that $e$ is dual to a codimension $1$ stratum of $\mathcal Y$. The intersection of ${\mathcal Z}$ with this stratum is a zero dimensional subscheme. Define $w_e$ to be the length of this subscheme. Since the subschemes on either side of a double divisor $\mathcal S$ share the same intersection with $\mathcal S$, if $e$ is an edge rather than a ray, the number $w_e$ depends only on $e$ and not on $v$. Range over all $v$ to determine decorations for all edges and rays in $\mathbb G_{\mathcal Z}$. 

As discussed previously, $\Sigma_X$ can be embedded as a subcomplex of $\RR_{\geq 0}^k$ -- the latter with its standard cone structure -- for $k$ equal to the number of irreducible components in $D$. We obtain a $1$-complex with vertex and edge decorations in $\RR_{\geq 0}^k\subset \mathbb R^k$. We refer to this object as an \textit{embedded tropical curve}. We abuse notation slightly and continue to denote it by $\mathbb G_{\mathcal Z}$, even after adding the decorations.

\subsection*{\sc Step II. Piecewise linear functions} The group of integral linear functions on $\RR_{\geq 0}^k$ is naturally identified with $\mathbb Z^k$. Indeed, given a set of integers $(a_1,\ldots, a_k)$ we can consider the function whose slope along the $i^{\mathsf{th}}$ coordinate axis is $a_i$. Every linear function in $\mathbb Z^k$ restricts to a function on $\Sigma_X$ that is piecewise linear, and linear on each cone\footnote{Note we use piecewise linear, rather than linear, as the adjective for these functions. The reason is that complexes come with a natural notion of piecewise linear function, but no natural notion of linear function, unlike fans in toric geometry. The embedding of $\Sigma_X$ into $\RR_{\geq 0}^k$ endows it with a notion of linear functions, but a different embedding would lead to a different notion of linear. }. A piecewise linear function, in this sense, gives rise to a Cartier divisor on $X$. If $\alpha$ is a piecewise linear function, we let $(\mathcal O_X(-\alpha),s_\alpha)$ be the associated Cartier divisor\footnote{The notation is chosen to be consistent with convention in Fulton's text on toric varieties~\cite{Ful93}. Specifically, on the toric variety $\PP^1$, the function that has slope $1$ along a ray and slope $0$ along the other ray is $\mathcal O_{\PP^1}(-1)$}. 

We can make sense of this construction on any component $\mathcal Y_v$ of the expansion as well. Let $v$ be a vertex of  $\mathbb G_{\mathcal Z}$. Restrict $\alpha$ to the star $\mathbb G_{\mathcal Z}(v)$ of $v$ in $\mathbb G_{\mathcal Z}$. The result is a piecewise \textit{affine} function at the star of $v$; subtract the constant $\alpha(v)$ to obtain a piecewise linear function on the (zero or) one-dimensional fan $\mathbb G_{\mathcal Z}(v)$. The construction above again gives a line bundle and a section on the associated component of the expansion. 

\subsection*{\sc Step III. The balancing condition} Fix a linear function $\alpha\in \ZZ^k$. By the procedure above, this determines a Cartier divisor $(\mathcal O_{\mathcal Y_v}(-\alpha),s_{v,\alpha})$ on the component $\mathcal Y_v$ of the expansion, together with a canonical section. Let $Z_v$ be the subscheme in the component $\mathcal Y_v$. The divisor associated to this piecewise linear function can be intersected with $Z_v$ to give a homomorphism:
\begin{eqnarray*}
\varphi_v: \ZZ^k&\to&  \ZZ\\
 \alpha &\mapsto&\mathsf{deg}( [Z_v]\cap c_1(\mathcal O_{\mathcal Y_v}(-\alpha))).
\end{eqnarray*}
The ambient space of the embedded tropical curve is canonically identified with the real positive dual of this copy of $\NN^k$. In other words, we have
\[
\mathbb G_{\mathcal Z}\hookrightarrow \RR_{\geq 0}^k = \Hom_{\mathbf{Mon}}(\NN^k,\RR_{\geq 0}).
\]
The homomorphism $\varphi_v$ above is an integral element in the ambient cone  $\RR_{\geq 0}^k$; we use this to enhance $\mathbb G_{\mathcal Z}$ by adding yet more decoration. We may consider the negative of $\varphi_v$ as a lattice point in the ambient vector space $\RR^k$. To each vertex $v$, appended a ray $e_v$ starting at $v$ in the direction of $-\varphi_v$; we call these \textit{appended rays}. The vector $-\varphi_v$ is a positive integral multiple of the primitive generator in the direction of $e_v$, and we equip these rays with a weight equal to this positive multiple.  

Finally, we call the resulting object \textit{the balanced tropical curve associated to ${\mathcal Z}$}. We view it as a new graph $\widehat{\mathbb G}_{\mathcal Z}$ equipped with an immersion into $\RR^k$ that restricts to an embedding on the $1$-complex $\mathbb G_{\mathcal Z}$ into the standard quadrant.

The term \textit{balanced} is justified as follows. When equipped with these new edges and their weight decorations, the sum of the outgoing vectors at each vertex, weighted by the edge decoration, is zero. Explicitly, we have at each vertex $v$:
\[
\sum_{e\in E(\mathbb G_{\mathcal Z})\colon v\prec e} w_e\cdot \vec e = \varphi_v,
\]
where (i) the sum is taken over all edges $e$ in the edge set of $\mathbb G_{\mathcal Z}$ that are incident to $v$, and (ii) the symbol $\vec e$ denotes the primitive integral vector along $e$, directed away from the $v$.

The balancing condition follows from geometric considerations -- we calculate the degree on $Z_v$ of the line bundles associated to piecewise linear functions $\alpha\in\mathbb Z^k$ in two different ways, as follows. The intersections of $Z_v$ with the divisors $\mathsf{div}(s_{v,\alpha})$ are algebraically transverse, so we may calculate the numbers on the right hand side of the balancing equation:
\[
[Z_v]\cap c_1(\mathcal O_{\mathcal Y_v}(-\alpha))
\]
by explicitly calculating lengths of scheme theoretic intersections. It suffices to check the balancing condition coordinate-by-coordinate in $\ZZ^k$. If we let $\varphi_i$ denote the piecewise linear function corresponding to the $i^{\mathsf{th}}$ divisor, the sum of the $i^{\textsf{th}}$-coordinate components of the weighted edge directions incident to $v$. This gives left hand side of the equation above. Ranging over over all $i$ gives the equation. 

In particular, if $\beta_v$ is equal to $0$, then there is no appended ray $e_v$, and the balancing condition simply states that the weighted of edges incident to a vertex in $\mathbb G_{\mathcal Z}$ is equal to $0$. More generally, by the formula above, the direction and weight of the appended ray is determine by the intersection numbers $\beta_v\cdot D_i$ over all $i$. 

\subsection*{\sc Step IV. Bounding the balanced curves} We now examine the asymptotic structure of $\widehat{\mathbb G}_{\mathcal Z}$. Since $\beta$ is fixed and equal to an effective sum of the $\beta_v$, there are only finitely many vertices of $\mathbb G_{\mathcal Z}$ that can have non-zero curve class decoration. From the discussion above, for any such vertex $v$ with $\beta_v$ nonzero, there is a uniquely determined appended ray $e_v$ with a weight decoration that makes the balancing condition hold. As a consequence, there are finitely many possibilities for the appended rays. 

The unbounded rays in $\mathbb G_{\mathcal Z}$ itself, rather than the enhanced $\widehat{\mathbb G}_{\mathcal Z}$, are each parallel to one of the axes in $\mathbb R_{\geq 0}^k$.  The sum of the weights of rays parallel to the $i^{\mathsf{th}}$-ray is $\beta\cdot D_i$, so there are finitely many possibilities for the unbounded edges of $\widehat{\mathbb G}_{\mathcal Z}$. Any unbounded ray in $\widehat{\mathbb G}_{\mathcal Z}$ is either an appended ray or an unbounded in $\mathbb G_{\mathcal Z}$. 

Putting the two preceding paragraphs together, there are only finitely many combinatorial possibilities for the asymptotic structure, meaning, outside of some large compact neighborhood of the origin, of the balanced tropical curve $\widehat{\mathbb G}_{\mathcal Z}$ 

We now claim that with these data fixed there are only finitely many balanced tropical curves $\widehat{\mathbb G}_{\mathcal Z}$ that can appear for fixed discrete data. Indeed, Nishinou--Siebert present a Newton polygon argument to deduce this in the course of the proof of~\cite[Proposition~2.1]{NS06} which assures this.\footnote{Alternatively, a balanced tropical curve in $\mathbb R^n$ is a polyhedral subcomplex of the dual tropical hypersurface of a subdivision of the associated Chow polytope, see~\cite[End of Section~4.4]{MS14}. The types are therefore bounded by the number of coherent subdivisions of the Chow polytope. See Bertrand--Brugalle--Lopez de Medrano~\cite{BBLdM} and Yu~\cite[Proposition~4.1]{Yu14} for related boundedness results.} 

The key consequence of the Nishinou--Siebert argument for us is that the underlying set of every balanced tropical curve in $\mathbb R^n$ can be obtained as a subset of a complete intersection of $n-1$ tropical hypersurfaces. Moreover, each such tropical hypersurface is the preimage of a balanced tropical curve in $\mathbb R^2$ under an integer linear projection. The number of combinatorial types tropical curves in $\mathbb R^2$ of given asymptotics is bounded by the number of coherent subdivisions of the dual lattice polytope of this projected tropical curve. The lattice polytope is determined by the images of the asymptotic directions under the projection, of which there are certainly a finite number of possibilities. See~\cite[Section~3]{Mi03} for a detailed discussion of the plane curve case.

The argument above immediately implies that, if we range over all $\widehat {\mathbb G}_{\mathcal Z}$, and ignore $2$-valent vertices and all isolated vertices $v$, there are finitely many possibilities for the combinatorial type of $\widehat {\mathbb G}_{\mathcal Z}$. The number of isolated points in the tropicalization is bounded because we have fixed the the curve class and holomorphic Euler characteristic. By the DT stability condition, non-tube $2$-valent vertices cannot contain a tube subscheme, so contribute to either the curve class or increase the Euler characteristic.  Therefore there are only finitely many such $2$-valent vertices. The number of tube vertices only depends on the combinatorial type determined in the previous paragraph. We conclude that the number of combinatorial types of embedded $1$-complexes arising from ideal sheaves is bounded after fixing the numerical data.
\end{proof}

The following remark should give some intuition for what an expansion looks like, which might be valuable in navigating the arguments that follow. 

\begin{remark}[The shape of a component in an expansion]
Let us work in dimension $3$ for concreteness. Let $\mathcal Y$ be an expansion of a threefold $X$. The top dimensional components of $\mathcal Y$ have the following four forms (i)  toroidal modification of $X$; there is at most one such component, (ii) a compactified $\mathbb G_m$-bundle over a locally closed boundary surface in $X$, (iii) a compactified $\mathbb G_m^2$-bundle over a boundary curve in $X$, or (iv) a toric threefold. 
\end{remark}

To conclude boundedness we claim that, after fixing the Hilbert polynomial with respect to the pullback of a polarization from $X$, the space of subschemes whose expansion has a given combinatorial type with numerical data fixed (i.e. subschemes with a fixed tropicalization) is parameterized by a quasi-projective scheme of finite type. 

Let us briefly explain the nature of this claim. The polarization is pulled back from $X$, it is typically not ample on an expansion, and we may have subschemes in these components that only contribute to the Hilbert polynomial via their holomorphic Euler characteristic. A basic principle is that the higher the codimension of the stratum of $X$ over which a component lives, the less control the Hilbert polynomial provides in the choice of subscheme in that component. The extreme case occurs where the stratum in question is a toric variety $Y$ that is collapsed to a point on projection to $X$. The problem here reduces to examining the moduli of subschemes of $Y$ with fixed contact orders along the toric boundary. But while the Hilbert polynomial is not a strong constraint, the balancing condition compensates for this. It constrains the contact orders of the subscheme, and thereby constrains the homology class of the subscheme.\footnote{This problem, with slightly different hypotheses, was considered by Katz--Payne, and their~\cite[Theorem~1.2]{KP11} can be seen as a manifestation of the same phenomenon.} 

\begin{lemma}\label{lem: relative-boundedness}
The stack of ideal sheaves on $X$ relative to $D$ with fixed combinatorial type and fixed numerical data is bounded. 
\end{lemma}

\begin{proof}
Fix an expansion $\mathcal Y$ and examine the irreducible components. We study the subschemes in an irreducible component of this expansion. As noted above, the components are all compactified torus bundles over locally closed strata of $X$. By the algebraic transversality hypothesis, the boundedness is insensitive to further logarithmic modifications or partial compactifications of $\mathcal Y$. In particular, after replacing $\mathcal Y$ with a further logarithmic modification, any component $V$ of $\mathcal Y$ can be presented as
\[
V\to W\to B
\]
with $V\to W$ logarithmically \'etale and with $B$ a closed stratum in $X$ and $W$ a $\mathbb P^r$-bundle over $B$, with $r = \mathsf{codim}_X V$. The logarithmically \'etale morphism can be taken to be a composition of an open immersion and sequence of blowups. 

Let $Z\subset V$ be a subscheme with the given combinatorial data. We claim that there are finitely many choices for the homology class of $Z$ in $V$. We argue as follows. Since the Hilbert polynomial in $X$ is fixed, there are finitely many choices for the homology class of the pushforward of $[Z]$ to $B$. On the other hand, since we have fixed the contact orders of $Z$, there are finitely many choices for the fiber degree of the class $[Z]$ after pushforward to the toric bundle $W$. Moreover, the cohomology of $V$ can be obtained as a module over that of $W$ by the blowup formula, and is generated by classes of exceptional divisors and exceptional curves in the blowup. The homology class $[Z]$ is determined by its intersection multiplicities with these exceptionals, which are in turn determined by the contact order. 

Choose a very ample line bundle on $V$ and observe that the degree of the subscheme $Z$ with respect to this very ample line bundle is now bounded. Since the holomorphic Euler characteristic of $Z$ is also fixed, there are finitely many possible values of the Hilbert polynomial for the subscheme $Z$. The possible choices for the ideal sheaf are parameterized by an open subset in a finite disjoint union of connected components of the Hilbert scheme of $V$. The result follows.
\end{proof}

\subsection{Finiteness of automorphisms} We continue with the notation from the previous sections. Let $\mathsf{DT}_{\beta,n}(X|D)$ be the moduli space of ideal sheaves on $X$ relative to $D$, with a fixed choice of target expansion moduli. There is a structure map
\[
\mathsf{DT}_{\beta,n}(X|D)\to \mathsf{Exp}(X|D). 
\]
We use this to deduce finiteness of automorphisms the space of relative ideal sheaves.

\begin{theorem}
The isotropy groups of the moduli space $\mathsf{DT}_{\beta,n}(X|D)$ are finite. 
\end{theorem}

\begin{proof}
We show that the objects of $\mathsf{DT}_{\beta,n}(X|D)$ have finite automorphism group. Recall that a point of this stack is given by a morphism
\[
\spec\CC\to \mathsf{Exp}(X|D)
\]
together with an ideal sheaf $\mathcal I_{\mathcal Z}$ on the expansion $\mathcal Y$ of $(X|D)$; here ${\mathcal Z}$ denotes the associated subscheme. Moreover, we have imposed the condition that the ideal sheaf gives rise to a subscheme that is a tube precisely along the tube components. In order to analyze the stabilizer group of the ideal sheaf $\mathcal I_{\mathcal Z}$, we first note that the expansion $\mathcal Y$ maps to a substack $B\mathbb G_m^r$ of $\mathsf{Exp}(X|D)$. Specifically, the map above determines a locally closed stratum of $\mathsf{Exp}(X|D)$ corresponding to a cone $\sigma$ of the associated cone space. The isotropy group at this point is naturally identified with the torus whose lattice of one-parameter subgroups is $\sigma^{\mathsf{gp}}(\NN)$. We claim that $\mathcal I_{\mathcal Z}$ is not stabilized by any positive dimensional subgroup of this isotropy group. 

To see this, we unwind how the isotropy group acts on $\mathcal Y$. The components of the expansion are either modifications of $X$, or an equivariantly compactified $\mathbb G^k_m$ bundle over $S^\circ$ where $S^\circ$ is a locally closed stratum in $X$ of codimension $k$. Each of these bundles therefore possesses a translation action by the torus. Upon restriction to each component, the isotropy group acts as by scaling the torus directions described above (however, not all such translation actions come from the action of isotropy).

We wish to show that, in the setup above, no positive dimensional subtorus in $\mathbb G_m^r$ stabilizes the ideal sheaf $\mathcal I_{\mathcal Z}$. If there was such a subgroup, then in particular, $\mathcal I_{\mathcal Z}$ would be preserved by any one-parameter subgroup of that subgroup. It therefore suffices to show that there is no no one-parameter subgroup of $\mathbb G_m^r$ that stabilizes $\mathcal I_{\mathcal Z}$.

Assume to the contrary, that there were a one-parameter subgroup of $\mathbb G_m^r$ that stabilizes $\mathcal I_{\mathcal Z}$. We may pullback the expansion $\mathcal Y$ along the inclusion $B\mathbb G_m\to B\mathbb G_m^r$. This choice of $B\mathbb G_m$ point corresponds to a ray in a cone inside $T$. On the interior of this ray, there may be tube vertices on the associated $1$-complex. Erasing these vertices by passing to the associated $1$-complex in $|T|$ gives rise to a contraction morphism
\[
\begin{tikzcd}
\mathcal Y\arrow{dr} \arrow{rr} & & \overline{\mathcal Y}\arrow{dl}\\
&B\mathbb G_m&
\end{tikzcd}
\]
which collapses the $\mathbb P^1$-bundles corresponding to these vertices to their base. Furthermore, the subscheme ${\mathcal Z}$ is the inverse image, along this contraction, of a subscheme $\overline {\mathcal Z}\subset\overline{\mathcal Y}$. By the definition of algebraic transversality, every component of $\mathcal Y$ has nonempty intersection with the subscheme. The only possibility for a $\mathbb G_m$ to stabilize a subscheme is by scaling the fibers of a $\mathbb P^1$-bundle component in $\overline{\mathcal Y}$. But in order for $\mathbb G_m$ to stabilize such a subscheme, that subscheme, restricted to this component, has to be pulled back along the projection, and therefore it has to be a tube subscheme. But we have blown down all the tube components of $\mathcal Y$, and thus by DT stability, the remaining $\mathbb P^1$-bundle components cannot contain tube subschemes. It follows that the stabilizer must be finite, as required. 
\end{proof}

\subsection{Properness}  
We now verify properness of the moduli problem. We only prove the properness here under the simplifying assumption that the generic fiber of a family of expansions is unexpanded. This conveys the essential geometric arguments in the proof. The removal of the assumption adds some technical subtleties, and we delay this to Section~\ref{sec:expanded}.

\begin{theorem}\label{thm: properness}
The moduli space $\mathsf{DT}_{\beta,n}(X|D)$ is proper. 
\end{theorem}

\begin{proof}
The stack $\mathsf{Exp}(X|D)$ is locally of finite presentation, see~\cite[Corollary~6.7]{CCUW}. The moduli space $\mathsf{DT}_{\beta,n}(X|D)$ is a relative Hilbert scheme over the stack of expansions and is therefore also quasiseparated; the arguments in the previous section show that it is of finite type. Therefore we need only check the weak valuative criterion for properness. Let $K$ be a valued field and consider a morphism
\[
\spec K\to  \mathsf{DT}_{\beta,n}(X|D). 
\]
Since $\mathsf{DT}_{\beta,n}(X|D)$ is of finite type, we may assume that $K$ is the function field of a smooth curve and the valuation is given by order of vanishing at a closed point. After pulling back the universal expansion and universal subscheme, we must check that, after replacing $K$ with an extension, the resulting family extends to yield a map
\[
\spec R \to  \mathsf{DT}_{\beta,n}(X|D).
\]
We prove the result here in the case where the fiber over $K$ is unexpanded.  
In the unexpanded case, this follows directly from Proposition \ref{thm: uniqueness-strong-case}; indeed, our construction has been rigged to arrange this, and we spell out the details as follows.  The general case is handled by a nearly identical argument in Section~\ref{sec:expanded}, but a more complicated formalism is required to carry it out. 

Suppose we have a cone space $T$ with Artin fan $\mathsf A_T$.  The data of a morphism $\spec R \rightarrow \mathsf A_T$ which sends the generic point to the open stratum is equivalent to giving an integral point in $T$.  Indeed, such a statement holds for an Artin cone by a direct analysis, and the general statement follows.  If we pass to a base change of $R$, the corresponding integral point of the composed morphism is obtained by scaling the original point.

To prove the existence of a stable limit, apply Proposition \ref{thm: uniqueness-strong-case}; possibly after a base change, there exists a $1$-complex $\mathsf G \subset \Sigma_X$ such that the corresponding expansion $\mathcal Y_G \rightarrow \spec R$ has an algebraically transverse limit.  Furthermore, by Remark \ref{rem:notubes}, there are no tube subschemes contained in the components associated to bivalent vertices of $\mathsf G$ .

The $1$-complex $\mathsf G$ defines a rational point of $T$.  After a possible scaling to ensure integrality, along with the corresponding base change of $R$, we have a point $P \in T$ and an associated morphism $\spec R \rightarrow \mathsf{Exp}(X|D)$.  By construction, the corresponding expansion $\mathcal Y_{R} \rightarrow \spec R$
factors through $\mathcal Y_G$, and the contracted components are precisely the tube vertices of $\Upsilon_P$.  As a result the flat limit of ${\mathcal Z}_K$ in $\mathcal Y_{R}$ will be algebraically transverse and stable.

For uniqueness of the limit, given an expansion $\mathcal Y \rightarrow  \spec R$ associated to an integral point $P'$ of $T$, the universal property of $\mathsf G$ implies that there is a subdivision $\Upsilon_{P'} \rightarrow \mathsf G$ and corresponding modification $\mathcal Y \rightarrow \mathcal Y_G$.  There are no tube subschemes for $\mathcal Y_G$, so the new vertices of the subdivision are precisely the tube vertices of this family, so $P'$ and $P$ coincide and the stable limit is unique.
\end{proof}

\begin{remark}[Stable pair moduli]\label{rem: stable-pairs}
In this remark, we sketch how to modify the definitions and constructions to handle stable pairs instead of ideal sheaves.
A \textit{stable pair} consists of a pair $(F,s)$ consisting of a sheaf $F$ of fixed Hilbert polynomial with purely $1$-dimensional support and a section $\mathcal O_X\xrightarrow{s} F$ whose cokernel is $0$-dimensional~\cite{PT09}. A stable pair $(\mathcal F,s)$ on an expansion $\mathcal Y$ of $X$ is algebraically transverse if the scheme theoretic support of $\mathcal F$ is algebraically transverse and the support of the cokernel is disjoint from the divisorial strata of of $\mathcal Y$. 

The existence and uniqueness of algebraically transverse limits is proven analogously to the case of subschemes.  Specifically, given a family of algebraically transverse stable pairs over a valuation ring, we find a dimensionally transverse limit for the schematic support using the tropicalization.  We can then apply the stable pairs arguments used by Li and Wu~\cite{LiWu15} to produce an algebraically transverse limit.  As before with the Gr\"{o}bner limit, one can tropically characterize the algebraically transverse limit.  
A \textit{tube stable pair} is a pair that is pulled back from a $0$-dimensional subscheme on a surface along a $\PP^1$-bundle, with the condition that the cokernel of the section vanishes.  The bivalent vertices added to the dimensionally transverse limit correspond to those components where the limit stable pair is not a tube stable pair (and  one can use this to give an alternate proof of existence as well).  

The moduli space $\mathsf{PT}(X|D)$ can be constructed as a relative stable pair space on the universal expansion over $\mathsf{Exp}(X|D)$, with the condition that only tube components host tube stable pairs.  Properness follows from the tropical limit algorithms, and the virtual class follows from~\cite[Proposition~10]{MPT10} as in the ideal sheaf case. 
\end{remark}

\section{Logarithmic Donaldson--Thomas theory}\label{sec: log-DT-theory}

We construct logarithmic Donaldson--Thomas invariants for a simple normal crossings threefold pair $(X|D)$, by showing that the moduli spaces constructed in the previous section come equipped with virtual classes and evaluation morphisms to the relative divisors. We also show that the virtual class, and hence the numerical theory, does not depend on the choices that have been in constructing the moduli stack of target expansions.  We define DT generating functions and state the basic rationality conjecture, analogous to the non-relative setting.

\subsection{Virtual classes} Each moduli space of ideal sheaves comes equipped with a structure map
\[
\mathfrak t:\mathsf{DT}_{\beta,n}(X|D)\to \mathsf{Exp}(X|D).
\]
Let $\mathcal I$ be an ideal sheaf. Recall that given a coherent sheaf $\mathcal E$ on $X$, there is a trace map:
\[
\mathrm{Ext}^i(\mathcal E,\mathcal E)\to H^i(X,\mathcal O_X),
\]
see~\cite[Chapter~10]{SheafBook}. The kernel $\mathrm{Ext}_0(\mathcal E,\mathcal E)$ of this map is called the \textit{traceless} Ext group. 

The groups
\[
\mathrm{Ext}^1(\mathcal I,\mathcal I)_0 \ \ \ \mathrm{Ext}^2(\mathcal I,\mathcal I)_0
\]
govern the deformation and obstruction spaces of the ideal sheaf $\mathcal I$ respectively. Additional details concerning the fact that the traceless Ext groups govern the obstructions to deformations of sheaves can be found in~\cite[Section~4]{HT10} and~\cite[Section~3]{RPWT}.

\begin{proposition}
The structure map
\[
\mathfrak t:\mathsf{DT}_{\beta,n}(X|D)\to \mathsf{Exp}(X|D).
\]
is equipped with a perfect obstruction theory, with deformation and obstruction spaces as above. If $\beta$ is the curve class, then the there is an associated virtual fundamental class $[\mathsf{DT}_{\beta,n}(X|D)]^{\mathsf{vir}}$ in Chow homology of degree
\[
\mathsf{vdim} \ \mathsf{DT}_{\beta,n}(X|D) = \int_\beta c_1(T_X). 
\]
The virtual class in Borel--Moore homology is defined by taking the image under the cycle class map.
\end{proposition}

\begin{proof}
The existence of the relative obstruction theory as described is established in the requisite generality in~\cite[Proposition~10]{MPT10}. Since the base $\mathsf{Exp}(X|D)$ of the obstruction theory is a connected Artin fan, it is irreducible and equidimensional. Virtual pullback of the fundamental class gives rise to the virtual fundamental class in the usual fashion~\cite{Mano12}. 
\end{proof}

\subsection{Virtual birational models} The spaces constructed in the paper to this point have all required a choice of tropical target moduli. We show that different choices lead to compatible theories. 

We supplement the notation in this section to keep track of the different theories arising from the tropical choices. Recall that the construction of $T$ in Section~\ref{sec: target-moduli} produces an infinite collection of spaces, depending on the choices made in the algorithmic steps of Section~\ref{sec: constructing-T}. Moreover, there are additional choices that are made in the construction of the universal family. 

We denote by $\Lambda$ the set of outputs of that construction. More precisely, each element in $\Lambda$ is a morphism of polyhedral complexes 
\[
\pi^\lambda: \Upsilon^\lambda\to T^\lambda
\]
where $T^\lambda$ is a moduli space of tropical expansions and $\pi^\lambda$ is the projection from the universal family. Let $\mathsf{Exp}(X|D)^\lambda$ be the associated moduli space of expansions and let 
\[
\mathcal Y^\lambda\to \mathsf{Exp}(X|D)^\lambda
\]
be the projection from the universal family. We fix the discrete data and establish the following birational invariance statement, parallel to~\cite[Section~3.6]{R19}, which is itself modeled on~\cite[Section~6]{AW}. 

\begin{theorem}[Independence of choices]\label{thm: independence-of-choices}
For any $\lambda,\mu\in \Lambda$ there exists $\varphi\in\Lambda$ and logarithmic modifications
\[
\begin{tikzcd}
 & \mathsf{DT}^\varphi_{\beta,n}(X|D)\arrow[swap]{dl}{\pi^\lambda}\arrow{dr}{\pi^\mu}& \\
 \mathsf{DT}^\lambda_{\beta,n}(X|D) & &\mathsf{DT}^\mu_{\beta,n}(X|D)
\end{tikzcd}
\]
such that 
\[
\pi^\lambda_\star[\mathsf{DT}^\varphi_{\beta,n}(X|D)]^{\mathsf{vir}} =  [\mathsf{DT}^\lambda_{\beta,n}(X|D)]^{\mathsf{vir}} \ \ \mathrm{and} \ \ \pi^\mu_\star [\mathsf{DT}^\varphi_{\beta,n}(X|D)]^{\mathsf{vir}} = [\mathsf{DT}^\mu_{\beta,n}(X|D)]^{\mathsf{vir}}. \]
\end{theorem}

\begin{proof}
We drop the discrete data to avoid overcrowding the notation. By applying Proposition~\ref{prop: common-refinements} in conjunction with the categorical equivalence between Artin fans and cone stacks, we can find a common refinement 
\[
\begin{tikzcd}
&\mathsf{Exp}^\varphi(X|D)\arrow{dl}\arrow{dr}&\\
\mathsf{Exp}^\lambda(X|D) & & \mathsf{Exp}^\mu(X|D)
\end{tikzcd}
\]
where both vertical arrows are birational. The definitions furnish a forgetful morphism
\[
\mathsf{DT}^\varphi(X|D)\to \mathsf{DT}^\lambda(X|D).
\]
We argue that pushforward identifies virtual classes. We have a universal diagram
\[
\begin{tikzcd}
\mathcal Z^\varphi \arrow[rr,hook]\arrow{dr}& &\mathcal Y^\varphi\arrow{dl} \arrow{r} & X\\
&  \mathsf{DT}^\varphi(X|D).&
\end{tikzcd}
\]
The analogous diagram exists for $\lambda$. We claim that the induced morphism
\[
p: \mathcal Y^\varphi\to \mathcal Y^\lambda\times_{\mathsf{Exp}^\lambda(X|D)} \mathsf{Exp}^\varphi(X|D) 
\]
is obtained by adding tube components to the target geometry. Indeed, this is a tautology -- by definition, tropicalizations of the expansions of this family of targets induce the \text{same} moduli map to $|T|$. Since the DT stability condition guarantees the tube components only contain tube subschemes, the universal subscheme $\mathcal Z^\varphi$ is simply obtained by pulling back the family $\mathcal Z^\lambda$ along the modification. It follows that we have a commutative square
\[
\begin{tikzcd}
\mathsf{DT}^\lambda(X|D)\arrow{r}\arrow{d}& \mathsf{DT}^\mu(X|D)\arrow{d}\\
 \mathsf{Exp}^\lambda(X|D)\arrow{r}& \mathsf{Exp}^\mu(X|D)
\end{tikzcd}
\]
As we have argued above, each of the two vertical arrows have perfect obstruction theories. The introduction of tube components has no effect on the obstruction theory. Indeed, the universal ideal sheaf $\mathcal I^\varphi$ is pullback of the universal ideal sheaf $\mathcal I^\lambda$ along the induced map of targets, so the right vertical obstruction pulls back to the \textit{same} obstruction theory as the one on the left. Since the lower horizontal is a subdivision, it is proper and birational; equality of virtual classes follows from the Costello--Herr--Wise theorem~\cite{Cos06,HW23}
\end{proof}

\begin{remark}
There is a further birational invariance statement that one might expect for the \textit{target}. Specifically, if $X'\to X$ is the blowup of $X$ along a stratum of $D$ with total transform $D'$, one may hope that there is an associated map $\mathsf{DT}(X'|D')\to \mathsf{DT}(X|D)$ that identifies virtual classes, in direct analogy with~\cite{AW}. In fact, when phrased without the DT stability condition, such a result is obvious since every expansion of $X'$ is also an expansion of $X$. However, a non-tube component in a subscheme of an expansion of $X'$ can become a tube component with respect to $X$. An appropriately corrected analysis does yield a birational invariance statement. We delay this to future work. 
\end{remark}

{
\subsection{Evaluations at the relative divisors} Let $(X|D)$ be as above. Let $E_1,\ldots E_k$ be the irreducible components of $D$. These components are smooth by hypothesis. Each intersection of distinct divisors $E_i\cap E_j$ is either empty or a connected divisor, again by hypothesis. Thus, each variety $E_i$ is equipped with a simple normal crossings divisor of its own. We abuse notation slightly and denote these by $(E_i|D)$. 

By applying the constructions of the preceding sections to $E_i$, we obtain a stack of expansions $\mathsf{Exp}(E_i|D)$. For our purposes, it will suffice to examine with the substack parameterizing those expansions whose associated $1$-complex is a finite set of points. We pass to the subscheme of the universal expansion where the irreducible components of the expansion are disjoint. 

For each positive integer $d_i$, we consider the Hilbert scheme of points in the universal expansion
\[
\mathcal E\to \mathsf{Exp}^\lambda(E_i|D)
\]
of zero dimensional subschemes of length $d_i$. As before, the symbol $\lambda$ indicates that we have chosen some stack of target expansions. By studying the locus comprising subschemes of expansions such that no component of the expansion is nonempty, we obtain a proper and finite type stack $\mathsf{Hilb}_\lambda^{d_i}(E_i|D)$ as a special case of the construction above. 

As a special case of the virtual birational invariance of the Donaldson--Thomas moduli spaces, we obtain the parallel statement for the Hilbert schemes of points on a curve or surface. However, since the Hilbert schemes of points on smooth curves and surfaces are unobstructed moduli schemes, the same argument produces an honest birationality statement -- different choices of $\lambda$ give rise to birational models of the logarithmic Hilbert scheme of points on logarithmically smooth (curves and) surfaces.

Given a curve class $\beta$, define $d_i$ to be the intersection number $D_i\cdot \beta$. We therefore have an evaluation space for relative insertions
\[
\mathsf{Ev}_\beta(X|D) = \mathsf{Hilb}^{d_1}(E_1|D)\times\cdots\times \mathsf{Hilb}^{d_k}(E_k|D). 
\]

\begin{lemma}
There exists compatible stacks $\mathsf{Exp}(X|D)$ and $\mathsf{Exp}(E_i|D)$ and an evaluation morphism 
\[
\mathsf{ev}: \mathsf{DT}_{\beta,n}(X|D)\to \mathsf{Ev}_\beta(X|D)
\]
of the associated moduli spaces of ideal sheaves and points. 
\end{lemma}

\begin{proof}
We first explain how the evaluation looks at the tropical level. The reader may first wish to reconsult with Remark~\ref{rem: asymptotics}.  Fix a space of expansions for $\mathsf{Exp}(X|D)$ with associated tropical moduli space $\mathsf{T}(X|D)$. The latter is a particular conical structure on the space of embedded $1$-complexes. Examine such a $1$-complex $G\hookrightarrow \Sigma$. Each divisor component $E_i$ determines a ray $\rho_i$. The transversality hypothesis implies that outside a bounded set in $\Sigma$, the $1$-complex is made up of a finite set of unbounded rays, each parallel to one of the rays $\rho_i$. The star of $\rho_i$ in $\Sigma$ is a cone complex of dimension one lower than that of $\Sigma$. Following the toric dictionary, rays parallel to $\rho_i$ determine points in the star $\Sigma(\rho_i)$. This star is canonically embedded in the tropicalization of $(E_i|D)$. For each cone $\sigma$ in the cone space of $\mathsf{T}(X|D)$, we therefore obtain a family over $\sigma$ of points in $\Sigma(\rho_i)$.

We now apply the construction in Section~\ref{sec: constructing-T} to the cone complex $\Sigma(\rho_i)$ to produce a tropical moduli space $\mathsf{T}(E_i|D)$. By the description above, given a cone of $\mathsf{T}(X|D)$, we produce a map to $\mathsf{T}(E_i|D)$. The image of $\sigma$ need not be contained in a single cone of $\mathsf{T}(E_i|D)$, but after replacing both $\mathsf{T}(E_i|D)$ and $\sigma$ with subdivisions $\mathsf{T}'(E_i|D)$ and $\sigma'$, we obtain a morphism
\[
\sigma'\to \mathsf{T}'(E_i|D)
\]
of cone spaces. For any choice of conical model, the morphism from $\mathsf{DT}_{\beta,n}(X|D)\to \mathsf{Exp}(X|D)$ factors through a finite type substack of the image, given by a finite collection of cones in $\mathsf{T}(X|D)$. It follows that after a finite collection of subdivision operations, we obtain a morphism
\[
\mathsf{DT}_{\beta,n}(X|D)\to \mathsf{Exp}'(E_i|D)
\]
again for an appropriate system of choices. In order to lift this to an evaluation morphism, we pullback the universal expansion $\mathcal E$ over $\mathsf{Exp}'(E_i|D)$ to $\mathsf{DT}_{\beta,n}(X|D)$. There is a inclusion of expansions
\[
\mathcal E\times_{\mathsf{Exp}'(E_i|D)} \mathsf{Exp}'(X|D)\hookrightarrow \mathcal Y.
\]
Intersecting the universal subscheme $\mathcal Z$ with the expansion, the algebraic transversality condition implies that we obtain a subscheme of length $d_i$, and thereby a point of $\mathsf{Hilb}^{d_i}(E_i|D)$. Performing the construction for each divisor, we have constructed the claimed morphism.
\end{proof}

\begin{remark}
When combined with the projection formula, the birationality statements for the moduli spaces of sheaves and of points arising from Theorem~\ref{thm: independence-of-choices} guarantee that numerical invariants are independent of the choices made in the construction of the spaces of expansions.
\end{remark}

\begin{remark}
This evaluation morphism imposes schematic tangency conditions via the Hilbert scheme of points on the boundary surfaces. A more refined evaluation is to the relative Hilbert scheme of points on the universal surfaces over $\mathsf{Exp}(X|D)$. This refined object is more natural from the point of view of the degeneration formula, and for the GW/DT correspondence. In particular, this is the way in which the gluing formula is proved on the Gromov--Witten side~\cite{R19}. We will examine this refined evaluation in future work, in its appropriate context.
\end{remark}

\subsection{Logarithmic DT invariants}\label{sec: invariants-conjectures}

We define descendent insertions for logarithmic DT invariants as follows.
Let 
\[
(\pi_{\mathsf{DT}}, \pi_X): \mathcal Y \rightarrow \mathsf{DT}_{\beta,n}(X|D) \times X
\]
denote the universal expansion over the DT moduli space.  While the map $\pi_{\mathsf DT}$ is not proper, it is proper when restricted to the universal subscheme $\mathcal Z \hookrightarrow \mathcal Y$.  We claim the universal ideal sheaf $\mathcal{I}_{\mathcal{Z}}$ is a perfect complex.  Indeed, it suffices to show it is Tor-finite, and for this we can take the derived tensor product with the skyscraper sheaf $k_p$ at any closed point $p$ of $\mathcal Y$. Since $\mathcal{I}_{\mathcal{Z}}$ is flat over $ \mathsf{DT}_{\beta,n}(X|D)$,
this is equivalent to showing Tor-finiteness of $\mathcal{I}|_Z$ for the point $$[Z] = \pi_{\mathsf{DT}}(p) \in \mathsf{DT}_{\beta,n}(X|D).$$  However, this follows from the algebraic transversality of $Z$, by Proposition $3.5$ of \cite{LiWu15}.

Since $\mathcal{I}_{\mathcal{Z}}$ is perfect, it admits Chern classes; we denote by $\mathrm{ch}_k(\mathcal{I}) \in H^\star(\mathcal{Y})$ the cohomological degree $k$ part of the Chern character. These classes are supported on $\mathcal{Z}$ for $k \geq 1$.

Given a cohomology class $\gamma \in H^\star(X)$, we define descendent operators for $k\geq 0$
\[
\tau_k(\gamma): H_\star(\mathsf{DT}_{\beta,n}(X|D), \mathbb{Q}) \rightarrow H_\star(\mathsf{DT}_{\beta,n}(X|D), \mathbb{Q})
\]
via the formula
\[
\tau_k(\gamma) := (-1)^{k+1}\pi_{\mathsf{DT},\star}\left(\mathrm{ch}_{k+2}(\mathcal{I})\cdot \pi_X^\star(\gamma)\right)\cap \pi_{\mathsf{DT}}^{\star}.
\]
Here, the pullback on homology is well-defined since $\pi_{\mathsf{DT}}$ is flat by \cite{Verdier}, and the pushforward is well-defined since $\pi_{\mathsf{DT}}$ is proper on $\mathcal{Z}$.

Given a cohomology class $\mu \in H^\star(\mathsf{Ev}_\beta(X|D))$, and cohomology classes
$\gamma_1, \dots, \gamma_r \in H^\star(X)$ and indices $k_1, \dots, k_r$, we can define
descendent invariants 
\[
\langle \tau_{k_{1}}(\gamma_1)\dots \tau_{k_{r}}(\gamma_r)| \mu \rangle_{\beta,n}
= \mathsf{deg}\left( \prod_{i=1}^{r}\tau_{k_{i}}(\gamma_i)\mathsf{ev}^\star\mu\cap[\mathsf{DT}]^{\mathsf{vir}}  \right)
\]
where $\mathsf{deg}$ denotes the degree of the pushforward to a point.  
If we sum over $n$ we have the partition function
\[
\mathsf{Z}_{\mathsf{DT}}\left(X,D;q | \prod_{i=1}^{r} \tau_{k_{i}}(\gamma_i)|\mu\right)_{\beta}
= \sum_{\chi} \langle \prod_{i=1}^{r} \tau_{k_{i}}(\gamma_i)| \mu \rangle_{\beta,n} q^{\chi}
\]
which is a Laurent series since for each fixed curve class $\beta$, the logarithmic DT spaces are empty for $\chi\ll0$.

In analogy with absolute and relative DT invariants, we have the following basic conjecture regarding this series.
\begin{conjecture}\label{conj: basic-conjecture}
\begin{enumerate}
\item[(i)]
The DT series for zero-dimensional subschemes is given by
\[
\mathsf{Z}_{\mathsf{DT}}\left(X,D;q)\right)_{\beta=0} = M(-q)^{\int_{X} c_{3}(T^\mathsf{log}_X\otimes K^{\mathsf{log}}_{X})},
\]
where
\[
M(q) = \prod_{n\geq1} \frac{1}{(1-q^n)^n}
\]
is the McMahon function.
\item[(ii)]
For any curve class $\beta$, and insertions $\gamma_1, \dots, \gamma_r$ of degree $\geq 2$ and relative insertion $\mu$, the normalized DT series 
\[
\mathsf{Z}'_{\mathsf{DT}}\left(X,D;q | \prod_{i=1}^{r} \tau_{k_{i}}(\gamma_i)|\mu\right)_{\beta}
:= 
\frac{\mathsf{Z}_{\mathsf{DT}}\left(X,D;q | \prod_{i=1}^{r} \tau_{k_{i}}(\gamma_i)|\mu\right)_{\beta}}{\mathsf{Z}_{\mathsf{DT}}\left(X,D;q)\right)_{\beta=0}}
\]
is the Laurent expansion of a rational function in $q$.
\end{enumerate}
\end{conjecture}
We expect the first part of this conjecture can be proven once the degeneration formalism is in place. 

\begin{remark}
The conjecture above is natural in the context of the Gromov--Witten/Donaldson--Thomas conjectures, which were studied for Calabi--Yau threefolds in~\cite{MNOP06a} and generalized to more general targets, including smooth pair relative geometries, in the sequel~\cite{MNOP06b}. The rationality of the reduced Donaldson--Thomas partition function is required in order for a precise formulation of the GW/DT equivalence, as the latter involves an exponential change of variables. The DT series itself is \textit{not} a rational function, as evidenced by the first part of the conjecture concerning the degree $0$ part, which must first be removed. The degree $0$ evaluation is known for empty and smooth boundary divisors~\cite{BF08,LP09,Li06}. Rationality of the DT generating series is known for all toric targets~\cite{MOOP}. 
\end{remark}

The normalized DT series that removed the degree $0$ contribution of the MacMahon function has a geometric interpretation in terms of the stable pairs theory of Pandharipande and Thomas~\cite{PT09}. By replacing the DT invariants in the definition of $\mathsf Z_{\mathsf{DT}}$ with the corresponding stable pairs invariants, we obtain a PT generating series denoted $\mathsf{Z}_{\mathsf{PT}}\left(X,D;q | \prod_{i=1}^{r} \tau_{k_{i}}(\gamma_i)|\mu\right)_{\beta}$. We expect rationality and wall-crossing statements for this generating series.

\begin{conjecture}\label{conj: PT-conjecture}

For any curve class $\beta$, relative insertion $\mu$, and insertions $\gamma_1, \dots, \gamma_r$, the PT series
\[
\mathsf{Z}_{\mathsf{PT}}\left(X,D;q | \prod_{i=1}^{r} \tau_{k_{i}}(\gamma_i)|\mu\right)_{\beta}
\]
is a rational function in $q$.
If the degree of each $\gamma_i$ is at least $2$, this series equals the normalized DT series
$\mathsf{Z}'_{\mathsf{DT}}$ from the previous conjecture.
\end{conjecture}

The GW/DT/PT correspondences, including the degree $0$ evaluations and their compatibility with logarithmic degeneration formulas, will be addressed more systematically in forthcoming work~\cite{MR22}.

%
%
%
%
%

\section{First examples}\label{sec: first-examples}

We sketch a handful of basic examples, focusing on the space of targets, where the main new complexity lies. The goal of this section is to show that in a reasonable range of test cases, the logarithmic Donaldson--Thomas moduli spaces can be worked with fairly explicitly. The final part of this section explains how to think about the tube vertices i the traditional Li--Wu context.

\subsection{Target yoga} There is a well-known procedure by which the strata of the expanded target geometry may be recognized from the tropical pictures, as a consequence of the toric dictionary concerning closed torus orbits and stars around cones~\cite[Section~3.1]{Ful93}. Let $G\hookrightarrow \Sigma$ be an embedded $1$-complex. Let $\mathcal Y_G$ be the associated expansion of $X$. Let $v$ be a vertex of $G$ that maps to a cone $\sigma$ of $\Sigma$ of dimension $k$. There is an associated irreducible component
\[
Y_v\hookrightarrow \mathcal Y_G\to X.
\]
This map factors through the inclusion of the closed stratum $X_\sigma\hookrightarrow X$. The induced morphism
\[
Y_v\to X_\sigma
\]
is a partially compactified torus bundle of rank $k$. For instance, if $X$ is a threefold and $\sigma$ has dimension $3$, then $Y_v$ is simply a toric threefold. The torus bundles are obtained from the line bundles associated to the divisorial logarithmic structure on $X$ and its blowups. The partial compactification of $Y_v$ is obtained (fiberwise) as the toric variety associated to the star of $v$ in $G$, which gives rise to a fan embedded in the star of $\sigma$. 

\subsection{Dual plane} Let $(X|D)$ be the pair consisting of $\PP^2$ and its toric boundary divisor. The most basic example is the moduli space of lines in $\PP^2$, with transverse contact orders. There is a canonical space of ideal sheaves $\mathsf{DT}(\PP^2|D)$, and it is identified with the blowup of the dual $\PP^2$ at its three torus fixed points. This example is worked out carefully in~\cite[Section~4.1]{R15b} for stable maps, and the situation is identical to the sheaf theory setup. The figures in op. cit. may be of particular use. 

The fibers of the expanded target family $\mathcal Y\to \mathsf{DT}(\PP^2|D)$ can be described as follows. The target is either isomorphic to $\PP^2\setminus \{p_1,p_2,p_3\}$ or the complement of the codimension $2$ strata in the deformation to the normal cone of a line in $\PP^2$. The latter arises as a transverse replacement in two ways, either when the subscheme limits to a coordinate line, or when the intersection point of the subscheme with the coordinate line limits to a coordinate point. 

\subsection{Plane curves} The case of degree $d$ curves in $\PP^2$ makes contact with the study of discriminants, and in particular, the secondary fan of a toric fan, constructed by Gelfand--Kapranov--Zelevinsky~\cite{GKZtriangulations}. A detailed construction of the stable pair moduli spaces in these cases has been undertaken Kennedy-Hunt~\cite{KH21}. The main result of this paper shows that the PT moduli spaces can be expressed explicitly as the vanishing locus of a section of a tautological vector bundle on smooth space that lies over a toric variety. The discussion that follows collapses to the dual plane for $d = 1$.

Let $\mathcal P_d$ be the convex hull of the lattice points $(0,0)$, $(d,0)$ and $(0,d)$. Let $\PP_d$ be the linear system of degree $d$ curves in $\mathbb P^2$, which coincides with the Hilbert scheme of degree $d$ curves. An open subset of this projective space parameterizes curves that are algebraically transverse to the coordinate lines in $\PP^2$. This subset is stable under the $\mathbb G_m^2$ action on $\mathbb P^2$. Limits under the action of this torus capture the relevant degenerations to construct the Donaldson--Thomas moduli space, as we now sketch. 

Any expansion $\mathbb P^2$ along its toric strata arising from our construction is necessarily toric. Let $G\subset \Sigma_{\PP^2}$ be a $1$-complex. Given an algebraically transverse subscheme
\[
 Z\subset \mathcal Y_G,
\]
of the associated expansion, we may proceed as in Section~\ref{sec: boundedness} and label each edge $E$ of $G$ with the length of the intersection of $Z$ with the divisor attached to $E$. Since the target is toric, the graph $G$ together with this weighting furnishes a balanced tropical curve. We abuse notation slightly and continue to denote this decorated object by the symbol $G$.

The secondary fan $T^\sim_d$ of the polytope $\mathcal P_d$ enters as a moduli space of tropical plane curves up to translation by the natural additive action of $\mathbb R^2$.  Recall that the secondary fan of $\mathcal P_d$ parameterizes subdivisions of $\mathcal P_d$ that are \textit{regular} or \textit{coherent} i.e. they are induced by a strictly convex piecewise linear function, see the chapter in the text~\cite{GKZtriangulations}. The dual of such a subdivision gives rise to a polyhedral subdivision of $\mathbb R^2$. The edges and rays of this subdivision carry a natural weight -- they correspond to edges in the coherent subdivision of $\mathcal P_d$, and the weight is defined to be the lattice length of the edge. This weighted $1$-dimensional polyhedral complex is a balanced tropical curve in $\RR^2$~\cite[Remark~2.3.10]{MS14}. The top dimensional cones correspond to maximal triangulations of $\mathcal P_d$.

Geometrically, toric variety of the secondary fan is isomorphic to the Chow quotient $\mathbb P_d/\!/\mathbb G_m^2$, of the Hilbert scheme by the dilating action of the torus in $\PP^2$. The Chow quotient carries a universal flat family of broken toric surfaces, arising as limits of $\mathbb G_m^2$-orbits in $\PP_d$, see~\cite[Section~3]{AM14}, and the total space of the universal family is also a toric variety. The total space of this family is birational to $\mathbb P_d$.\footnote{Notice that the Chow quotient itself has dimension $2$ lower than $\mathbb P_d$. It can be interpreted as a ``rubber variant'' of the logarithmic DT space. } We explain how some combinatorial adjustments give rise to an instance of the spaces constructed here.  

Let $T_d$ be the fan of the universal family of the Chow quotient. At the combinatorial level, these data are captured by the maps of polyhedral complexes
\[
\begin{tikzcd}
\Sigma_d \arrow{r}\arrow{d} & \RR^2\\
T_d.
\end{tikzcd}
\]
The toric variety associated to $\Sigma_d$ is the universal family \textit{of the universal family} of the Chow quotient, interpreted a moduli space of pairs of a tropical curve together with a point in $\mathbb R^2$. See~\cite{KH21} for further details. 
In order to make a connection to the Donaldson--Thomas moduli space, a further modification is required. The expansions we propose are always equipped with a map to $\PP^2$, while the degenerations arising from regular triangulations need not be modifications of the constant family. The adaptation is as follows. The secondary fan comes equipped with a universal family of polyhedral subdivisions of $\mathbb R^2$:
\[
\begin{tikzcd}
\Sigma_d \arrow{r}\arrow{d} & \RR^2\\
T_d.
\end{tikzcd}
\]
Let $\widetilde \Sigma_d$ be the fiber product $\Sigma_d\times_{\RR^2} \Sigma_{\PP^2}$, which is the common refinement of $\Sigma_d$ with the constant family $T_d\times\Sigma_{\PP^2}$. The toric map induced by $\widetilde\Sigma_d\to T_d$ need not be flat with reduced fibers, but by universal weak semistable reduction, there is a canonical such family over a new fan $T'_d$, see~\cite{KH21,Mol16}. 

Finally, let $T''_d$ denote the quotient of $T'_d$ by the equivalence relation that identifies two embedded tropical curves if they have the same underlying $1$-complex (i.e. we forget the weights along the edges described above).  It is straightforward to check that this is a cone space, and is a union of cones in a moduli space of tropical expansions $T(\Sigma)$ constructed in Section~\ref{sec: constructing-T}.  We refer the reader to~\cite{KH21} for further details. 

Let $\mathsf{Exp}(\PP^2|D)$ and $\mathsf A_d$ be the Artin fans associated to $T''_d$ and $T'_d$ respectively. The constructions in the previous section give rise to a Donaldson--Thomas moduli space; given a subscheme in an expansion, we can recover the weights on the edges using the schematic intersection numbers with the corresponding divisors.  As a result, the structure morphism to $\mathsf{Exp}(\PP^2|D)$ factors through
\[
\mathsf{DT}(\PP^2|D)\to\mathsf{A}_d.
\]
The subdivision $\widetilde \Sigma$ gives rise to a (non-representable) birational modification $\mathbb P'_d$ of the universal orbit of the Chow quotient $\PP_d/\!/ \mathbb G_m^2$.This toric stack is naturally identified with $\mathsf{DT}(\PP^2|D)$. A rigorous treatment of the identification is given in~\cite{KH21}. Analogous statements hold for any toric surface. 

We conclude with a series of remarks.

\begin{remark}
The stable pair moduli spaces in degree $d$ on $\PP^2$ but with arbitrary holomorphic Euler characteristic admit an explicit and elegant description. A complete picture, including the basic geometry and virtual class, has been given in~\cite{KH21}. The upshot of Kennedy-Hunt's results are that the logarithmic pair spaces are obtained as the relative Hilbert schemes of points on the universal subscheme over the toric variety constructed above, parallel to~\cite[Propostion~B.10]{PT10}. At the combinatorial level, this amounts to allowing a fixed additional number of vertices on the balanced $1$-complexes considered above. The additional vertices allow for expansions along the divisors to accommodate support for the cokernel of the section of the stable pair. 
\end{remark}

\begin{remark}
The subdivision step above can be avoided by working directly with the logarithmic multiplicative group $\mathbb G^2_{\mathsf{log}}$. The tropicalization is the $\RR^2$ appearing above, with ``fan'' structure given by the single non-strictly convex cone $\RR^2$. We avoid the detour through the details of the logarithmic multiplicative group, but a reader interested in the details of this perspective may extract them using the arguments in~\cite{RW19}.
\end{remark}

\begin{remark}
The moduli space above is naturally identified with the universal family over the Chow quotient. The Chow quotient itself can also be interpreted as a Donaldson--Thomas moduli space for the rubber moduli space $\PP^2$, see~\cite{MW17}. The rubber moduli will be discussed elsewhere. 
\end{remark}

\begin{remark}
The relationship between the secondary fan and enumerative geometry should be credited to Katz~\cite[Section~9]{Kat09}. Katz suggested that the universal family of the Hilbert quotient of a toric variety should function as a moduli space of polarized target expansions for Gromov--Witten theory relative to the toric boundary. In retrospect, it seems more natural to realize this connection within the context of ideal sheaves or stable pairs. A detailed study of the role of the secondary fan in logarithmic curve counting theories seems worthwhile. 
\end{remark}

\subsection{Subgroups in toric threefolds} Let $X$ be a smooth projective toric threefold with $D$ the toric boundary divisor. Let $\Sigma_X$ be the fan in the cocharacter space $N_\RR$. We make the further assumption that there exists a nonzero primitive vector $v$ such that the ray generated by $v$ and by $-v$ are both rays in $\Sigma_X$. This guarantees that the associated compactified one-parameter subgroup
\[
\varphi_v: \PP^1\to X
\]
is transverse to $D$. 

The Donaldson--Thomas moduli space in the class $[\mathsf{im}(\varphi)]$ is again related to the geometry of Chow quotients. Consider the line $\ell_v$ spanned by $v$. The relevant embedded $1$-complexes in this instance are precisely the parallel translates of $\ell$ in $N_\RR$. Given such a translate $\ell$ of $\ell_v$, we equip it with a canonical polyhedral structure by taking the common refinement with $\Sigma_X$ inside $N_\RR$. 

The moduli space of such embedded $1$-complexes is canonically identified with the fan of the Chow quotient $X/\!/\mathbb G_m$ by results of Chen--Satriano~\cite{CS12}. The identification is somewhat indirect, so we spell it out. In loc. cit. the authors identify the Chow quotient with the moduli space of logarithmic stable maps to $X$ in the class of $\varphi_v$. The tropical moduli space of maps is precisely the set of parallel translates $\ell_v$ with the polyhedral structure described above, for instance from the arguments in~\cite{R15b}. However, it is straightforward to see that all tropical maps in this case are forced to be embedded complexes, so we may identify the tropical moduli space with the set of embedded $1$-complexes. The Donaldson--Thomas moduli space for this data is then equal to the Chow quotient. The identification of coarse moduli spaces occurs can be upgraded using the Chow quotient stack. 

\begin{remark}
Stable map, ideal sheaf, and stable pair spaces coincide in this case. 
\end{remark}

\subsection{An embedded point and a subgroup} We maintain the notation from the previous example, and consider subschemes in the curve class $\varphi_v$, but we increase the Euler characteristic by $1$. Geometrically, we consider the class of a one-parameter subgroup together with an embedded point. As the additional point may wander the target, the space of relative ideal sheaves includes expansions caused by the presence of the additional point on the boundary $D$. The relevant $1$-complexes are therefore translates of the line $\ell$ together with an additional point, which can lie anywhere in $N_\RR$, and whose presence introduces a vertex. 

The space of expansions can be constructed as follows. Let $S_v$ be the toric fan associated to the Chow quotient $X/\!/\mathbb G_m$ as above and let $U$ be the universal orbit of the Chow quotient. The latter is a refinement of $\Sigma_X$. The product fan $\Sigma_U\times S_v$ can be viewed as a moduli space of pairs $(p,\ell)$ where, as above, $\ell$ is a parallel translate of $\ell_v$ and $p$ is a point in $\Sigma_U$. Consider the incidence locus
\[
\mathsf{IL}:=\left\{(p,[\ell]): p\in \ell\right\}\subset \Sigma_U\times S_v.
\]
This incidence locus has a fan structure by viewing it as the preimage of the diagonal in the projection 
$$\Sigma_U\times S_v\to S_v\times S_v$$ where the morphisms are the projection of the universal family and the identity map on the two factors. Note that since the vertical maps are both combinatorially flat with reduced fibers, the fiber product is canonically a cone complex. 

Consider any subdivision $T$ of $\Sigma_U\times S_v$ such that the locus $\mathsf{IL}$ is a union of faces. Over $T$, we obtain two families of subcomplexes -- the first is the family of parallel translates of $\ell_v$ while the second is simply the universal family $U$ over $T$. Each of these map to $T\times \Sigma_U$, and we choose a conical structure on the union of their images. The image gives the required conical structure. 

The Artin fan associated to $T$ gives rise to a moduli space $\mathsf{Exp}(X|D)$ of expansions of $X$ along $D$. Note that the fibers of the morphism
\[
\mathsf{DT}(X|D)\to \mathsf{Exp}(X|D)
\]
also have a straightforward description. Fixing the target expansion also fixes a $1$-complex in $\Sigma_X$. For each vertex of $\Sigma_X$ that lies on the line $\ell$, we choose a torsor for the one-parameter subgroup $v$, in the corresponding component. The closure provides a subscheme $Z$ with the required homology class. Therefore, additional moduli is provided by the choice of point in the expansion. In order to account for the moduli of tangent directions when the point lies on on $Z$, consider the moduli space of points in the expansion and blowup the incidence locus with $Z$. 


\subsection{Lines in $\PP^3$} Let $(X|D)$ be the toric pair $(\PP^3|H)$ where $H$ is the union of the coordinate planes. We take the fan of $\PP^3$ as formed by the standard basis vectors and the negative of their sum in $\RR^3$. The class of a generic line in $\PP^3$ can be described using similar consideration as the one above. First we consider the moduli spaces of lines in $\PP^3$. The generic line in $\PP^3$ that meets only the codimension $1$ strata of the toric boundary has tropicalization a single vertex with $4$ outgoing infinite rays, in each of the three coordinate directions and $(-1,-1,-1)$. The relevant embedded $1$-complexes required to construct the moduli space of expansions are balanced tropical curves whose asymptotic rays are the ones described above.

Once again the space of $1$-complexes can naturally be identified with the elements of the moduli space of tropical maps, because all tropical maps for this moduli problem are forced to be embeddings. The moduli space coincides with an explicit blowup of $\Mbar_{0,4}\times \PP^3$, following~\cite[Theorem~B]{R15b}. 

If we consider a line together with an embedded point, thereby raising the holomorphic Euler characteristic, the tropical space of expansions is obtained from the one above by adding an additional point, and subdividing the incidence locus where this point lies on the $1$-complex. As before, the Donaldson--Thomas moduli space has a map to the space of targets, whose fibers can be explicitly described as the blowup of an algebraic incidence locus, stratum-wise on the space of target expansions. We leave the details to the reader. 

\subsection{Revisiting the Li--Wu setup}\label{sec: revisiting-Li-Wu} The notions of tube vertices, tube components, and tube subschemes are not present in Li--Wu's construction of the DT moduli problem for smooth pairs $(X|D)$, and our formalism can be made to specialize to theirs by choosing the right cone structure on the cone space $T$ of $1$-complexes. However, tube components can be \textit{forced} to appear in the Li--Wu setup if one makes less efficient choices for $T$. We explain this now, since it may help readers who are familiar with relative DT and GW theory already.\footnote{The content of the present subsection came out of a email discussion between R.~Thomas and the second author. We take the opportunity to thank him for his questions.}

\subsubsection{Li--Wu via $1$-complexes} Consider a pair $(X|D)$ with $D$ smooth and connected. The cone complex $\Sigma_X$ of this space is just $\mathbb R_{\geq 0}$. A particular type of $1$-complex in $\Sigma_X$ can be obtained by choosing a finite collection of distinct points $q_1<q_2<\cdots<q_L$ in $\mathbb R_{\geq 0}$. Given this choice, we can subdivide $\mathbb R_{\geq 0}$ by adding these points as vertices, taking all edges between them, as well as the unbounded ray starting at $q_L$. This subdivision can be viewed as a $1$-complex embedded in $\mathbb R_{\geq 0}$. See Figure~\ref{fig: Li-Wu-picture}.

\begin{figure}

\tikzset{every picture/.style={line width=0.75pt}} 

\begin{tikzpicture}[x=0.75pt,y=0.75pt,yscale=-1,xscale=1]

\draw    (120,820) -- (308,820) ;
\draw [shift={(310,820)}, rotate = 180] [color={rgb, 255:red, 0; green, 0; blue, 0 }  ][line width=0.75]    (10.93,-3.29) .. controls (6.95,-1.4) and (3.31,-0.3) .. (0,0) .. controls (3.31,0.3) and (6.95,1.4) .. (10.93,3.29)   ;
\draw [shift={(120,820)}, rotate = 0] [color={rgb, 255:red, 0; green, 0; blue, 0 }  ][fill={rgb, 255:red, 0; green, 0; blue, 0 }  ][line width=0.75]      (0, 0) circle [x radius= 3.35, y radius= 3.35]   ;
\draw    (120,720) -- (308,720) ;
\draw [shift={(310,720)}, rotate = 180] [color={rgb, 255:red, 0; green, 0; blue, 0 }  ][line width=0.75]    (10.93,-3.29) .. controls (6.95,-1.4) and (3.31,-0.3) .. (0,0) .. controls (3.31,0.3) and (6.95,1.4) .. (10.93,3.29)   ;
\draw [shift={(120,720)}, rotate = 0] [color={rgb, 255:red, 0; green, 0; blue, 0 }  ][fill={rgb, 255:red, 0; green, 0; blue, 0 }  ][line width=0.75]      (0, 0) circle [x radius= 3.35, y radius= 3.35]   ;
\draw    (150,720) ;
\draw [shift={(150,720)}, rotate = 0] [color={rgb, 255:red, 0; green, 0; blue, 0 }  ][fill={rgb, 255:red, 0; green, 0; blue, 0 }  ][line width=0.75]      (0, 0) circle [x radius= 3.35, y radius= 3.35]   ;
\draw [shift={(150,720)}, rotate = 0] [color={rgb, 255:red, 0; green, 0; blue, 0 }  ][fill={rgb, 255:red, 0; green, 0; blue, 0 }  ][line width=0.75]      (0, 0) circle [x radius= 3.35, y radius= 3.35]   ;
\draw    (180,720) ;
\draw [shift={(180,720)}, rotate = 0] [color={rgb, 255:red, 0; green, 0; blue, 0 }  ][fill={rgb, 255:red, 0; green, 0; blue, 0 }  ][line width=0.75]      (0, 0) circle [x radius= 3.35, y radius= 3.35]   ;
\draw [shift={(180,720)}, rotate = 0] [color={rgb, 255:red, 0; green, 0; blue, 0 }  ][fill={rgb, 255:red, 0; green, 0; blue, 0 }  ][line width=0.75]      (0, 0) circle [x radius= 3.35, y radius= 3.35]   ;
\draw    (210,720) ;
\draw [shift={(210,720)}, rotate = 0] [color={rgb, 255:red, 0; green, 0; blue, 0 }  ][fill={rgb, 255:red, 0; green, 0; blue, 0 }  ][line width=0.75]      (0, 0) circle [x radius= 3.35, y radius= 3.35]   ;
\draw [shift={(210,720)}, rotate = 0] [color={rgb, 255:red, 0; green, 0; blue, 0 }  ][fill={rgb, 255:red, 0; green, 0; blue, 0 }  ][line width=0.75]      (0, 0) circle [x radius= 3.35, y radius= 3.35]   ;
\draw [color={rgb, 255:red, 128; green, 128; blue, 128 }  ,draw opacity=1 ]   (200,750) -- (200,797) ;
\draw [shift={(200,800)}, rotate = 270] [fill={rgb, 255:red, 128; green, 128; blue, 128 }  ,fill opacity=1 ][line width=0.08]  [draw opacity=0] (10.72,-5.15) -- (0,0) -- (10.72,5.15) -- (7.12,0) -- cycle    ;
\draw [shift={(200,750)}, rotate = 90] [color={rgb, 255:red, 128; green, 128; blue, 128 }  ,draw opacity=1 ][line width=0.75]      (0,-11.18) .. controls (-3.09,-11.18) and (-5.59,-8.68) .. (-5.59,-5.59) .. controls (-5.59,-2.5) and (-3.09,0) .. (0,0) ;

\draw (329,817) node [anchor=north west][inner sep=0.75pt]    {$\mathbb{R}_{\geq }{}_{0}$};
\draw (329,708) node [anchor=north west][inner sep=0.75pt]   [align=left] {\textsf{A typical 1-complex}};
\draw (139,689) node [anchor=north west][inner sep=0.75pt]    {$q_{1}$};
\draw (169,689) node [anchor=north west][inner sep=0.75pt]    {$q_{2}$};
\draw (199,689) node [anchor=north west][inner sep=0.75pt]    {$q_{3}$};

\end{tikzpicture}

\caption{A typical $1$-complex arising from the Li--Wu theory of expansions.}\label{fig: Li-Wu-picture}
\end{figure}

Note that there are other $1$-complexes, which come from disconnected graphs and the discussion that follows can be carried over to that setup, but we ignore these both for simplicity, and because they can be easily avoided in the smooth pair case. 

Denote the set of $1$-complexes above by $|T^c(\Sigma_X)|$. It is a subset of $|T(\Sigma_X)|$ discussed in Section~\ref{sec:tropical-moduli}. There is a canonical cone space structure on $|T^c(\Sigma_X)|$, where two points lie in the same cone if the associated $1$-complexes have the same number of vertices. This cone space is not of finite type: it has one cone of dimension $r$ for each natural number. In the corresponding universal $1$-complex, there are no tube vertices. 

Correspondingly, there is a canonical stack of expansions which we denote $\mathsf{Exp}^c(X|D)$, and it is equipped with a universal family 
\[
\mathcal Y\to \mathsf{Exp}^c(X|D).
\]
The stack and its universal family are precisely those found in~\cite{LiWu15}, as well as Li's earlier work in GW theory~\cite{Li01}. Using this stack, the constructions in Section~\ref{moduli-of-sheaves} can be adapted to reproduce the Li--Wu moduli space. 

\subsubsection{Artificially introducing tube complexes} To introduce tube phenomena, perform any sequence of blowup $\mathcal Y'\to \mathcal Y$ along strata. Examine the composition map
\[
\mathcal Y^\sim\to \mathsf{Exp}^c(X|D). 
\]
This map is likely not flat, but by performing universal toroidal semistable reduction~\cite{AK00,Mol16}, we obtain a map between blowups:
\[
\mathcal Y'\to \mathsf{Exp}^b(X|D)
\]
that is now flat with reduced fibers. 

To orient ourselves, let us choose a point $q$ in $\mathsf{Exp}^b(X|D)$. It maps to a point of $\mathsf{Exp}^c(X|D)$ that parameterizes expansions with $k$ expanded components. By pulling back to $q$ we now have two expansions over $q$:
\[
\pi: \mathcal Y'_q\to \mathcal Y\to q.
\]
The space $\mathcal Y'_q$ is still an expansion of $X$ along $D$, since any subdivision of an expansion is still an expansion. It may have, and can always be arranged to have, more than $k$ expanded components. Let us assume that $\mathcal Y'_q$ has $k+m$ expanded components. The map $\pi$ contracts these new components. Let us call these $m$ components the \textit{extra components}. 

The construction above can be performed using the formalism developed in Section~\ref{sec:tropical-moduli}. Precisely, since $\mathsf{Exp}^b(X|D)$ and $\mathcal Y'$ are both subdivisions, they can be obtained from appropriate cone structures on $|T^c(\Sigma_X)|$ and on its universal $1$-complex. The tube vertices, as defined earlier in the paper, give tube components in this context. These correspond exactly to the extra components. 

We can now pull back the DT moduli space $\mathsf{DT}^b_{\beta,n}(X|D)$ along the blowup map
\[
\mathsf{Exp}^c(X|D)\to \mathsf{Exp}^b(X|D).
\]
We can pullback the universal subscheme in $\mathcal Y$ to the blowup $\mathcal Y'$. Since in the usual Li--Wu space, a component of the expansion can never contain a tube subscheme, we can now can see now that the only subschemes that can be contained in a tube/extra components of a fiber of $\mathcal Y'$ are tube subschemes. 

The moral of the story is that in the smooth pair setting, all the possible outputs from different choices in our construction can be obtained by the procedure described above. Since we have the Li--Wu space already in this setting, the formalism of Section~\ref{sec:tropical-moduli} can be avoided -- the paper does not produce anything beyond the system of logarithmic modifications of the Li--Wu space. But in the general setting of simple normal crossings pairs, there does not seem to exist a minimal model for the stack $\mathsf{Exp}(X|D)$. As such, the system of models must be constructed independently of the minimal model. The situation is similar to the work of Molcho--Wise on the Picard group~\cite{MW18} and is explored in detail in forthcoming work of Kennedy-Hunt on general logarithmic Hilbert and Quot schemes~\cite{KH22}. 

%
%

\section{Transverse limits and generically expanded targets}\label{sec:expanded}

We verify the details of the valuative criterion for Theorem~\ref{thm: properness} that were postponed i.e. for families with a generic target expansion. An analysis via normalization requires a thorough treatment of the rubber geometry. Instead, we handle it with some additional formalism concerning logarithmic structures and tropicalizations in these circumstances; the overall structure is similar to the special case dealt with previously, so parallel reading may be advisable. 

\subsection{Preliminaries on valuation rings} Let $K$ be a complete discretely valued field with algebraically closed residue field and valuation ring $R$. Denote the spectra of $K$ and $R$ by $S^\circ$ and $S$ respectively, and the inclusion by $j:S^\circ\to S$. Equip the scheme $S^\circ$ with the standard logarithmic structure with monoid $\NN$. The tropicalization of $S^\circ$ is canonically identified with the real dual $\Hom(\NN,\RR_{\geq 0})$ of $\NN$, where the Hom is taken in the category of monoids. The tropicalization is therefore a ray and we denote it by $\rho$. 

In the situation typically considered, and in particular the one handled in Section~\ref{sec: strong-transversality}, one starts with a trivial logarithmic structure on the generic point of the spectrum of a DVR $\spec R$, and endows the full spectrum with the divisorial logarithmic structure. The characteristic monoid at the closed point can be understood as being the image of $R$ under the valuation. We are now in the case where there is an existing logarithmic structure on the generic point, which we can and will choose to have characteristic monoid $\mathbb N$, that must be extended to the closed point. We now explain how to do this. A geometric perspective is recorded in the remark below. 

Consider the rank $2$ lattice obtained as the direct sum of $\overline M_{S^\circ}^{\mathsf{gp}}$ and the group of exponents of a uniformizer for $R$; as presented, there is a natural identification with $\mathbb Z^2$, by sending the generator of $\overline M_{S^\circ}$ to $(1,0)$ and the positive generator of the value group to $(0,1)$. Let $N_S$ be the dual vector space of this group. There is a generization map on characteristic monoids from the closed point to the open point. Dualizing this map, the ray $\rho$ above, i.e. the dual cone of the characteristic monoid at $S^\circ$, embeds in $N_S$. Finally, there is a canonical quotient map $N_S\to \RR$, defined by taking the quotient of the span of $\rho$. The image of the lattice $N_S$ is canonically identified with the value group of $K$. 

We examine fine and saturated extensions of the logarithmic structure from $S^\circ$ to $S$; denote the inclusion by $j\colon S^\circ\hookrightarrow S$. A sufficient class is obtained as follows. Begin with the direct image logarithmic structure $j_\star M_{S^\circ}$.\footnote{The direct image logarithmic structure is typically not fine. It is possible to work with it nonetheless, but as we prefer to use standard toric machinery, we will avoid the direct use of this logarithmic structure.} The characteristic monoid of this structure at the closed point is isomorphic to the set of lexicographically nonnegative elements in $\ZZ^2$. The dual of the characteristic group at the closed point is the vector space $N_S$ defined above and the ray $\rho$ is equipped with an embedding in $N_S$. 

Collect the set of all $2$-dimensional rational polyhedral cones $\sigma\subset N_S$ containing $\rho$ as a face and whose image in the value group is the nonnegative elements. There is an ordering among elements in this set by inclusion. The dual monoids $P_\sigma$ determine fine and saturated submonoids of the characteristic monoid of direct image. We obtain, for each choice of cone $\sigma$, a fine and saturated logarithmic structure extending the one on $S^\circ$ whose characteristic monoid at the closed point is $P_\sigma$. Denote the extension by $S_\sigma$. 

\subsection*{Terminology} The logarithmic structures extending the given one on $S^\circ$ above will be called \textit{logarithmic extensions}. The set of such cones $\sigma$ form a filtered system under inclusion. In the discussion that follows will say that a statement holds for \textit{for sufficiently small $\sigma$} to indicate that it is true after replacing $\sigma$ with any element in a lower interval in the filtered system. We will correspondingly refer to the extensions as \textit{sufficiently small (fine and saturated) extensions}. 

\begin{remark}[Geometric perspective on the extensions]
Geometrically, the reader may visualize $S$ as the germ of a curve whose generic point lies in the interior of a boundary divisor of a toric variety and specializes to a codimension $2$ orbit, meeting this orbit with multiplicity $1$. One obtains a natural logarithmic structure on $S$ by pulling back the toric logarithmic structure. However, we can also blow up this toric variety, lift $S$ to this blowup by taking strict transform, and perform the same construction. Iterating this construction, we obtain a sequence of logarithmic structures on $S$. These are all fine and saturated, but not naturally isomorphic. Further blowups lead to smaller cones in the description above. The reader may wish to contrast this with the situation where $S$ specializes from the interior of a toric variety to the interior of a divisor, meeting the divisor with multiplicity $1$. In this case as well, pullback gives rise to a natural logarithmic structure But performing blowups to the ambient toric variety do not affect the pullback logarithmic structure on $S$. 
\end{remark}

For a finite extension $K'$ of $K$ equipped with the natural valuation, the corresponding morphism $\spec R'\to\spec R$ induces a map from the value group of $K'$ to the value group of $K$. It is ramified cover, \'etale away from from the closed points. We can equip $\spec R'$ with the a logarithmic structure by pulling back the logarithmic structure from $\spec R$ and then saturating the result. The map on value groups is an inclusion of a finite index sublattice, and in particular, if we denote by the spectrum by $S'$, then the associated vector spaces $N_S$ and $N_{S'}$ above are naturally identified. The ray $\rho$ is canonically identified the dual cone of the characteristic monoid at the generic point in $\spec R$, i.e. $\spec K$. Since the map $\spec K'\to \spec K$ is \'etale, the morphism above is strict with the above logarithmic structures. In particular, the characteristic monoids at the generic points are naturally identified. As a result, the integral structure on the ray $\rho$ induced from the lattice in $N_S$ and in $N_{S'}$ are the same. 

Before proceeding, we remind the reader that a map from the logarithmic scheme $S_\sigma$ to an Artin cone is equivalent to the data of a map from $\sigma$ to the corresponding cone~\cite[Section~6.2]{CCUW}.

\subsection{Expansions and tropicalizations} Let $\mathcal Y_\rho\to S^\circ$ be an expansion of $X$ and let $\Sigma_\rho\to \rho$ be the associated tropical family, noting that $\Sigma_\rho$ is embedded in $\Sigma_X\times\rho$. Let $Z_\rho$ be an algebraically transverse subscheme of $\mathcal Y_\rho$.\footnote{As before, we assume that these schemes are defined over a finitely generated subfield of $K$ to avoid foundational issues.}

\begin{lemma}
There exists a rough expansion $\mathcal Y_\sigma$ of $X$ over a sufficiently small extension $S_\sigma$ of $S^\circ$, extending the family $\mathcal Y_\rho$, such that the closure of $Z_\rho$ in $\mathcal Y_\sigma$ is dimensionally transverse. Moreover, the family $\mathcal Y_\sigma\to S_\sigma$ can be chosen to have relative logarithmic rank at most $1$, i.e. the rank of the relative characteristic monoid of the expansion over the base at most $1$. 
\end{lemma}

\begin{proof}
Let $\sigma_0$ be any smooth cone giving rise to a fine and saturated extension and consider $\Sigma_X\times \sigma_0$. The given family $\Sigma_\rho$ is embedded in $\Sigma_X\times \sigma_0$. Since $\Sigma_X\times \sigma_0$ is smooth it can be embedded in a vector space, and by the completion theorem for fans, we obtain a complete fan refining $\Sigma_X\times \sigma_0$ and extending $\Sigma_\rho$; this was referenced in Remark~\ref{rem: rel-comp}. For any sufficiently small replacement $\sigma$ of $\sigma_0$, every cone of this subdivision surjects onto a cone of $\sigma$, guaranteeing flatness. We obtain a complete subdivision $\Sigma^0_\sigma$ of $\Sigma_X\times\sigma$ and therefore a rough expansion of $X$ over $S_\sigma$. 

Consider the set of all subdivisions $\Sigma_\sigma$ of $\Sigma_\sigma^0$ that do not change $\Sigma_\rho$. We claim any sufficiently refined subdivision has the property that the closure $Z_\rho$ is dimensionally transverse. To see this, consider the closure of a component $Y$ of the generic fiber of $\mathcal Y_\rho$. By the tropical compactification results for subschemes of logarithmic schemes, there exists a sequence of blowups of strata in the closure of this component in $\mathcal Y_\sigma$, such that the closure of $Z_\rho\cap Y$ is dimensionally transverse~\cite[Theorems~1.1 {\it \&} 1.2]{U15}. The generic fiber subscheme has been supposed to be algebraically transverse and has dimension $1$, so any stratum that it intersects has codimension at most $1$ and dominates the base $S_\sigma$. Therefore the generic fiber is disjoint from the centers of these blowups. Repeat this for each component $Y$ of $\mathcal Y_\rho$ to obtain a new total space $\mathcal Y_\sigma$ for the target with cone complex $\Sigma_\sigma$ and the closure of the subscheme $Z_\rho$ is dimensionally transverse. After replacing $\sigma$ with a sufficiently small member in the filtered system and pulling back the target family, every cone of $\Sigma_\sigma$ maps surjectively onto a cone in $\sigma$, guaranteeing flatness. Pass to the union of strata that meet the closure of $Z_\rho$ nontrivially and obtain a rough expansion of relative logarithmic rank equal to $1$. 
\end{proof}

Given a subscheme $Z_\rho\hookrightarrow \mathcal Y_\rho$ as above, choose an extension $\mathcal Y_\sigma$ over $S_\sigma$ such that the closure of $\mathcal Z_\rho$ is dimensionally transverse. Let $\Sigma_\sigma$ be the corresponding embedded complex in $\Sigma_X\times\sigma$; we make the following definition.

\begin{definition}[Tropicalization with expanded target]
The tropicalization of $Z_\rho$ with respect to the family $\mathcal Y_\sigma$ over $S_\sigma$ above is the union of cones in $\Sigma_\sigma$ such that the closure of $Z_\rho$ has nonempty intersection with the corresponding logarithmic stratum of $\mathcal Y_\sigma$. The tropicalization is denoted $\trop(Z_\rho\hookrightarrow \mathcal Y_\sigma)$ viewed as a subset of $\Sigma_X\times \sigma$. 
\end{definition}

The tropicalization of a subscheme of the interior of a logarithmically smooth scheme is the image under a valuation map and is manifestly insensitive to subdivisions. 

\begin{lemma}[Birational invariance]
Let $\mathcal Y'_{\sigma'}$ and $\mathcal Y_{\sigma}$ be two rough expansions over $S_{\sigma'}$ and $S_\sigma$, both extending $\mathcal Y_\rho$. After restricting both families to a sufficient small cone $\sigma''$ the two tropicalizations
\[
\trop(Z_\rho\hookrightarrow \mathcal Y_{\sigma'}) \ \ \ \textrm{and} \ \ \ \trop(Z_\rho\hookrightarrow \mathcal Y_{\sigma})
\]
give the same subsets of $\Sigma_X\times\sigma''$. 
\end{lemma}

\begin{proof}
Consider two families as in the statement and let $\Sigma'_{\sigma'}$ and $\Sigma_{\sigma'}$ be the associated cone complexes. As we have done previously, we can complete both complexes to complete subdivisions and guarantee them to be flat after shrinking. By shrinking further we can assume that $\sigma$ and $\sigma'$ coincide, and examine them as cone complexes embedded in $\Sigma_X\times\sigma$. Let $\Sigma''_\sigma$ be the common refinement. After shrinking again, this common refinement is flat over $\sigma$. Since any subdivision of a cone complex is dominated by an iterated stellar subdivision, we reduce further to the case where $\Sigma_\sigma$ is replaced by a single stellar subdivision i.e. a weighted blowup of a stratum, necessarily in the special fiber. Since the closure $Z_\sigma$ of $Z_\rho$ is dimensionally transverse to the strata of $Y_\sigma$, the set of strata that the proper  and total transforms intersect nontrivially are the same. The strict transform is equal to the closure of $Z_\rho$ in a weighted blowup $\mathcal Y_{\sigma}'$ so the tropicalizations coincide. 
\end{proof}

\begin{lemma}[Algebraic transversality]\label{lem:expanded-strong-transversality}
After replacing $S_\sigma$ with a ramified base change $S_\sigma'$, there exists an expansion $\mathcal Y_\sigma$ of $X$ over a sufficiently small extension $S'_\sigma$ of $S_\rho$, extending the family $\mathcal Y_\rho$, such that the closure of $Z_\rho$ in $\mathcal Y_\sigma$ is algebraically transverse. 
\end{lemma}

\begin{proof}
After obtaining a rough expansion whose logarithmic structure has relative rank $1$ with the requisite dimensional transversality property, we perform the ramified base change to obtain an expansion of $X$ with reduced fibers~\cite[Section~5]{AK00}. We handle the passage from dimensional transversality to algebraic transversality similarly to Section~\ref{sec: strong-transversality} as we now explain. However, the reader is informed however that after establishing dimensional transversality as we have, the problem is formally identical to the one considered by Li--Wu, whose arguments can now be plugged in~\cite{LiWu15}. 

Let $\mathcal Y'_\sigma$ denote the expansion guaranteeing dimensional transversality. Any failure of algebraic transversality necessarily occurs at the double intersection of two components of the special fiber of $\mathcal Y'_\sigma$. If there is a Zariski open neighborhood of this double intersection inside the total space that is irreducible, then the Zariski local situation is identical to Section~\ref{sec: strong-transversality}. In this case, the equation at this double intersection is $xy = f$, where $f$ is an element of the valuation ring $R$. The tropicalization of this local neighborhood is therefore an edge of length equal to the valuation of $f$. We associate, as Section~\ref{sec: strong-transversality}, an initial degeneration to each rational point on this edge, and again all but finitely many of these points have tubular initial degeneration. We may blowup by including these non-tubular points to the polyhedral structure. After a base change to make these components reduced, algebraic transversality follows as in the case already treated. 

The remaining case is similar. We may pass to a Zariski open neighborhood $U$, in the total space of the degeneration, of the double intersection. The neighborhood is necessarily reducible and its normalization consists of components $U_1$ and $U_2$ glued along a smooth divisor $D$. Let $Z_1$ and $Z_2$ be the restriction of $Z_\rho$ to $U_1$ and $U_2$ respectively, and note that $Z_1$ and $Z_2$ are both necessarily algebraically transverse. Treating the $U_i$ separately, we may apply the results of Section~\ref{sec: strong-transversality}. Specifically, the tropicalization of each $U_i$ is a ray $\RR_{\geq 0}$ and every rational point on this ray gives rise to an initial degeneration, finitely many of which are non-tubular. These determine an expansion and a base change that give rise to a family of expansions of $U_i$. such that the closure of $Z_i$ is algebraically transverse. We obtain expansions of both $U_i$ along $D$ such that the closure of $Z_i$ is algebraically transverse. 

These expansions can be naturally glued; indeed, any expansion of $U_i$ along $D$ can be obtained by first performing a subdivision of the family $\mathcal Y_\sigma$ and then pulling back along the inclusion of $U_i$. Since algebraic transversality can be checked after normalizing, it follows that there exists a family $\mathcal Y_\sigma$ over $S'_\sigma$ of expansions such that the closure of the generic fiber is algebraically transverse. Finally, blow down any components in the special fiber whose subschemes are tubular using Lemma~\ref{lem: contraction-lemma}. We obtain an algebraically transverse expansion without tubular components. 
\end{proof}

\subsection{Verification of the valuative criterion} We fix a conical structure on the moduli space $T$ of tropical expansions and on its universal family. This determines a moduli space $\mathsf{DT}(X|D)$ of relative ideal sheaves equipped with a structure map to the stack $\mathsf{Exp}(X|D)$ of target expansions. We check that for each map
\[
\spec K \to \mathsf{DT}(X|D)
\]
from the spectrum of a discretely valued field, there exists a canonical base change $K\subset K'$ of a well-defined smallest order, and an extension of the map to the spectrum of the corresponding valuation ring. 

We have shown existence in the case where $\spec K$ maps to the locus in $\mathsf{Exp}(X|D)$ where the logarithmic structure is trivial and therefore assume that it maps to the boundary of $\mathsf{Exp}(X|D)$. By replacing $T$ with a subdivision i.e. by performing a weighted blowing up of stratum to which it maps, we may choose a lift of the morphism from $\spec K$ so that it maps to the interior of a logarithmic divisor in $\mathsf{Exp}(X|D)$. Note that the blowup is proper, so it is sufficient to check the valuative criterion in the case where the generic point maps to a divisor. Indeed, this follows from applying the valuative criterion to the blowup map itself: if we find a completion of the map to $\spec R$ on the blowup, we can compose with the blow down to obtain an extension.

We have fixed $S^\circ\subset S$. Pull back the logarithmic structure, universal expansion, and subscheme to form a family
\[
\begin{tikzcd}
\mathcal Z_\rho\arrow{rr}\arrow{dr} & & \mathcal Y_\rho\arrow{dl}\\
&S^\circ.&
\end{tikzcd}
\]
In keeping with the notation of the preceding discussion, $\rho$ denotes the ray corresponding to the exceptional divisor that contains the image of $\spec K$. We fix these data for the following. 

\begin{proposition}
There exists a sufficiently small extension $S_\sigma$ of $S^\circ$ and its logarithmic structure, a ramified base change $S_\sigma'\to S_\sigma$, and an expansion $\mathcal Y_{\sigma'}$ over $S_{\sigma'}$ such that, after pullback, the closure of $Z_\rho$ is algebraically transverse and satisfies the DT stability condition. 
\end{proposition}

\begin{proof}
Apply Lemma~\ref{lem:expanded-strong-transversality} to obtain a finite extension of $K$ and a sufficiently small extension $S'_\sigma$ and an expansion $\mathcal Y_\sigma\to S'_\sigma$ such that the closure of $Z_\rho$ is algebraically transverse. The tropicalization gives rise to a family of $1$-complexes $\Sigma_\sigma\to\sigma$. Choose any point $P$ in the relative interior of $\Sigma$ and let $\mathbb G_P$ be the associated embedded $1$-complex. Each vertex $v$ of the $1$-complex $\mathbb G_P$ corresponds to a component of the special fiber of the expansion family. Let $Z_v\hookrightarrow Y_v$ be the associated subscheme. A $2$-valent vertex that lies on a line in $\Sigma_X$ determines a component of the expansion that is a $\PP^1$-bundle over a codimension $2$ stratum in $X$; we examine whether the associated subscheme $Z_v$ is a tube, i.e. it is pulled back from a $0$-dimensional subscheme along the bundle projection. Let $\underline {\mathbb G}_P$ be the $1$-complex obtained by erasing all such tube vertices. Formally, if $v$ is such a vertex that is incident to two edges, delete both $v$ and these edges and connect the two vertices that were incident to $v$ by a single edge. The result is an embedded $1$-complex. This gives rise to a map from the ray generated by $P$ to the complex $|T|$. After shrinking $\sigma$ further, for example until $P$ is extremal, this defines a map
\[
\sigma\to T.
\]
The union of the $1$-complexes $\underline {\mathbb G}_P$ do not necessarily form a cone complex, though their support is certainly the same as a that of a cone complex. 

By pulling back the universal tropical expansion over $T$ to $\sigma$ we obtain a subdivision $\Upsilon_\sigma$ of $\Sigma_X\times\sigma$ whose support equal to that of $\Sigma_\sigma$. Pass to a common refinement $\Delta_\sigma$, which again, after shrinking $\sigma$ is flat over $\sigma$ and therefore gives rise to a rough expansion over $S_\sigma'$. After performing another ramified base change, we obtain a new expansion and a morphism of expansions associated to the map $\Delta_\sigma\to \Upsilon_\sigma$
\[
\mathcal Y_{\sigma}^\Delta\to \mathcal Y_{\sigma}^\Upsilon, \ \ \textrm{over } S''_\sigma. 
\]
As algebraic transversality is stable under additional subdivision, the closure of $\mathcal Z_\rho$ inside $\mathcal Y_{\sigma}^\Delta$ is algebraically transverse. By construction, any component of the special fiber that is contracted by this map carries a tube subscheme. The resulting family is therefore also algebraically transverse, and satisfies the DT stability condition. We conclude existence of limits. 
\end{proof}

The order of the base change required for an extension can be controlled. We produced a map $\sigma\to T$ and momentarily view it as a map only at the level of the $\QQ$-structures. The cone $\sigma$ can be equipped with an integral structure coming from the value group of $R$ or any of finite extensions. When equipped with the integral structure coming from $R$ the map to $T$ above may not be given by an integer linear map. The restriction to $\rho$ is an inclusion. There is a minimal base change in these circumstances, obtained by passage to a finite index sublattice of the standard lattice in $N_S$, such that the map can be extended, giving rise to a family of stable ideal sheaves. Algebraic transversality and DT stability are unaffected by further base change, we can conclude \textit{a posteriori} that the family produced by this minimal extension is transverse and stable. 

Call the minimal extension constructed above as the \textit{distinguished} extension. We conclude the proof of properness by verifying the distinguished extension is the only one.

\begin{proposition}
Given a morphism $\spec K\to \mathsf{DT}(X|D)$ and an extension $\spec R''\to \mathsf{DT}(X|D)$ from a ramified base change, it is obtained by pulling back the distinguished extension above.
\end{proposition}

\begin{proof}
Given the morphism from $S^\circ\to \mathsf{DT}(X|D)$ we have obtained the distinguished extension to a logarithmic morphism $S_\sigma'\to \mathsf{DT}(X|D)$ as in the previous proposition. Let $R\subset R''$ be a finite base change and consider an extension
\[
\spec R'' \to \mathsf{DT}(X|D).
\]
By the universal property of the direct image, this lifts to a logarithmic map from the direct image logarithmic structure on $\spec R''$ to $\mathsf{DT}(X|D)$. The target is fine and saturated so for any sufficiently small fine and saturated extension, the schematic map is enhanced to a logarithmic map $S''_\tau\to \mathsf{DT}(X|D)$. This can be seen by direct analysis of the monoids, see also~\cite[Section~2.2.5]{MW18}. 

We view $\tau$ and $\sigma$ as two cones in the same vector space $N_S$, but that span possibly different rank $2$ lattices. By shrinking, we may and do assume that $\tau$ and $\sigma$ coincide set theoretically on $N_S$ i.e. their real points coincide, and we obtain two families of embedded $1$-complexes in $\Sigma_X\times\sigma$. Denote these by $\Upsilon$ and $\Upsilon'$ respectively, and note that they are different polyhedral structures on the same set in $\Sigma_X\times\sigma$. We may therefore pass to the common refinement of these two complexes and choose a sublattice of $\sigma$ such that the associated map has reduced fibers. 

Let $P$ be a point in the relative interior of $\sigma$ and let $\mathbb G_P$ be the fiber of this common refinement over $P$. As in the previous proposition, we may examine the closure of $Z_\rho$ for vertices in $\mathbb G_P$ that carry tube subschemes, and erasing the corresponding tube vertices. This gives rise to a new graph and map from $\sigma$ to $|T|$ as discussed in the previous proposition. Now note that the outcome of erasing the vertices with tube subschemes for $\mathbb G_P$ is the same as erasing such vertices for the fibers over $p$ in $\Upsilon$ and $\Upsilon'$. The map from the base cone $\sigma$ to $T$ that is induced by the map above therefore coincides in both cases, at the level of the real points. 

For the base change, observe that as discussed immediately before the proposition statement, there is a minimum order of base change required to extend the given map from $\spec K$. It follows that up to a further ramified base change, the two extensions coincide, establishing separatedness. 
\end{proof}

\bibliographystyle{siam} 
\bibliography{LogarithmicDTv6}

\begin{thebibliography}{10}

\bibitem{ACP}
{\sc D.~Abramovich, L.~Caporaso, and S.~Payne}, {\em The tropicalization of the
  moduli space of curves}, Ann. Sci. {\'E}c. Norm. Sup{\'e}r., 48 (2015),
  pp.~765--809.

\bibitem{AC11}
{\sc D.~Abramovich and Q.~Chen}, {\em Stable logarithmic maps to
  {D}eligne-{F}altings pairs {II}}, Asian J. Math., 18 (2014), pp.~465--488.

\bibitem{ACGS15}
{\sc D.~Abramovich, Q.~Chen, M.~Gross, and B.~Siebert}, {\em Decomposition of
  degenerate {G}romov-{W}itten invariants}, Compos. Math., 156 (2020),
  pp.~2020--2075.

\bibitem{ACGS20}
\leavevmode\vrule height 2pt depth -1.6pt width 23pt, {\em Punctured
  logarithmic maps}, arXiv:2009.07720,  (2020).

\bibitem{ACMUW}
{\sc D.~Abramovich, Q.~Chen, S.~Marcus, M.~Ulirsch, and J.~Wise}, {\em
  Skeletons and fans of logarithmic structures}, in Nonarchimedean and Tropical
  Geometry, M.~Baker and S.~Payne, eds., Simons Symposia, Springer, 2016,
  pp.~287--336.

\bibitem{ACMW}
{\sc D.~Abramovich, Q.~Chen, S.~Marcus, and J.~Wise}, {\em Boundedness of the
  space of stable logarithmic maps}, {J. Eur. Math. Soc.}, 19 (2017),
  pp.~2783--2809.

\bibitem{AK00}
{\sc D.~Abramovich and K.~Karu}, {\em Weak semistable reduction in
  characteristic 0}, Invent. Math., 139 (2000), pp.~241--273.

\bibitem{AW97}
{\sc D.~Abramovich and J.~Wang}, {\em Equivariant resolution of singularities
  in characteristic {$0$}}, Math. Res. Lett., 4 (1997), pp.~427--433.

\bibitem{AW}
{\sc D.~Abramovich and J.~Wise}, {\em {Birational invariance in logarithmic
  Gromov--Witten theory}}, Comp. Math., 154 (2018), pp.~595--620.

\bibitem{AM14}
{\sc K.~B. Ascher and S.~Molcho}, {\em Logarithmic stable toric varieties and
  their moduli}, Algebraic Geometry, 3 (2016), pp.~296--319.

\bibitem{BN19}
{\sc L.~Battistella and N.~Nabijou}, {\em Relative quasimaps and mirror
  formulae}, Int. Math. Res. Not. IMRN,  (2021), pp.~7885--7931.

\bibitem{BF08}
{\sc K.~Behrend and B.~Fantechi}, {\em Symmetric obstruction theories and
  {Hilbert} schemes of points on threefolds}, Algebra \& Number Theory, 2
  (2008), pp.~313--345.

\bibitem{BBLdM}
{\sc B.~Bertrand, E.~Brugall\'{e}, and L.~L\'{o}pez~de Medrano}, {\em Planar
  tropical cubic curves of any genus, and higher dimensional generalisations},
  Enseign. Math., 64 (2018), pp.~415--457.

\bibitem{BG84}
{\sc R.~Bieri and J.~R. Groves}, {\em The geometry of the set of characters
  iduced by valuations.}, J. Reine Angew. Math. (Crelle's Journal), 347 (1984),
  pp.~168--195.

\bibitem{BU18}
{\sc M.~Brandt and M.~Ulirsch}, {\em Symmetric powers of algebraic and tropical
  curves: a non-{A}rchimedean perspective}, Trans. Amer. Math. Soc. Ser. B, 9
  (2022), pp.~586--618.

\bibitem{CCUW}
{\sc R.~Cavalieri, M.~Chan, M.~Ulirsch, and J.~Wise}, {\em A moduli stack of
  tropical curves}, {Forum Math. Sigma}, 8 (2020), pp.~1--93.

\bibitem{Che10}
{\sc Q.~Chen}, {\em Stable logarithmic maps to {D}eligne-{F}altings pairs {I}},
  Ann. of Math., 180 (2014), pp.~341--392.

\bibitem{CS12}
{\sc Q.~Chen and M.~Satriano}, {\em Chow quotients of toric varieties as moduli
  of stable log maps}, Algebra and Number Theory Journal, 7 (2013),
  pp.~2313--2329.

\bibitem{Cos06}
{\sc K.~Costello}, {\em {Higher genus Gromov-Witten invariants as genus zero
  invariants of symmetric products}}, Ann. of Math.,  (2006), pp.~561--601.

\bibitem{CLS11}
{\sc D.~A. Cox, J.~B. Little, and H.~K. Schenck}, {\em Toric varieties},
  vol.~124 of Graduate Studies in Mathematics, American Mathematical Society,
  Providence, RI, 2011.

\bibitem{DT}
{\sc S.~K. Donaldson and R.~P. Thomas}, {\em Gauge theory in higher
  dimensions}, in The geometric universe ({O}xford, 1996), Oxford Univ. Press,
  Oxford, 1998, pp.~31--47.

\bibitem{Verdier}
{\sc A.~Douady and J.-L. Verdier}, {\em {S{\'e}minaire de g{\'e}ometrie
  analytique de l'{\'E}cole normale sup{\'e}rieure} 1974/1975}, Asterisque,
  exposes VI--IX (1976).

\bibitem{EI06}
{\sc G.~Ewald and M.-N. Ishida}, {\em {Completion of real fans and
  Zariski-Riemann spaces}}, Tohoku Math. J., 58 (2006), pp.~189--218.

\bibitem{Ful93}
{\sc W.~Fulton}, {\em Introduction to Toric Varieties}, Princeton University
  Press, 1993.

\bibitem{GKZtriangulations}
{\sc I.~M. Gelfand, M.~M. Kapranov, and A.~V. Zelevinsky}, {\em {Triangulations
  and Secondary Polytopes}}, in Discriminants, Resultants, and Multidimensional
  Determinants, Springer, 1994, pp.~214--251.

\bibitem{BGS11}
{\sc J.~I.~B. Gil and M.~Sombra}, {\em When do the recession cones of a
  polyhedral complex form a fan?}, Discrete \& Computational Geometry, 46
  (2011), pp.~789--798.

\bibitem{GS13}
{\sc M.~Gross and B.~Siebert}, {\em {Logarithmic Gromov-Witten invariants}}, J.
  Amer. Math. Soc., 26 (2013), pp.~451--510.

\bibitem{Gub13}
{\sc W.~Gubler}, {\em A guide to tropicalizations}, in Algebraic and
  combinatorial aspects of tropical geometry, vol.~589 of Contemp. Math., Amer.
  Math. Soc., Providence, RI, 2013, pp.~125--189.

\bibitem{HKT}
{\sc P.~Hacking, S.~Keel, and J.~Tevelev}, {\em {Stable pair, tropical, and log
  canonical compactifications of moduli spaces of del Pezzo surfaces}}, Invent.
  Math., 178 (2009), pp.~173--227.

\bibitem{HW23}
{\sc L.~Herr and J.~Wise}, {\em Costello's pushforward formula: errata and
  generalization}, manuscripta mathematica, 171 (2023), pp.~621--642.

\bibitem{SheafBook}
{\sc D.~Huybrechts and M.~Lehn}, {\em The geometry of moduli spaces of
  sheaves}, Cambridge University Press, 2010.

\bibitem{HT10}
{\sc D.~Huybrechts and R.~P. Thomas}, {\em {Deformation-obstruction theory for
  complexes via Atiyah and Kodaira--Spencer classes}}, {Math. Ann.}, 346
  (2010), pp.~545--569.

\bibitem{IP03}
{\sc E.-N. Ionel and T.~H. Parker}, {\em {Relative Gromov-Witten invariants}},
  Ann. of Math.,  (2003), pp.~45--96.

\bibitem{Kap93}
{\sc M.~Kapranov}, {\em {Chow quotients of Grassmannians. {I}}}, in {IM Gelfand
  Seminar}, vol.~16, 1993, pp.~29--110.

\bibitem{Kat89}
{\sc K.~Kato}, {\em {Logarithmic structures of Fontaine-Illusie}}, Algebraic
  analysis, geometry, and number theory (Baltimore, MD, 1988),  (1989),
  pp.~191--224.

\bibitem{Kat09}
{\sc E.~Katz}, {\em Tropical invariants from the secondary fan}, Adv. Geom., 9
  (2009), pp.~153--180.

\bibitem{KP11}
{\sc E.~Katz and S.~Payne}, {\em Realization spaces for tropical fans}, in
  Combinatorial aspects of commutative algebra and algebraic geometry,
  Springer, 2011, pp.~73--88.

\bibitem{KT06}
{\sc S.~Keel and J.~Tevelev}, {\em Geometry of {C}how quotients of
  {G}rassmannians}, Duke Math. J., 134 (2006), pp.~259--311.

\bibitem{KKMSD}
{\sc G.~Kempf, F.~Knudsen, D.~Mumford, and B.~Saint-Donat}, {\em Toroidal
  embeddings {I}}, Lecture Notes in Mathematics, 339 (1973).

\bibitem{KH21}
{\sc P.~Kennedy-Hunt}, {\em Logarithmic {Pandharipande--Thomas} spaces and the
  secondary polytope}, arXiv:2112.00809,  (2021).

\bibitem{KH22}
\leavevmode\vrule height 2pt depth -1.6pt width 23pt, {\em The logarithmic quot
  space: foundations and tropicalisation}, arXiv:2308.14470,  (2023).

\bibitem{LP09}
{\sc M.~Levine and R.~Pandharipande}, {\em Algebraic cobordism revisited},
  Invent. Math., 176 (2009), pp.~63--130.

\bibitem{LR01}
{\sc A.-M. Li and Y.~Ruan}, {\em {Symplectic surgery and Gromov-Witten
  invariants of Calabi-Yau 3-folds}}, Invent. Math., 145 (2001), pp.~151--218.

\bibitem{Li01}
{\sc J.~Li}, {\em Stable morphisms to singular schemes and relative stable
  morphisms}, J. Diff. Geom., 57 (2001), pp.~509--578.

\bibitem{Li02}
\leavevmode\vrule height 2pt depth -1.6pt width 23pt, {\em {A degeneration
  formula of GW-invariants}}, J. Diff. Geom., 60 (2002), pp.~199--293.

\bibitem{Li06}
\leavevmode\vrule height 2pt depth -1.6pt width 23pt, {\em Zero dimensional
  donaldson--thomas invariants of threefolds}, Geometry \& Topology, 10 (2006),
  pp.~2117--2171.

\bibitem{LiWu15}
{\sc J.~Li and B.~Wu}, {\em {Good degeneration of Quot-schemes and coherent
  systems}}, Communications in Analysis and Geometry, 23 (2015), pp.~841--921.

\bibitem{MS14}
{\sc D.~Maclagan and B.~Sturmfels}, {\em Introduction to {T}ropical
  {G}eometry}, vol.~161 of Graduate Studies in Mathematics, American
  Mathematical Society, Providence, RI, 2015.

\bibitem{Mano12}
{\sc C.~Manolache}, {\em Virtual pull-backs}, J. Algebr. Geom., 21 (2012),
  pp.~201--245.

\bibitem{MW17}
{\sc S.~Marcus and J.~Wise}, {\em {Logarithmic compactification of the
  Abel-Jacobi section}}, arXiv:1708.04471,  (2017).

\bibitem{MNOP06a}
{\sc D.~Maulik, N.~Nekrasov, A.~Okounkov, and R.~Pandharipande}, {\em
  {Gromov--Witten theory and Donaldson--Thomas theory, I}}, Compos. Math., 142
  (2006), pp.~1263--1285.

\bibitem{MNOP06b}
\leavevmode\vrule height 2pt depth -1.6pt width 23pt, {\em {Gromov--Witten
  theory and Donaldson--Thomas theory, II}}, Compos. Math., 142 (2006),
  pp.~1286--1304.

\bibitem{MOOP}
{\sc D.~Maulik, A.~Oblomkov, A.~Okounkov, and R.~Pandharipande}, {\em
  {Gromov-Witten/Donaldson-Thomas correspondence for toric 3-folds}}, Invent.
  Math., 186 (2011), pp.~435--479.

\bibitem{maulik-okounkov}
{\sc D.~Maulik and A.~Okounkov}, {\em Quantum groups and quantum cohomology},
  Ast\'{e}risque,  (2019), pp.~ix+209.

\bibitem{MPT10}
{\sc D.~Maulik, R.~Pandharipande, and R.~P. Thomas}, {\em Curves on {K3}
  surfaces and modular forms}, J. Topol., 3 (2010), pp.~937--996.

\bibitem{MR22}
{\sc D.~Maulik and D.~Ranganathan}, {\em Logarithmic enumerative geometry for
  curves and sheaves}, arXiv:2311.14150,  (2023).

\bibitem{Mi03}
{\sc G.~Mikhalkin}, {\em Enumerative tropical geometry in {${\mathbb{R}^2}$}},
  J. Amer. Math. Soc, 18 (2005), pp.~313--377.

\bibitem{Mol16}
{\sc S.~Molcho}, {\em Universal stacky semistable reduction}, Israel J. Math.,
  242 (2021), pp.~55--82.

\bibitem{MW18}
{\sc S.~Molcho and J.~Wise}, {\em The logarithmic {P}icard group and its
  tropicalization}, Compos. Math., 158 (2022), pp.~1477--1562.

\bibitem{Mum72}
{\sc D.~Mumford}, {\em An analytic construction of degenerating abelian
  varieties over complete rings}, Compos. Math., 24 (1972), pp.~239--272.

\bibitem{NS06}
{\sc T.~Nishinou and B.~Siebert}, {\em Toric degenerations of toric varieties
  and tropical curves}, Duke Math. J., 135 (2006), pp.~1--51.

\bibitem{Ols03}
{\sc M.~C. Olsson}, {\em Logarithmic geometry and algebraic stacks}, Ann. Sci.
  {\'E}c. Norm. Sup{\'e}r., 36 (2003), pp.~747--791.

\bibitem{PP12}
{\sc R.~Pandharipande and A.~Pixton}, {\em {Gromov-Witten/Pairs correspondence
  for the quintic 3-fold}}, J. Amer. Math. Soc., 30 (2017), pp.~389--449.

\bibitem{PT09}
{\sc R.~Pandharipande and R.~P. Thomas}, {\em Curve counting via stable pairs
  in the derived category}, Invent. Math., 178 (2009), pp.~407--447.

\bibitem{PT10}
\leavevmode\vrule height 2pt depth -1.6pt width 23pt, {\em {Stable pairs and
  BPS invariants}}, J. Amer. Math. Soc, 23 (2010), pp.~267--297.

\bibitem{Pay07}
{\sc S.~Payne}, {\em Fibers of tropicalization}, Math. Z., 262 (2009),
  pp.~301--311.

\bibitem{R15b}
{\sc D.~Ranganathan}, {\em {Skeletons of stable maps I: rational curves in
  toric varieties}}, J. Lond. Math. Soc., 95 (2017), pp.~804--832.

\bibitem{R19}
\leavevmode\vrule height 2pt depth -1.6pt width 23pt, {\em Logarithmic
  {G}romov-{W}itten theory with expansions}, Algebr. Geom., 9 (2022),
  pp.~714--761.

\bibitem{RW19}
{\sc D.~Ranganathan and J.~Wise}, {\em Rational curves in the logarithmic
  multiplicative group}, Proc. Amer. Math. Soc., 148 (2020), pp.~103--110.

\bibitem{RG71}
{\sc M.~Raynaud and L.~Gruson}, {\em {Criteres de platitude et de
  projectivit\'e: Techniques de platification d'un module}}, Invent. Math., 13
  (1971), pp.~1--89.

\bibitem{Tev07}
{\sc J.~Tevelev}, {\em Compactifications of subvarieties of tori}, Amer. J.
  Math., 129 (2007), pp.~1087--1104.

\bibitem{stacks-project}
{\sc {The Stacks Project Authors}}, {\em \itshape stacks project}.
\newblock \url{http://stacks.math.columbia.edu}, 2019.

\bibitem{RPWT}
{\sc R.~P. Thomas}, {\em A holomorphic {C}asson invariant for {C}alabi-{Y}au
  3-folds, and bundles on {$K3$} fibrations}, J. Differential Geom., 54 (2000),
  pp.~367--438.

\bibitem{U15}
{\sc M.~Ulirsch}, {\em Tropical compactification in log-regular varieties},
  Math. Z., 280 (2015), pp.~195--210.

\bibitem{U13}
\leavevmode\vrule height 2pt depth -1.6pt width 23pt, {\em Functorial
  tropicalization of logarithmic schemes: the case of constant coefficients},
  Proc. Lond. Math. Soc., 114 (2017), pp.~1081--1113.

\bibitem{Yu14}
{\sc T.~Y. Yu}, {\em The number of vertices of a tropical curve is bounded by
  its area}, L'Enseignement Math{\'e}matique, 60 (2015), pp.~257--271.

\end{thebibliography}

\end{document}